\pdfoutput=1
\documentclass[11pt,a4paper,reqno]{amsart}
\usepackage[british]{babel}

\usepackage[DIV=9,oneside,BCOR=0mm]{typearea}
\usepackage{microtype}
\usepackage[T1]{fontenc}
\usepackage{setspace}
\usepackage{amssymb}
\usepackage{amsmath}
\usepackage{amsthm}
\usepackage{enumitem}
\usepackage{mathrsfs}
\usepackage[colorlinks,citecolor=blue,urlcolor=black,linkcolor=red,linktocpage]{hyperref}
\usepackage{mathrsfs}
\usepackage{mathtools}
\usepackage{xcolor}

\newtheorem{thm}{Theorem}[section]

\newtheorem{cor}[thm]{Corollary}
\newtheorem{lem}[thm]{Lemma}
\newtheorem{prop}[thm]{Proposition}

\newtheorem{claim}[thm]{Claim}

\theoremstyle{definition}

\newtheorem{exmpl}[thm]{Example}

\newtheorem{definition}[thm]{Definition}

\newtheorem{remark}[thm]{Remark}
\newtheorem{remarks}[thm]{Remarks}

\renewcommand{\epsilon}{\varepsilon}
\newcommand{\defeq}{\mathrel{\mathop:}=}

\DeclareMathOperator{\spt}{supp}

\DeclareMathOperator{\Aut}{Aut}

\makeatletter
\def\moverlay{\mathpalette\mov@rlay}
\def\mov@rlay#1#2{\leavevmode\vtop{%
		\baselineskip\z@skip \lineskiplimit-\maxdimen
		\ialign{\hfil$\m@th#1##$\hfil\cr#2\crcr}}}
\newcommand{\charfusion}[3][\mathord]{
	#1{\ifx#1\mathop\vphantom{#2}\fi
		\mathpalette\mov@rlay{#2\cr#3}
	}
	\ifx#1\mathop\expandafter\displaylimits\fi}
\makeatother

\begin{document}
%%%%%%%%%%%%%%%%%%%%%%%%%%%%%%%%

\setlist{noitemsep}

\author{Friedrich Martin Schneider}
\address{F.M.~Schneider, Institute of Discrete Mathematics and Algebra, TU Bergakademie Freiberg, 09596 Freiberg, Germany}
\email{martin.schneider@math.tu-freiberg.de}
\author{S{\l}awomir Solecki}
\address{S.~Solecki, Department of Mathematics, Cornell University, Ithaca, NY 14853,~USA}
\email{ssolecki@cornell.edu}
\thanks{The first-named author acknowledges funding of the Excellence Initiative by the German Federal and State Governments. The second-named author acknowledges funding of NSF grant DMS-1800680.}

\title[Concentration, classification, and dynamics]{Concentration of measure, classification of submeasures, and dynamics of~$L_{0}$}
\date{\today}


\begin{abstract} 
 Exhibiting a new type of measure concentration, we prove uniform concentration bounds for measurable Lipschitz 
 functions on product spaces, where Lipschitz is taken with respect to the metric induced by a weighted covering of the index set of 
 the product. Our proof combines the Herbst argument with an analogue of Shearer's lemma for differential entropy. We give a quantitative ``geometric'' classification of diffuse submeasures into elliptic, parabolic, and hyperbolic. We prove that any non-elliptic submeasure (for example, any measure, or any pathological submeasure) has a property that we call \emph{covering concentration}. Our results have strong consequences for the dynamics of the corresponding topological $L_{0}$-groups.
\end{abstract}


\maketitle

%%%%%%%%%%%%%%%%%%%%%%%%%%%%%%%%%%%%%%%%%%
%%%%%%%%%%%%%%%%%%%%%%%%%%%%%%%%%%%%%%%%%%

\tableofcontents

\section{Introduction}

The present paper makes contributions to three areas: the probabilistic theme of 
concentration of measure in product spaces; the set theoretic and measure theoretic theme of submeasures; 
and the topological dynamical theme of extreme amenability.

{\em Concentration of measure in products.} 
We introduce a generalization of the Hamming metric on product spaces and prove concentration of measure for it. 
(The book~\cite{ledoux} is a rich source of information 
on concentration of measure.) 
Generalizations of the Hamming metric in the context of concentration of measure 
were considered by Talagrand~\cite{talagrand95,talagrand96}. 
Our approach appears to be orthogonal to Talagrand's. We start with a sequence of sets 
${\mathcal C} = (C_0, \dots, C_{m-1})$ covering a non-empty set $N$ together 
with a sequence of positive real numbers, weights, $w=(w_0, \dots , w_{m-1})$. 
The sequences $\mathcal C$ and $w$ will be the parameters determining the metric.
Given a family of sets $\Omega_j$, $j\in N$, we define a metric~$d_{{\mathcal C}, w}$ 
on $\prod_{j\in N} \Omega_j$ as follows: for two points
$x=(x_0, \dots, x_{m-1})$ and $y= (y_0, \dots , y_{m-1})$ in the product, let 
\[
d_{{\mathcal C}, w}(x,y) \, \defeq \, \inf\nolimits_I \sum\nolimits_{i\in I} w_i,
\]
where $I$ runs over all $I\subseteq \{ 0, \dots, m-1\}$ with 
\[
\{ j\in N\mid x_j \not= y_j\} \, \subseteq \, \bigcup\nolimits_{i\in I} C_i. 
\]
Note that if the sets $C_i$, $i<m$, form a partition of $N$ into one-element sets (so $m=|N|$) and $w_i= 1/|N|$ for each $i<m$, then 
$d_{{\mathcal C}, w}$ coincides with the normalized Hamming metric. 

We prove a concentration of measure theorem in product spaces for the above metric $d_{{\mathcal C}, w}$. Our interest in such a concentration of measure theorem comes from applications in topological dynamics in proving extreme amenability of certain Polish groups. 
To state the concentration of measure theorem, we extract a natural number $k$ from the sequence $\mathcal C$; we call 
$\mathcal C$ a {\em $k$-cover of} $N$ if each element of $N$ belongs to at least $k$ entries of the sequence $\mathcal C$. 
We consider now a family of standard Borel probability spaces indexed by the set $N$: $(\Omega_{j},\mu_{j})_{j \in N}$. 
Let $\mathbb P$ be the product measure on $\prod_{j\in N} \Omega_j$. 
Assuming that $\mathcal C$ is a $k$-cover of $N$, we prove in 
Theorem~\ref{theorem:covering.concentration} that for each measurable function 
$f \colon \prod_{j\in N} \Omega_j \to \mathbb{R}$ that is $1$-Lipschitz with respect to 
$d_{\mathcal{C},w}$ and for every~$r \in \mathbb{R}_{> 0}$, 
\[
{\mathbb P} ( \{ x \mid f(x) - \mathbb{E}_{\mathbb{P}}(f) \geq r \}) \, 
\leq \, \exp \!\left( -\tfrac{kr^{2}}{4\sum_{i<m} w_{i}^{2}} \right). 
\]
The advancement consists of the presence of $k$ in the exponent on the right-hand side of the above inequality. 
Our proof of concentration of measure uses the entropy method developed by Ledoux~\cite{ledoux95,ledoux96,ledoux99} building on the so-called \emph{Herbst argument}, which originates in an unpublished letter by Herbst to Gross. 
The second main ingredient of our proof is a result by Madiman--Tetali~\cite[Corollary~VIII]{MadimanTetali} (Lemma~\ref{lemma:entropic.loomis.whitney} in the present paper), which relates differential entropy on product spaces with covering numbers of covers of the underlying index sets and in turn constitutes an analogue of Shearer's lemma for Shannon entropy of discrete random variables~\cite[Section~V, page~33, item~(22)]{shearer}.
For a broader background on concentration of measure, the reader may consult~\cite{ledoux}.

{\em Submeasures as pseudo-metrics.} 
A real-valued function $\phi$ on a Boolean algebra $\mathcal A$ is a {\em submeasure} if it is subadditive, 
monotone with respect to the natural ordering of $\mathcal{A}$, and assigns the value $0$ to the zero element of $\mathcal{A}$.
For some background on submeasures the reader may consult, for example, 
the papers~\cite{HererChristensen,KaltonRoberts,solecki99,todorcevic04,talagrand80,talagrand08}. 
For concreteness, let us make use of Stone's representation theorem for Boolean algebras~\cite{stone}
and assume that $\mathcal{A}$ is a Boolean algebra of subsets of some set $X$.
A submeasure can be viewed as a metric, or a pseudo-metric, on an algebra of sets that respects the structure of the algebra, namely, 
$\phi$ induces a pseudo-metric on $\mathcal A$ by the formula 
\begin{equation}\label{E:sume}
	d_\phi(A, B) \, \defeq \, \phi((A\setminus B) \cup (B \setminus A)).
\end{equation}
Of course, $d_\phi$ is a metric precisely when $\phi$ is strictly positive on non-empty sets in $\mathcal A$. Seeing submeasures 
as pseudo-metrics 
yields connections between submeasures and nets of $mm$-spaces, on the one hand, and submeasures and Polish topological groups, 
on the other, which, in turn, connects the concentration of measure result above with extreme amenability of certain Polish 
groups. Before we explain 
these relationships, we describe our classification of submeasures, which will be important in our considerations.

{\em Classification of submeasures.} 
With each submeasure $\phi$ defined on a Boolean algebra $\mathcal{A}$ of subsets of a set $X$, we associate 
a function $h_\phi\colon {\mathbb R}_{>0}\to {\mathbb R}_{>0}$, whose value 
at $\xi>0$ measures how thickly, relative to $\xi$, the family of elements of $\mathcal{A}$ 
with submeasure not exceeding $\xi$ covers the underlying set~$X$. More precisely, we consider
the covering number of a family of sets as introduced by Kelley~\cite{kelley59}:
for a family $\mathcal B$ of subsets of $X$,
the covering number of $\mathcal B$ is the supremum of the ratios 
\[
\frac{\max\{ k\mid \vert\{ i<n \mid x\in B_i\}\vert \geq k\hbox{ for each }x\in X\}}{n}, 
\]
where $(B_0, \dots , B_{n-1})$ varies over all sequence of elements of $\mathcal B$ with $n\geq 1$. 
Now, $h_\phi(\xi)$ is defined to be equal to the covering 
number of the family 
\[
{\mathcal A}_{\phi, \xi} \, \defeq \, \{ A\in {\mathcal A}\mid \phi(A)\leq \xi\}
\]
divided by $\xi$. In Theorem~\ref{theorem:submeasure.classification}, 
we show that the asymptotic behavior of $h_\phi$ at $0$ is rather restricted, for example, the quantity 
$h_\phi(\xi)$ tends to a limit, possibly infinite, as $\xi$ tends to $0$. A key point in this proof is Lemma~\ref{lemma:convergence}, 
which is analogous to certain convergence results on subadditive sequences, but appears not to be derivable from these results. 
We classify submeasures into hyperbolic, parabolic, and elliptic according to the asymptotic behavior of~$h_\phi$; using 
Landau's big $O$ notation, the submeasure $\phi$ is
hyperbolic if $\frac{1}{h_\phi(\xi)}= O(\xi)$ as~$\xi \to 0$, elliptic if $h_\phi(\xi)= O(\xi)$ as $\xi \to 0$, and parabolic otherwise. In 
Theorem~\ref{theorem:submeasure.classification}, we relate this classification to the two well-studied classes of submeasures: 
measures and 
pathological submeasures. In particular, using a result of Christensen~\cite{christensen}, 
we show that a submeasure is hyperbolic precisely when it is pathological. (Recall that a submeasure that is additive on pairs of 
disjoint sets is called a {\em measure}; a submeasure is called {\em pathological} if it does not have a non-zero measure below it.)

{\em Submeasures as functors from probability spaces to nets of $mm$-spaces.} 
An {\em $mm$-space}, or a {\em metric measure space}, is a standard Borel space equipped with a probability measure and 
a pseudo-metric that are compatible with each other. 
Assume we have a submeasure $\phi$ defined on an algebra $\mathcal{A}$ of subsets of some set $X$. The family of all partitions of 
the underlying set $X$ into sets in $\mathcal A$ with the relation of refinement forms a directed partial order. Given 
a standard Borel probability space $(\Omega, \mu)$, we associate with each such partition $\mathcal B$ an $mm$-space 
by equipping the product space $\Omega^{\mathcal B}$ of all function from $\mathcal B$ to $\Omega$ with the product measure 
arising from $\mu$ and 
a pseudo-metric $\delta_{\phi, {\mathcal B}}$ that naturally extends formula \eqref{E:sume} by setting
\[
\delta_{\phi, {\mathcal B}}(x,y) \, \defeq \, \phi\!\left(\bigcup \{ B \in {\mathcal B}\mid x(B) \not= y(B)\}\right). 
\]
This procedure associates with $\phi$ a net of $mm$-spaces indexed by finite partitions of $X$ into elements of $\mathcal A$. A
natural question arises whether the nets of $mm$-spaces obtained this way are L{\'e}vy, that is, 
whether they exhibit concentration of measure.
Using our concentration of measure result, 
we prove in Theorem~\ref{theorem:covering.concentration.for.non.elliptic.submeasures} that the nets of $mm$-spaces associated 
with hyperbolic and 
parabolic submeasures are L{\'e}vy. 
On the other hand, in Example~\ref{example:berry.esseen}, we exhibit an elliptic submeasure such that the net of $mm$-spaces 
associated with it is not L{\'e}vy, showing that Theorem~\ref{theorem:covering.concentration.for.non.elliptic.submeasures} is 
essentially~sharp.

{\em Submeasures as functors from topological groups to topological groups.} 
Given a topological group $G$, we consider the topological group $L_0(\phi, G)$ of all functions $f$ from $X$ to $G$ 
that are constant on the elements of a finite partition $\mathcal{B} \subseteq \mathcal{A}$ of $X$, with $\mathcal B$ depending on $f$. The group $L_0(\phi, G)$ 
is equipped with pointwise multiplication. The topology on it is defined again by extending formula \eqref{E:sume}. 
Given $\epsilon>0$ and a neighborhood $U$ of the neutral element in $G$, a basic neighborhood 
of $f\in L_0(\phi, G)$ in $L_0(\phi, G)$ consists of all $g\in L_0(\phi, G)$ such that 
\[
\phi(\{ x \in X \mid g(x) \not\in U f(x) \}) \, < \, \epsilon . 
\]
A construction of this type was first carried out by Hartman--Mycielski~\cite{HartmanMycielski}, in the case of $\phi$ being 
a measure, and 
by Herer--Christensen~\cite{HererChristensen}, in the case of a general submeasure. 
We ask when $L_0(\phi, G)$ is extremely amenable, that is, for what $\phi$ and $G$, does each continuous 
actions of $L_0(\phi, G)$ on a compact Hausdorff space have a fixed point? Results pertaining to this questions were obtained 
by Herer--Christensen~\cite{HererChristensen}, Glasner~\cite{glasner98}, Pestov~\cite{pestov02}, 
Farah--Solecki \cite{FarahSolecki}, Sabok~\cite{sabok}, and Pestov--Schneider~\cite{PestovSchneider}. For a broader background on extreme amenability the reader 
may consult \cite{PestovBook}. 
Our classification of submeasures plays a role here, too. In Theorem~\ref{theorem:whirly.amenability}, we 
connect covering concentration of submeasures~$\phi$ and extreme amenability of groups $L_0(\phi, G)$ 
for amenable $G$. 
Using this theorem and our result on L{\'e}vy nets described above, we show in Corollary~\ref{corollary:whirly.amenability} 
that if $\phi$ is hyperbolic or parabolic and $G$ is amenable, then 
$L_0(\phi, G)$ is extremely amenable, in fact, it is even whirly amenable. This gives a common strengthening of the results 
from \cite{HererChristensen,glasner98,pestov02,PestovSchneider} and also of a large portion of the results from \cite{FarahSolecki,sabok}. 
In the other direction, by extending an argument from \cite{PestovSchneider}, we show in 
Proposition~\ref{proposition:reflecting.amenability} 
that if $\phi$ is parabolic or elliptic and $G$ is not amenable, then $L_0(\phi, G)$ is not extremely 
amenable, in fact, it is not even amenable.

\section{Measure concentration and entropy}

The purpose of this preliminary section is to provide the background material necessary for stating and proving the results of Section~\ref{section:covering.concentration}. This will include both a quick review of generalities concerning concentration of measure (Section~\ref{section:measure.concentration}) and a discussion of a specific information-theoretic method for establishing concentration inequalities (Section~\ref{section:entropy.method}).

\subsection{A review of measure concentration}\label{section:measure.concentration}

Let us briefly recall some of the general background concerning the phenomenon of \emph{measure concentration}~\cite{levy,milman,MilmanSchechtman,GromovMilman}. For more details, the reader is referred to~\cite{ledoux,massart07}. For a start, let us clarify some pieces of notation. If $(X,d)$ is a pseudo-metric space, then, for any $A \subseteq X$ and $\epsilon \in \mathbb{R}_{>0}$, we let 
\begin{displaymath}
	B_{d}(A,\epsilon ) \, \defeq \, \{ x \in X \mid \exists a \in A \colon \, d(a,x) < \epsilon \} .
\end{displaymath} 
Let us note that, if $X$ is a standard Borel space and $d\colon X\times X\to {\mathbb R}$ is 
a Borel measurable pseudo-metric on $X$, then for any Borel measurable $A \subseteq X$ and 
$\epsilon \in \mathbb{R}_{>0}$, the set 
$B_{d}(A,\epsilon )$ is $\mu$-measurable for every probability measure $\mu$ on $X$; 
see~\cite[Theorem~2.12]{crauel}.

From this point on, when talking about subsets of a standard Borel space or functions on such a space, we will say {\em measurable} for {\em Borel measurable} and use {\em $\mu$-measurable} if we mean measurability with respect to 
a measure $\mu$. 

\begin{definition}\label{definition:concentration} Let $(X,d,\mu)$ be a \emph{metric measure space}, that is, $X$ is a standard Borel space, $d$ is a
measurable pseudo-metric on $X$, and $\mu$ is a probability measure on~$X$. The mapping $\alpha_{(X,d,\mu)} \colon \mathbb{R}_{>0} \to [0,1]$ defined by \begin{displaymath}
	\alpha_{(X,d,\mu)}(\epsilon) \, \defeq \, 1 - \inf \!\left\{ \mu (B_{d}(A,\epsilon)) \left\vert \, A \subseteq X \text{ measurable, } \, \mu (A) \geq \tfrac{1}{2} \right\} \right.
\end{displaymath} is called the \emph{concentration function} of $(X,d,\mu)$. A net $(X_{i},d_{i},\mu_{i})_{i \in I}$ of metric measure spaces is said to be a \emph{L\'evy net} if, for every family of measurable sets $A_{i} \subseteq X_{i}$ ($i \in I$), \begin{displaymath}
	\liminf\nolimits_{i \in I} \mu_{i}(A_{i}) \, > \, 0 \quad \Longrightarrow \quad \forall \epsilon \in \mathbb{R}_{> 0} \colon \ \lim\nolimits_{i \in I} \mu_{i}(B_{d_{i}}(A_{i},\epsilon)) \, = \, 1 .
\end{displaymath} \end{definition}

Let us recollect some basic facts about concentration. Given two measurable spaces $S$ and~$T$ as well as a measure $\mu$ on $S$, the \emph{push-forward measure} of $\mu$ along a measurable map $f \colon S \to T$ will be denoted by $f_{\ast}(\mu)$, that is, $f_{\ast}(\mu)$ is the measure on $T$ defined by $f_{\ast}(\mu)(B) \defeq \mu (f^{-1}(B))$ for every measurable subset $B \subseteq T$.

\begin{remark}\label{remark:concentration} The following hold. \begin{enumerate}
	\item[$(1)$] For every metric measure space $(X,d,\mu)$, the map $\alpha_{(X,d,\mu)} \colon \mathbb{R}_{>0} \to [0,1]$ is monotonically decreasing.
	\item[$(2)$] Let $(X_{0},d_{0},\mu_{0})$ and $(X_{1},d_{1},\mu_{1})$ be metric measure spaces. If there exists a measurable $1$-Lipschitz map $f \colon (X_{0},d_{0}) \to (X_{1},d_{1})$ with $f_{\ast}(\mu_{0}) = \mu_{1}$, then \begin{displaymath}
			\alpha_{(X_{1},d_{1},\mu_{1})} \, \leq \, \alpha_{(X_{0},d_{0},\mu_{0})}
		\end{displaymath} (see~\cite[Lemma~2.2.5]{PestovBook}).
	\item[$(3)$] A net $(X_{i},d_{i},\mu_{i})_{i \in I}$ of metric measure spaces is a L\'evy net if and only if \begin{displaymath}
					\lim\nolimits_{i \in I} \alpha_{(X_{i},d_{i},\mu_{i})}(\epsilon ) \, = \, 0
				\end{displaymath} for every $\epsilon \in \mathbb{R}_{>0}$ (see~\cite[Remark~1.3.3]{PestovBook}).
\end{enumerate} \end{remark}

In this work, we deduce concrete estimates for concentration functions of a large family of metric measure spaces by bounding the measure-theoretic entropy of their $1$-Lipschitz functions. Fundamental to this approach is the following elementary observation, where we let $\mathbb{E}_{\mu}(f) \defeq \int f \, d\mu$ for a probability space $(X,\mu)$ and a $\mu$-integrable function $f \colon X \to \mathbb{R}$.

\begin{prop}[\cite{ledoux}, Proposition~1.7]\label{proposition:concentration.function} Let $(X,d,\mu)$ be a metric measure space and~consider any function $\alpha \colon \mathbb{R}_{>0} \to \mathbb{R}_{\geq 0}$. Suppose that, for every bounded measurable $1$-Lipschitz function $f \colon (X,d) \to \mathbb{R}$ and every $r \in \mathbb{R}_{>0}$, \begin{displaymath}
	\mu (\{ x \in X \mid f(x) - \mathbb{E}_{\mu}(f) \geq r \}) \, \leq \, \alpha (r) .
\end{displaymath} Then $\alpha_{(X,d,\mu)}(r) \leq \alpha \!\left(\tfrac{r}{2}\right)$ for all $r \in \mathbb{R}_{>0}$. \end{prop}

The concentration results to be proved in Section~\ref{section:covering.concentration} will be shown to have interesting applications in topological dynamics (see Section~\ref{section:applications}). As this will require us to connect concentration of measure with the study of general topological groups, we conclude this section by briefly recollecting and commenting on the concept of measure concentration in uniform spaces, as introduced by Pestov~\cite[Definition~2.6]{pestov02}. To clarify some terminology, let $X$ be a uniform space, in the usual sense of Bourbaki~\cite[Chapter~II]{bourbaki}. An entourage $U$ of $X$ will be called \emph{open} if $U$ constitutes an open subset of $X \times X$ with respect to the product topology generated from the topology induced by the uniformity of $X$ (see~\cite[Chapter~II, \S1.2]{bourbaki} for details). It is easy to see that, for any open entourage $U$ of $X$ and any subset $A \subseteq X$, \begin{displaymath}
	U[A] \, \defeq \, \{ y \in X \mid \exists x \in A \colon \, (x,y) \in U \}  
\end{displaymath} is an open (in particular, Borel measurable) subset of $X$. Moreover, let us recall that the collection of all open entourages of $X$ forms a fundamental system of entourages of $X$, that is, a filter base of the uniformity of $X$ (\cite[Chapter~II, \S1.2, Corollary~2]{bourbaki}).

\begin{definition}[\cite{pestov02}, Definition~2.6]\label{definition:concentration.in.uniform.spaces} Let $X$ be a uniform space. A net $(\mu_{i})_{i \in I}$ of Borel probability measures on $X$ is said to \emph{concentrate in $X$} (or called a \emph{L\'evy net in $X$}) if, for every family $(A_{i})_{i \in I}$ of Borel subsets of $X$ and any open entourage $U$ of $X$, \begin{displaymath}
	\liminf\nolimits_{i \in I} \mu_{i}(A_{i}) \, > \, 0 \quad \Longrightarrow \quad \lim\nolimits_{i \in I} \mu_{i}(U[A_{i}]) \, = \, 1 .
\end{displaymath} \end{definition}

\begin{remark}[\cite{GromovMilman}, 2.1; \cite{pestov02}, Lemma~2.7]\label{remark:concentration.in.uniform.spaces} 
Let $(X_{i},d_{i},\mu_{i})_{i \in I}$ be a L\'evy net of metric measure spaces, let $Y$ be a uniform space, and let $f_{i} \colon X_{i} \to Y$ for each 
$i \in I$. If the family $(f_{i})_{i \in I}$ is uniformly equicontinuous, that is, for every entourage $U$ of $Y$ there exists~$\epsilon \in \mathbb{R}_{>0}$ such that \begin{displaymath}
	\forall i \in I \, \forall x,y \in X_{i} \colon \quad d_{i}(x,y) \leq \epsilon \, \Longrightarrow \, (f_{i}(x),f_{i}(y)) \in U ,
\end{displaymath} then the net $((f_{i})_{\ast}(\mu_{i}))_{i \in I}$ concentrates in $X$. \end{remark}

\subsection{The entropy method and the Herbst argument}\label{section:entropy.method}

The idea of applying information-theoretic arguments to derive concentration inequalities has its origin in the pioneering work of Marton~\cite{marton86,marton96} and Ledoux~\cite{ledoux95,ledoux96}. The presentation here will focus on the so-called \emph{Herbst argument} developed by Ledoux building on an idea of Herbst. For a comprehensive introduction to this method, the reader is referred to~\cite[Section~1.2.3]{massart07}.  We start off with a definition.

\begin{definition}[\cite{massart07}, Definition~2.11; or \cite{ledoux}, page~91] Let $(\Omega, \mu)$ be a probability space and let $f \colon \Omega \to \mathbb{R}_{\geq 0}$ be $\mu$-integrable. The \emph{entropy} of $f$ with respect to $\mu$ is defined as \begin{displaymath}
	\mathrm{Ent}_{\mu}(f) \, \defeq \, \int f(x) \ln f(x) \, d\mu(x) - \left( \int f(x) \, d\mu(x) \right) \ln \! \left( \int f(x) \, d\mu(x)  \right) .
\end{displaymath} \end{definition}

For an arbitrary probability space $(\Omega, \mu)$ and a $\mu$-integrable function $f \colon \Omega \to \mathbb{R}_{\geq 0}$ with $\mathbb{E}_{\mu}(f) > 0$, the quantity $\mathrm{Ent}_{\mu}(f)$ coincides, up to a normalizing constant, with the \emph{Kullback--Leibler divergence} or \emph{relative entropy} of the probability measure $\nu$ with respect to $\mu$, where
$\nu (A) \defeq \tfrac{1}{\mathbb{E}_{\mu}(f)} \int_{A} f \, d\mu$ for every measurable subset $A \subseteq \Omega$. For more details on relative entropy, we refer to~\cite[Section~IX]{MadimanTetali}.

We recall the following dual characterization of entropy, where $\overline{\mathbb{R}} \defeq \mathbb{R} \cup \{ -\infty, \infty \}$.

\begin{prop}[\cite{massart07}, Proposition~2.12; or \cite{ledoux}, page~98]\label{proposition:dual.entropy} Let $(\Omega, \mu)$ be a probability space and let $f \colon \Omega \to \mathbb{R}_{\geq 0}$ be $\mu$-integrable. Then \begin{displaymath}
	\mathrm{Ent}_{\mu}(f) \, = \, \sup \! \left\{ \int g f \, d\mu \left\vert \, g \colon \Omega \to \overline{\mathbb{R}} \textit{ measurable}, \, \int \exp \circ g \, d\mu \leq 1 \right\} . \right.
\end{displaymath} \end{prop}

We note a slight variation of Proposition~\ref{proposition:dual.entropy}.

\begin{cor}\label{corollary:dual.entropy} Let $(\Omega, \mu)$ be a probability space and let $f \colon \Omega \to \mathbb{R}_{\geq 0}$ be $\mu$-integrable. Then \begin{displaymath}
	\mathrm{Ent}_{\mu}(f) \, = \, \sup \! \left\{ \int g f \, d\mu \left\vert \, g \colon \Omega \to \mathbb{R} \textit{ measurable}, \, \int \exp \circ g \, d\mu \leq 1 \right\} . \right.
\end{displaymath} \end{cor}

\begin{proof} Clearly, if $\int f \, d\mu = 0$, then $\mathrm{Ent}_{\mu}(f) = 0$ and $f(x) = 0$ for $\mu$-almost every~$x \in \Omega$, so that the desired equality holds trivially. Therefore, we may and will assume that $\alpha \defeq \int f \, d\mu > 0$. Moreover, thanks to Proposition~\ref{proposition:dual.entropy}, it suffices to verify that \begin{equation}\label{entropy}
	\mathrm{Ent}_{\mu}(f) \, \leq \, \sup \! \left\{ \int g f \, d\mu \left\vert \, g \colon \Omega \to \mathbb{R} \text{ measurable}, \, \int \exp \circ g \, d\mu \leq 1 \right\} . \right.
\end{equation} For this, let $\epsilon \in \mathbb{R}_{>0}$. Put $\beta \defeq \mu (B)$ for the measurable set $B \defeq \{ x \in \Omega \mid f(x) = 0 \}$. Choose any $\delta \in \mathbb{R}_{>0}$ with $\alpha \delta \leq \epsilon$ and then $n \in \mathbb{N}$ such that $\exp (-n) \leq 1 - \exp (-\delta)$. Consider the measurable function $g \colon \Omega \to \mathbb{R}$ defined by \begin{displaymath}
	g(x) \, \defeq \, \begin{cases}
			\, \ln f(x) - \ln \alpha - \delta & \text{if } x \in \Omega \setminus B , \\
			\, -n & \text{otherwise}
		\end{cases}
\end{displaymath} for all $x \in \Omega$. We observe that \begin{displaymath}
	\int \exp \circ g \, d\mu \, = \, \exp (-\delta) \alpha^{-1} \int_{\Omega \setminus B} f \, d\mu + \exp (-n)\beta \, \leq \, \exp (-\delta) + \exp (-n) \, \leq \, 1
\end{displaymath} and \begin{align*}
	\int f g \, d\mu \, & = \, \int f(x) (\ln f(x) - \ln \alpha - \delta) \, d\mu(x) \\
	& = \, \int f(x) \ln f(x) \, d\mu(x) - \alpha \ln \alpha - \alpha \delta \, = \, \mathrm{Ent}_{\mu}(f) - \alpha \delta \, \geq \, \mathrm{Ent}_{\mu}(f) - \epsilon .
\end{align*} This proves~\eqref{entropy} and hence completes the argument. \end{proof}

When estimating entropy in Section~\ref{section:covering.concentration}, we will moreover make use of the following.

\begin{lem}[\cite{ledoux}, Corollary~5.8]\label{lemma:ledoux} Let $(\Omega,\mu)$ be a probability space and $f \colon \Omega \to \mathbb{R}$ be $\mu$-integrable. Then \begin{displaymath}
	\mathrm{Ent}_{\mu}(\exp \circ f) \, \leq \, \int \int_{f(x) \geq f(y)} (f(x)-f(y))^{2}\exp(f(x)) \, d \mu(y) \, d\mu(x) .
\end{displaymath} \end{lem}

\begin{proof} Applying Jensen's inequality and Fubini's theorem, we see that \begin{align*}
	\mathrm{Ent}_{\mu}(\exp \circ f) \, &= \, \int f(x) \exp(f(x)) \, d\mu(x) - \mathbb{E}_{\mu}(\exp \circ f) \ln \mathbb{E}_{\mu}(\exp \circ f) \\
	& \leq \, \int f(x) \exp(f(x)) \, d\mu(x) - \left( \int \exp(f(x)) \, d\mu(x) \right) \! \left( \int f(x) \, d\mu(x)  \right) \\
	& = \, \frac{1}{2} \int \int (f(x)-f(y))(\exp (f(x)) - \exp (f(y))) \, d\mu(y) \, d \mu (x) \\
	& = \, \int\nolimits_{f(x)\geq f(y)} (f(x)-f(y))(\exp (f(x)) - \exp (f(y))) \, d(\mu \otimes \mu)(x,y) .
\end{align*} 
Furthermore, a straightforward application of the mean value theorem shows that, if $a,b \in \mathbb{R}$ and $a \geq b$, then 
$\exp(a)-\exp(b) \, \leq \, \exp (a)(a-b)$, thus 
\begin{displaymath}
	(a-b)(\exp(a)-\exp(b)) \, \leq \, (a-b)^{2}\exp (a) .
\end{displaymath} 
Combining this inequality with Fubini's theorem, we conclude that 
\begin{align*}
	\mathrm{Ent}_{\mu}(\exp \circ f) \, & \leq \, \int\nolimits_{f(x)\geq f(y)} (f(x)-f(y))^{2}\exp (f(x)) \, d(\mu \otimes \mu)(x,y) \\
	& = \, \int \int\nolimits_{f(x)\geq f(y)} (f(x)-f(y))^{2}\exp (f(x)) \, d\mu(y) \, d\mu (x) .\qedhere
\end{align*} \end{proof}

Our interest in entropy is due to the following fact, known as the \emph{Herbst~argument}.

\begin{prop}[Herbst argument, \cite{massart07}, Proposition~2.14]\label{proposition:herbst} Let $(\Omega,\mu)$ be a probability space, let $f \colon \Omega \to \mathbb{R}$ be $\mu$-integrable, and let $D \in \mathbb{R}_{>0}$. Suppose that, for each $\lambda \in \mathbb{R}_{> 0}$, 
\begin{displaymath}
	\mathrm{Ent}_{\mu}(\exp \circ (\lambda f)) \, \leq \, \tfrac{1}{2}\lambda^{2}D\int \exp \circ (\lambda f) \, d\mu .
\end{displaymath} 
Then, for each $\lambda \in \mathbb{R}_{> 0}$, 
\begin{displaymath}
	\int \exp (\lambda (f(x) - \mathbb{E}_{\mu}(f))) \, d\mu(x) \, \leq \, \exp \! \left( \tfrac{1}{2}\lambda^{2}D \right) .
\end{displaymath} \end{prop}

The Herbst argument provides a technique for proving concentration of measure, via combining it with Proposition~\ref{proposition:concentration.function} and the following well-known fact.

\begin{prop}\label{proposition:markov} Let $(\Omega,\mu)$ be a probability space, let $f \colon \Omega \to \mathbb{R}$ be $\mu$-integrable, and let $D \in \mathbb{R}_{>0}$. Suppose that for each $\lambda \in \mathbb{R}_{> 0}$ 
\begin{displaymath}
	\int \exp (\lambda (f(x) - \mathbb{E}_{\mu}(f))) \, d\mu(x) \, \leq \, \exp \! \left( \tfrac{1}{2}\lambda^{2}D \right) .
\end{displaymath} 
Then, for each $r \in \mathbb{R}_{> 0}$,
\begin{displaymath}
	\mu (\{ x \in \Omega \mid f(x) - \mathbb{E}_{\mu}(f) \geq r \}) \, \leq \, \exp \! \left( -\tfrac{r^{2}}{2D} \right) .
\end{displaymath} \end{prop}

\begin{proof} Let $r \in \mathbb{R}_{> 0}$. By Markov's inequality, our hypothesis implies that \begin{align*}
	\mu (\{ x \in \Omega \mid f(x)& - \mathbb{E}_{\mu}(f) \geq r \}) \, = \, \mu\!\left(\left\{ x \in \Omega \left\vert \, \exp \!\left( \lambda (f(x) - \mathbb{E}_{\mu}(f)) \right) \geq \exp \!\left( \lambda r \right) \right\}\right) \right. \\
		& \leq \, \exp \! \left( -\lambda r \right) \int \exp \!\left( \lambda (f(x) - \mathbb{E}_{\mu}(f)) \right) \, d\mu(x) \, \leq \, \exp  \! \left( \tfrac{1}{2}\lambda^{2}D - \lambda r \right)
\end{align*} for every $\lambda \in \mathbb{R}_{>0}$. Choosing $\lambda \defeq \tfrac{r}{D}$, we conclude that \begin{displaymath}
	\mu (\{ x \in \Omega \mid f(x) - \mathbb{E}_{\mu}(f) \geq r \}) \, \leq \, \exp  \! \left( \tfrac{1}{2}\! \left( \tfrac{r}{D} \right)^{2} \! D - \left( \tfrac{r}{D} \right) \! r \right) \, = \, \exp\!\left( -\tfrac{r^{2}}{2D} \right) .\qedhere
\end{displaymath} \end{proof}

\section{Covering concentration}\label{section:covering.concentration}

In this section, we prove concentration of measure for a new class of metric measure spaces, namely for products of probability spaces equipped with a pseudo-metric naturally arising from any weighted covering of the underlying index set (Theorem~\ref{theorem:covering.concentration} and Corollary~\ref{corollary:covering.concentration}). In addition to the tools outlined in Section~\ref{section:entropy.method}, the main technical ingredient is given by Lemma~\ref{lemma:entropic.loomis.whitney} below. Our concentration inequalities will be formulated in terms of Kelley's covering number~\cite{kelley59} -- a concept we recall in Definition~\ref{definition:covering.number.abstract}. For convenience in later considerations, we choose an abstract approach via Boolean algebras. The more concrete situation for covers of sets will be clarified in Definition~\ref{definition:covering.number.concrete} and Remark~\ref{remark:covering.number.concrete}. For a start, we set up some notation concerning finite partitions of unity in Boolean algebras.

\begin{definition} Let $\mathcal{A}$ be a Boolean algebra. A \emph{finite partition of unity in $\mathcal{A}$} is a finite subset $\mathcal{B} \subseteq \mathcal{A}\setminus \{ 0 \}$ such that \begin{enumerate}
		\item[---$\,$] $\bigvee \mathcal{B} = 1$, and
		\item[---$\,$] $A \wedge B = 0$ for any two distinct $A,B \in \mathcal{B}$.
	\end{enumerate} Denote by $\Pi ({\mathcal A})$ the set of all finite partitions of unity in $\mathcal A$. For any $\mathcal{B},\mathcal{C}\in \Pi ({\mathcal A})$, \begin{displaymath}
	\mathcal{C} \preceq \mathcal{B} \quad :\Longleftrightarrow \quad \forall B \in \mathcal{B} \ \exists C \in \mathcal{C} \colon \ B \subseteq C \, .
\end{displaymath} Moreover, for any finite subset $\mathcal{B} \subseteq \mathcal{A}$, let \begin{displaymath}
	\left. \langle \mathcal{B} \rangle_{\mathcal{A}} \, \defeq \, \left\{ \left( \bigwedge \mathcal{B}_{0} \right) \wedge \left( \bigwedge\nolimits_{B \in \mathcal{B}\setminus\mathcal{B}_{0}} \neg B \right) \, \right\vert \mathcal{B}_{0} \subseteq \mathcal{B} \right\} \setminus \{ 0 \} \, .
\end{displaymath} \end{definition}

\begin{remark} Let $\mathcal{A}$ be a Boolean algebra. If $\mathcal{B}$ is a finite subset of $\mathcal{A}$, then $\langle \mathcal{B} \rangle_{\mathcal{A}}$ is a finite partition of unity in $\mathcal{A}$. \end{remark}

We proceed to the definition of Kelley's covering number~\cite{kelley59}.

\begin{definition}\label{definition:covering.number.abstract} Let $\mathcal{A}$ be a Boolean algebra. Let $m \in \mathbb{N}_{\geq 1}$ and $\mathcal{C} = (C_{i})_{i < m} \in \mathcal{A}^{m}$. We define $\langle \mathcal{C} \rangle_{\mathcal{A}} \defeq \langle \{ C_{i} \mid i < m \} \rangle_{\mathcal{A}}$ and call \begin{displaymath}
	t_{\mathcal{A}}(\mathcal{C}) \, \defeq \, \sup \{ k \in \mathbb{N} \mid \forall B \in \langle \mathcal{C} \rangle_{\mathcal{A}} \colon \, \vert \{ i < m \mid B \leq C_{i} \} \vert \geq k \} 
\end{displaymath} the \emph{covering multiplicity} of $\mathcal{C}$ in $\mathcal{A}$. Let $k \in \mathbb{N}_{\geq 1}$. Then $\mathcal{C}$ is said to be \begin{enumerate}
	\item[---$\,$] a \emph{$k$-cover} in $\mathcal{A}$ if $t_{\mathcal{A}}(\mathcal{C}) \geq k$,
	\item[---$\,$] a \emph{cover} in $\mathcal{A}$ if $\mathcal{C}$ a $1$-cover in $\mathcal{A}$, and 
	\item[---$\,$] \emph{uniform} (in $\mathcal{A}$) if $\vert \{ i < m \mid B \leq C_{i} \} \vert = t_{\mathcal{A}}(\mathcal{C})$ for every $B \in \langle \mathcal{C} \rangle_{\mathcal{A}}$.
\end{enumerate} The \emph{covering number} of a subset $\mathcal{B} \subseteq \mathcal{A}$ is defined to be \begin{displaymath}
	\left. c_{\mathcal{A}}(\mathcal{B}) \, \defeq \, \sup \! \left\{ \tfrac{t_{\mathcal{A}}((B_{i})_{i < n})}{n} \, \right\vert n \in \mathbb{N}_{\geq 1}, \, (B_{i})_{i < n} \in \mathcal{B}^{n} \right\} .
\end{displaymath} \end{definition}

The definition above is stable under partition refinement in the following sense.

\begin{remark}\label{remark:covering.number.partition.refinement} Let $\mathcal{A}$ be a Boolean algebra. Let $m \in \mathbb{N}_{\geq 1}$ and $\mathcal{C} = (C_{i})_{i < m} \in \mathcal{A}^{m}$. Consider any $\mathcal{B} \in \Pi (\mathcal{A})$ with $\langle \mathcal{C} \rangle_{\mathcal{A}} \preceq \mathcal{B}$. Then \begin{displaymath}
	t_{\mathcal{A}}(\mathcal{C}) \, = \, \sup \{ k \in \mathbb{N} \mid \forall B \in \mathcal{B} \colon \, \vert \{ i < m \mid B \leq C_{i} \} \vert \geq k \} \, .
\end{displaymath} Moreover, $\mathcal{C}$ is uniform in $\mathcal{A}$ if and only if $\vert \{ i < m \mid B \leq C_{i} \} \vert = t_{\mathcal{A}}(\mathcal{C})$ for each $B \in \mathcal{B}$. \end{remark}

Furthermore, let us point out the following simple, but useful observation about uniform refinements of covers.

\begin{lem}\label{lemma:covering.numbers} Let $\mathcal{A}$ be a Boolean algebra. Let $m \in \mathbb{N}_{\geq 1}$ and let $\mathcal{C} = (C_{i})_{i < m} \in \mathcal{A}^{m}$ be a cover in $\mathcal{A}$. Then there exists a uniform $t_{\mathcal{A}}(\mathcal{C})$-cover $\mathcal{C}^{\ast} = (C_{i}^{\ast})_{i < m} \in \mathcal{A}^{m}$ in $\mathcal{A}$ such that $C_{i}^{\ast} \leq C_{i}$ for each $i < m$. \end{lem}

\begin{proof} Let $k \defeq t_{\mathcal{A}}(\mathcal{C})$ and let us denote by $\mathcal{P}_{k}(m)$ the set of all $k$-element subsets of $\{ 0,\ldots,m-1 \}$. Consider $\mathcal{B} \defeq \langle \{ C_{i} \mid i < m \} \rangle_{\mathcal{A}} \in \Pi (\mathcal{A})$. Since $\mathcal{C}$ is a $k$-cover in $\mathcal{A}$, there exists a map $\pi \colon \mathcal{B} \to \mathcal{P}_{k}(m)$ such that \begin{displaymath}
	\forall B \in \mathcal{B} \ \forall i \in \pi (B) \colon \qquad B \, \leq \, C_{i} \, .
\end{displaymath} For each $i < m$, let $C_{i}^{\ast} \defeq \bigvee \{ B \in \mathcal{B} \mid i \in \pi (B) \}$. Clearly, $\mathcal{C}^{\ast} \defeq (C_{i}^{\ast})_{i < m} \in \mathcal{A}^{m}$ and $C_{i}^{\ast} \leq C_{i}$ whenever $i < m$. Since $\mathcal{B}$ is a partition of unity in $\mathcal{A}$, the definition of $\mathcal{C}^{\ast}$ moreover entails that \begin{displaymath}
	\vert \{ i < m \mid B \leq C_{i}^{\ast} \} \vert \, = \, \vert \pi (B) \vert \, = \, k 
\end{displaymath} for each $B \in \mathcal{B}$. According to Remark~\ref{remark:covering.number.partition.refinement}, as $\langle \mathcal{C}^{\ast} \rangle_{\mathcal{A}} \preceq \mathcal{B}$, this implies that $\mathcal{C}^{\ast}$ is a uniform $k$-cover in $\mathcal{A}$. \end{proof}

We are going to clarify the concepts introduced above in the concrete setting of set covers. Given a set $X$, let us denote by $\mathcal{P}(X)$ the power set of $X$, which constitutes a Boolean algebra with respect to the usual set-theoretic operations.

\begin{definition}\label{definition:covering.number.concrete} Let $X$ be a set, $k,m \in \mathbb{N}_{\geq 1}$. A sequence $\mathcal{C} \in \mathcal{P}(X)^{m}$ is called \begin{enumerate}
	\item[---$\,$] a \emph{$k$-cover} of $X$ if $\mathcal{C}$ is a $k$-cover in $\mathcal{P}(X)$,
	\item[---$\,$] a \emph{cover} of $X$ if $\mathcal{C}$ is a cover in $\mathcal{P}(X)$, and
	\item[---$\,$] \emph{uniform} (over $X$) if $\mathcal{C}$ is uniform in $\mathcal{P}(X)$.
\end{enumerate} \end{definition}

Of course, a finite sequence of subsets of a set $X$ constitutes a cover of $X$ in the sense of Definition~\ref{definition:covering.number.concrete} if and only if its union coincides with $X$. Let us mention some additional elementary observations.

\begin{remarks}\label{remark:covering.number.concrete} (1) Let $X$ be a set and let $\mathcal{C} = (C_{i})_{i < m} \in \mathcal{P}(X)^{m}$ with $m \in \mathbb{N}_{\geq 1}$. Then \begin{displaymath}
	t_{\mathcal{P}(X)}(\mathcal{C}) \, = \, \sup \{ k \in \mathbb{N} \mid \forall x \in X \colon \, \vert \{ i < m \mid x \in C_{i} \} \vert \geq k \} \, . 
\end{displaymath} Furthermore, the sequence $\mathcal{C}$ is uniform over $X$ if and only if, for every $x \in X$, \begin{displaymath}
	\vert \{ i < m \mid x \in C_{i} \} \vert \, = \, t_{\mathcal{P}(X)}(\mathcal{C}) \, .
\end{displaymath}

(2) Let $\mathcal{A}$ be a Boolean algebra. Let $k,m \in \mathbb{N}_{\geq}$ and let $\mathcal{C} = (C_{i})_{i < m} \in \mathcal{A}$. Consider any $\mathcal{B} \in \Pi ({\mathcal A})$ with $\langle \mathcal{C} \rangle_{\mathcal{A}} \preceq \mathcal{B}$. Then $\mathcal{C}$ is a (uniform) $k$-cover in $\mathcal{A}$ if and only if the sequence $(\{ B \in \mathcal{B} \mid B \leq C_{i} \})_{i < m} \in \mathcal{P}(\mathcal{B})$ is a (uniform) $k$-cover of the set $\mathcal{B}$. \end{remarks}

Let us now proceed to an analogue of Shearer's lemma~\cite[p.~33, item~(22)]{shearer} for differential entropy due to Madiman--Tetali~\cite[Corollary~VIII]{MadimanTetali}, which simultaneously generalizes earlier work of Han~\cite{han}. This result (Lemma~\ref{lemma:entropic.loomis.whitney} below) was proved by Madiman--Tetali extending an argument by Massart~\cite[Section~2.1.1]{massart00} proving Han's inequality for differential entropy. For the sake of convenience, we will include another proof of Lemma~\ref{lemma:entropic.loomis.whitney}, which is based on Ledoux's proof of Han's inequality for differential entropy~\cite[Proposition~5.6]{ledoux}.

To clarify some notation, let $N$ be a finite set and let $(\Omega_{j})_{j\in N}$ be a family of measurable spaces. If $x\in \prod_{j \in S} \Omega_{j}$ and $y\in \prod_{j \in T} \Omega_{j}$ for disjoint subsets $S,T\subseteq N$, then we will write $(x,y)$ for the unique element of $\prod_{j \in S\cup T} \Omega_{j}$ that projects to $x$ and~$y$. Furthermore, if $f \colon \prod_{j \in N} \Omega_{j} \to \mathbb{R}$ is a measurable function, then, for any subset $S \subseteq N$ and $z \in \prod_{j \in N\setminus S} \Omega_{j}$, the map \begin{displaymath}
	f_{z} \colon\, \prod\nolimits_{j \in S} \Omega_{j} \, \longrightarrow \, \mathbb{R}, \quad x \, \longmapsto \, f(x, z)
\end{displaymath} is measurable, too. (Note that $S$ can be recovered from $z$, so there is no ambiguity about the domain of $f_z$.) Now, for each $j \in N$, let $\mu_{j}$ be a probability measure on~$\Omega_{j}$. Set $\mu \defeq (\mu_{j})_{j \in N}$. Given a subset $B \subseteq N$, we consider the probability measure \begin{displaymath}
	P_{B}^{\mu} \, \defeq \, \bigotimes\nolimits_{j \in B} \mu_{j}
\end{displaymath} on the measurable space $\prod_{j \in B} \Omega_{j}$. We set \begin{displaymath}
	{\mathbb P}^{\mu} \, \defeq \, P^{\mu}_N.
\end{displaymath} With this notation, Fubini's theorem states that, for every $\mathbb{P}^{\mu}$-integrable function $f \colon \prod_{j \in N} \Omega_{j} \to \mathbb{R}$ and every $B \subseteq N$, the map $f_{z}$ is $P^{\mu}_{B}$-integrable for $P^{\mu}_{N\setminus B}$-almost every $z \in \prod_{j \in N\setminus B} \Omega_{j}$, and  \begin{displaymath}
	\int f \, d{\mathbb P}^{\mu} \, = \, \int \int f_{z} \, dP^{\mu}_{B} \, dP^{\mu}_{N\setminus B}(z) .
\end{displaymath}

By a \emph{standard Borel probability space}, we mean a pair $(\Omega, \mu)$ consisting of a standard Borel space $\Omega$ and a probability measure $\mu$ on $\Omega$.

\begin{lem}[Madiman--Tetali~\cite{MadimanTetali}, Corollary~VIII]\label{lemma:entropic.loomis.whitney} Let $N$ be a finite non-empty set. Let $k,m \in \mathbb{N}_{\geq 1}$ and suppose that $\mathcal{C} = (C_{i})_{i < m} \in \mathcal{P}(N)^{m}$ is a uniform $k$-cover of~$N$. Consider any family of standard Borel probability spaces $(\Omega_{j},\mu_{j})_{j\in N}$ and let $\mu \defeq (\mu_{j})_{j \in N}$. Then, for every bounded measurable function $f \colon \prod_{j\in N} \Omega_j \to \mathbb{R}_{\geq0}$, \begin{displaymath}
	\mathrm{Ent}_{\mathbb{P}^{\mu}}(f) \, \leq \, \frac{1}{k}\sum_{i < m} \int \mathrm{Ent}_{P^{\mu}_{C_{i}}}\!\left(f_{z}\right) \, dP^{\mu}_{N\setminus C_{i}}(z) .
\end{displaymath} \end{lem}

\begin{proof} We include a proof for the sake of convenience. Without loss of generality, we may assume that $N = \{ 0,\ldots,n-1\}$ for some $n \in \mathbb{N}_{\geq 1}$. We abbreviate $X=\prod_{j\in N}\Omega_j$, $\mathbb{P} \defeq \mathbb{P}^{\mu}$ and $P_{B} \defeq P_{B}^{\mu}$ for any $B \subseteq N$. We use Corollary~\ref{corollary:dual.entropy}. To this end, let $g \colon X \to \mathbb{R}$ be measurable such that $\int \exp \circ g \, d\mathbb{P} \leq 1$. Since $\exp \circ g$ takes only positive values, $\int \exp (g(y, x)) \, dP_{\{ 0,\ldots,j\}}(y) > 0$ for all $j \in N$ and $x \in \prod_{i=j+1}^{n-1} \Omega_{i}$. Furthermore, invoking Fubini's theorem, we find some measurable subset $S \subseteq X$ with $\mathbb{P}(S) = 1$ such that $\int \exp \!\left(g \!\left(y, x\!\!\upharpoonright_{\{ j+1,\ldots ,n-1\}} \right)\right) \! \, dP_{\{ 0,\ldots,j\}}(y) < \infty$ for all $j \in N$ and $x \in S$. For each $j \in N$, consider the measurable map $g^{j} \colon X \to \mathbb{R}$ given by 
\begin{displaymath}
	g^{j}(x) \, \defeq \, \ln \!\left( \frac{\int \exp \!\left(g\!\left(y, x\!\!\upharpoonright_{\{ j,\ldots ,n-1\}}\right) \right) \! \, dP_{\{ 0,\ldots,j-1\}}(y)}{\int \exp \!\left(g\!\left(y, x\!\!\upharpoonright_{\{ j+1,\ldots,n-1 \}}\right)\right) \! \, dP_{\{ 0,\ldots,j\}}(y)} \right)
\end{displaymath} for all $x \in S$ and $g^{j}(x) \defeq 0$ for all $x \in X\setminus S$. Note that, by Fubini's theorem, for each $j \in N$ and $P_{N\setminus \{ j \}}$-almost every $z \in \prod_{j' \in N\setminus \{ j \}} \Omega_{j'}$, \begin{equation}\label{base}
	\int \exp \circ g^{j}_{z} \, d\mu_{j} \, = \, \int \frac{\int \exp \!\left(g\!\left(y, x, z\!\!\upharpoonright_{\{ j+1,\ldots,n-1\}}\right)\right) \! \, dP_{\{ 0,\ldots,j-1\}}(y)}{\int \exp \!\left(g\!\left(y, z\!\!\upharpoonright_{\{ j+1,\ldots,n-1\}}\right) \right) \! \, dP_{\{ 0,\ldots,j\}}(y)} \, d\mu_{j}(x) 
		= \, 1.
\end{equation} 

Given any non-empty subset $B \subseteq N$, define the measurable function \begin{displaymath}
	h^{B} \defeq \sum\nolimits_{j\in B} g^{j} \colon \, X \, \longrightarrow \, \mathbb{R} . 
\end{displaymath} Note that $h^B$ does not depend on the $j$-th coordinates with $j<\min B$. 
We claim that, for every non-empty $B \subseteq N$ and $P_{N\setminus B}$-almost every $z \in \prod\nolimits_{j\in N \setminus B} \Omega_{j}$,
\begin{equation}\label{dual.entropy}
	\int \exp \circ h^{B}_{z} \, dP_{B} \, = \, 1 .
\end{equation} The proof of~\eqref{dual.entropy} proceeds by induction. For a start, let $B \subseteq N$ with $\vert B \vert = 1$, that is, $B = \{ j \}$ for some $j \in N$. Then, for $P_{N\setminus B}$-almost every $z \in \prod_{\ell\in N\setminus B} \Omega_\ell$, 
\begin{align*}
	\int \exp \circ h^{B}_{z} \, dP_{B} \, &= \, \int \exp \circ g^{j}_{z} \, d\mu_{j} \, \stackrel{\eqref{base}}{=} \, 1.
\end{align*} For the inductive step, let $B \subseteq N$ with $\vert B \vert > 1$ and suppose that~\eqref{dual.entropy} holds for every non-empty proper subset of $B$. Denote by $j$ the smallest element of $B$ and let $B' \defeq B \setminus \{ j \}$. Then there exists a measurable subset $T \subseteq \prod\nolimits_{\ell \in N\setminus B'} \Omega_{\ell}$ with $P_{N\setminus B'}(T) = 1$ such that, for every $z \in T$, \begin{equation}\label{inductive.hypothesis}
	 \int \exp \circ h^{B'}_{z} \, d P_{B'} \, = \, 1 .
\end{equation} 
Thanks to the Measurable Projection Theorem, see \cite[Theorem~2.12]{crauel}, the set $T' \defeq \{ {z \!\!\upharpoonright_{N\setminus B}} \mid z \in T \}$ is a $P_{N \setminus B}$-measurable subset of $\prod_{\ell \in N\setminus B} \Omega_{\ell}$. For each $z \in T'$, there exists some $\omega \in \Omega_{j} = \prod_{\ell\in \{ j\}} \Omega_\ell$ with $(\omega,z) \in T$, so that Fubini's theorem yields that \begin{align*}
	\int \exp \circ h^{B}_{z} \, dP_{B} \, & = \, \int \left(\exp \circ g^{j}_{z} \right) \! \left(\exp \circ h^{B'}_{z}\right) \, d\!\left(\mu_{j} \otimes P_{B'}\right) \\
	&= \, \int \int \left(\exp \circ g^{j}_{(y, z)} \right) \! \left(\exp \circ h^{B'}_{(y, z)}\right) \, d\mu_{j} \, d P_{B'}(y) \\
	&= \, \int \left( \int \exp \circ g^{j}_{(y, z)} \, d\mu_{j} \right) \exp \! \left( h^{B'}_{(\omega, z)} (y)\right) \, d P_{B'}(y) \\
	&\stackrel{\eqref{base}}{=} \, \int \exp \! \left( h^{B'}_{(\omega, z)} (y)\right) \, d P_{B'}(y) \, \stackrel{\eqref{inductive.hypothesis}}{=} \, 1, 
\end{align*} where the third equality follows from $h^{B'}$ not depending on the $j$-th coordinate. Since $P_{N \setminus B}(T') \geq P_{N\setminus B'}(T) = 1$, this completes our induction and therefore proves~\eqref{dual.entropy}. 

Thanks to Proposition~\ref{proposition:dual.entropy}, our assertion~\eqref{dual.entropy} implies that, for every non-empty $B \subseteq N$ and $P_{N\setminus B}$-almost every $z \in \prod\nolimits_{j\in N \setminus B} \Omega_{j}$, 
\begin{equation}\label{dual.entropy.two}
	\int h^{B}_{z} f_{z} \, dP_{B} \, \leq \, \mathrm{Ent}_{P_{B}}\!\left(f_{z}\right) .
\end{equation} Furthermore, for each $x \in S$, \begin{align*}
	\sum\nolimits_{j \in N} g^{j}(x) \, & = \, \sum\nolimits_{j \in N} \ln \!\left( \int \exp \!\left(g\!\left(y, x\!\!\upharpoonright_{\{ j,\ldots,n-1\}}\right)\right) \! \, dP_{\{ 0,\ldots,j-1\}}(y) \right) \\
		& \qquad - \sum\nolimits_{j \in N} \ln \! \left( \int \exp \!\left(g\!\left(y, x\!\!\upharpoonright_{\{ j+1,\ldots,n-1 \}}\right)\right) \! \, dP_{\{ 0,\ldots,j\}}(y) \right) \\
		&= \, g(x) - \ln \!\left( \int \exp \circ g \, d{\mathbb P} \right) \, \geq \, g(x) - \ln (1) \, = \, g(x) .
\end{align*} Since $\mathcal{C}$ is a uniform $k$-cover of $N$, this entails that \begin{displaymath}
	\sum\nolimits_{i < m} h^{C_{i}}(x) \, = \, \sum\nolimits_{i < m} \sum\nolimits_{j\in C_{i}} g^{j}(x) \, = \, k \sum\nolimits_{j \in N} g^{j}(x) \, \geq \, kg(x) 
\end{displaymath} for every $x \in S$, that is, $g \leq \frac{1}{k} \sum_{i < m} h^{C_{i}}$ $\mathbb{P}$-almost everywhere. Combining this with Fubini's theorem and~\eqref{dual.entropy.two}, we conclude that 
\begin{align*}
	\int g f \, d{\mathbb P}\, & \leq \, \frac{1}{k} \sum_{i < m} \int h^{i} f \, d{\mathbb P} \, 
	= \, \frac{1}{k} \sum_{i < m} \int \int \left(h^{C_{i}} f\right)_{z} \, dP_{C_{i}} \, dP_{N\setminus C_{i}}(z) \\
	&= \, \frac{1}{k} \sum_{i < m} \int \int h^{C_{i}}_{z} f_{z} \, dP_{C_{i}} \, dP_{N\setminus C_{i}}(z) \, \stackrel{\eqref{dual.entropy.two}}{\leq} \, \frac{1}{k} \sum_{i < m} \int \mathrm{Ent}_{P_{C_{i}}}\!\left(f_{z}\right) \, dP_{N\setminus C_{i}}(z) .
\end{align*} 
By Proposition~\ref{proposition:dual.entropy}, the conclusion follows. \end{proof}

\begin{cor}\label{corollary:entropic.loomis.whitney} Let $N$ be a finite non-empty set. Let $k,m \in \mathbb{N}_{\geq 1}$ and suppose that $\mathcal{C} = (C_{i})_{i < m} \in \mathcal{P}(N)^{m}$ is a uniform $k$-cover of~$N$. Consider any family of standard Borel probability spaces $(\Omega_{j},\mu_{j})_{j\in N}$ and let $\mu \defeq (\mu_{j})_{j \in N}$. Then, for every bounded measurable function $f \colon \prod_{j \in N} \Omega_{j} \to \mathbb{R}$,
\begin{align*}
	\mathrm{Ent}_{\mathbb{P}^{\mu}}(\exp \circ f) \leq 
	 \frac{1}{k}\sum_{i < m} \iiint_{\!\substack{\vspace{-3mm}f_z(x) \geq f_{z}(y)}} \hspace{-12mm} ( f_{z}(x) \! - \! f_{z}(y))^{2} \exp (f_{z}(x)) \, dP^{\mu}_{C_{i}}(y) \, dP^{\mu}_{C_{i}}(x) \, dP^{\mu}_{N\setminus C_{i}}(z) .
\end{align*} \end{cor}

\begin{proof} This is an immediate consequence of Lemma~\ref{lemma:entropic.loomis.whitney} and Lemma~\ref{lemma:ledoux}. \end{proof}

Next up, we introduce a pseudo-metric on the product of a family of sets naturally associated with any weighted covering of the underlying index set.

\begin{definition}\label{definition:covering.metric} Let $N$ be a finite non-empty set. Let $m \in \mathbb{N}_{\geq 1}$ and suppose that $\mathcal{C} = (C_{i})_{i < m} \in \mathcal{P}(N)^{m}$ is a cover of~$N$. Moreover, let $w = (w_i)_{i < m}$ be a sequence of non-negative reals. For a family of sets $(\Omega_{j})_{j \in N}$, we define the pseudo-metric 
\begin{displaymath}
	d_{\mathcal{C},w} \colon \, \prod\nolimits_{j\in N} \Omega_{j} \times \prod\nolimits_{j \in N} \Omega_{j} \, \longrightarrow \, \mathbb{R}_{\geq 0}
\end{displaymath} 
by setting 
\begin{displaymath}
	d_{\mathcal{C},w} (x,y) \, \defeq \, \inf \left\{ \sum\nolimits_{i \in I} w_{i} \left\vert \, I \subseteq m, \, \{ j \in N \mid x_{j} \ne y_{j} \} \subseteq \bigcup\nolimits_{i \in I} C_{i} \right\} \right.
\end{displaymath} for all $x,y \in \prod\nolimits_{j\in N} \Omega_{j}$. \end{definition}

Now everything is prepared to state and prove our first main result.

\begin{thm}\label{theorem:covering.concentration} Let $N$ be a finite non-empty set. Let $k,m \in \mathbb{N}_{\geq 1}$ and suppose that $\mathcal{C} = (C_{i})_{i < m} \in \mathcal{P}(N)^{m}$ is a $k$-cover of~$N$. Let $w = (w_i)_{i < m}$ be a sequence of non-negative reals. Consider any family of standard Borel probability spaces $(\Omega_{j},\mu_{j})_{j \in N}$ and set $\mu \defeq (\mu_{j})_{j \in N}$. Let $f \colon \prod_{j\in N} \Omega_j \to \mathbb{R}$ be measurable and $1$-Lipschitz with respect to 
$d_{\mathcal{C},w}$. Then, for every $r \in \mathbb{R}_{> 0}$, 
\begin{displaymath}
	\left. {\mathbb P}^{\mu} \!\left( \left\{ x \in \prod\nolimits_{j\in N}\Omega_j \, \right\vert f(x) - \mathbb{E}_{\mathbb{P}^{\mu}}(f) \geq r \right\} \right) \, \leq \, \exp \!\left( -\tfrac{kr^{2}}{4\Vert w \Vert_{2}^{2}} \right) .
\end{displaymath}\end{thm}

\begin{proof} Of course, the desired statement holds trivially if $w = 0$. Therefore, we may and will assume that $w \ne 0$. Due to Lemma~\ref{lemma:covering.numbers}, there exists a uniform $k$-cover $\mathcal{C}^{\ast} = (C_{i}^{\ast})_{i<m} \in \mathcal{P}(N)^{m}$ of $N$ such that $C_{i}^{\ast} \subseteq C_{i}$ for each $i < m$. Since $f$ is $1$-Lipschitz with respect to $d_{\mathcal{C},w}$, 
\begin{displaymath}
	\vert f_{z}(x) - f_{z}(y) \vert \, \leq \, d_{\mathcal{C},w}((x, z),(y,z)) \, \leq \, w_{i} 
\end{displaymath} whenever $i < m$, $x,y \in \prod_{j \in C_{i}^{\ast}} \Omega_{j}$ and $z \in \prod_{j \in N \setminus C_{i}^{\ast}} \Omega_{j}$. As the pseudo-metric $d_{\mathcal{C},w}$ is bounded, $f$~being \mbox{$1$-Lipschitz} with respect to $d_{\mathcal{C},w}$ moreover implies that $f$ is bounded. By Corollary~\ref{corollary:entropic.loomis.whitney} and Fubini's theorem, it follows that, for every $\lambda \in \mathbb{R}_{>0}$, \begin{align*}
	\mathrm{Ent}_{\mathbb{P}^{\mu}}(\exp \circ (\lambda f)) \, &\leq \, \frac{1}{k}\sum_{i < m} \iiint_{\!\substack{\vspace{-3mm}f_z(x) \geq f_{z}(y)}} \hspace{-12mm} \left( \lambda w_{i} \right)^{2} \exp \!\left(\lambda f_{z}(x)\right) \, dP^{\mu}_{C^{\ast}_{i}}(y) \, dP^{\mu}_{C^{\ast}_{i}}(x) \, dP^{\mu}_{N\setminus C^{\ast}_{i}}(z) \\[1.0ex] 
	&\leq \, \frac{1}{k}\sum_{i < m} \left( \lambda w_{i} \right)^{2} \int \int \exp \!\left(\lambda f_{z}(x)\right) \, dP^{\mu}_{C^{\ast}_{i}}(x) \, dP^{\mu}_{N\setminus C^{\ast}_{i}}(z) \\
	&= \, \frac{1}{k}\sum_{i < m} \left( \lambda w_{i} \right)^{2} \int \exp \circ (\lambda f) \, d\mathbb{P}^{\mu} \, = \, \frac{\lambda^{2}\Vert w \Vert_{2}^{2}}{k} \int \exp \circ (\lambda f) \, d\mathbb{P}^{\mu} .
\end{align*} Using Proposition~\ref{proposition:herbst} and Proposition~\ref{proposition:markov} with $D \defeq \frac{2\Vert w \Vert_{2}^{2}}{k}$ gives the conclusion. \end{proof}

\begin{cor}\label{corollary:covering.concentration} Let $N$ be a finite non-empty set. Let $k,m \in \mathbb{N}_{\geq 1}$ and suppose~that $\mathcal{C} = (C_{i})_{i < m} \in \mathcal{P}(N)^{m}$ is a $k$-cover of~$N$. Let $w = (w_i)_{i < m}$ be a sequence of non-negative reals. Consider any family of standard Borel probability spaces $(\Omega_{j},\mu_{j})_{j \in N}$. Let $X \defeq \prod_{j \in N} \Omega_{j}$ and $\mathbb{P} \defeq \bigotimes_{j \in N} \mu_{j}$. Then, for every~$r \in \mathbb{R}_{> 0}$, \begin{displaymath}
	\alpha_{(X,d_{\mathcal{C},w}, {\mathbb P})} (r) \, \leq \, \exp \!\left( -\tfrac{kr^{2}}{8\Vert w \Vert_{2}^{2}} \right) .
\end{displaymath}\end{cor}

\begin{proof} This is an immediate consequence of Theorem~\ref{theorem:covering.concentration} and Proposition~\ref{proposition:concentration.function}. \end{proof}

\section{A classification of submeasures}\label{subsection:classification}

Our objective in this section is to give a quantitative classification of diffuse submeasures in terms of the asymptotics of weighted covering ratios (as detailed in Definition~\ref{D:class} and Theorem~\ref{theorem:submeasure.classification}). 
We start with recalling the notion of submeasure and various standard definitions concerning this concept. 

\begin{definition} Let $\mathcal{A}$ be a Boolean algebra. A function $\phi \colon \mathcal{A} \to \mathbb{R}$ is called a \emph{submeasure} if \begin{enumerate}
	\item[---$\,$] $\phi (0) = 0$,
	\item[---$\,$] $\phi$ is \emph{monotone}, that is, $\phi (A) \leq \phi (B)$ for all $A,B \in \mathcal{A}$ with $A \leq B$, and
	\item[---$\,$] $\phi$ is \emph{subadditive}, that is, $\phi (A \vee B) \leq \phi (A) + \phi (B)$ for all $A,B \in \mathcal{A}$.
\end{enumerate} Let $\phi \colon \mathcal{A} \to \mathbb{R}$ be a submeasure. Then $\phi$ is called a \emph{measure} if $\phi (A \vee B) = \phi (A) + \mu (B)$ for any two $A,B \in \mathcal{A}$ with $A \wedge B = 0$. The submeasure $\phi$ is called \emph{pathological} if there does not exist a non-zero measure $\mu \colon \mathcal{A} \to \mathbb{R}$ with $\mu \leq \phi$. Furthermore, $\phi$ is said to be \emph{diffuse} if, for every $\epsilon > 0$, there exists a finite subset $\mathcal{B} \subseteq \mathcal{A}$ such that $\bigvee \mathcal{B} = 1$ and $\phi (B) \leq \epsilon$ for each $B \in \mathcal{B}$. 
\end{definition}

Our classification of diffuse submeasures will be formulated in terms of the asymptotic behavior of a certain function associated with any such submeasure. The definition of the function relies on the notion of covering number (Definition~\ref{definition:covering.number.abstract}).

\begin{definition} Let $\mathcal{A}$ be a Boolean algebra and let $\phi \colon \mathcal{A} \to \mathbb{R}$ be a diffuse submeasure. For $\xi \in \mathbb{R}_{>0}$, let \begin{displaymath}
	\mathcal{A}_{\phi,\xi} \, \defeq \, \{ A \in \mathcal{A} \mid \phi(A) \leq \xi \} \, .
\end{displaymath} Define $h_{\phi} \colon \mathbb{R}_{>0} \to \mathbb{R}_{>0}$ by \begin{displaymath}
	\left. h_\phi(\xi) \, \defeq \, \tfrac{c_{\mathcal{A}} \left(\mathcal{A}_{\phi,\xi}\right)}{\xi} \, = \, \tfrac{1}{\xi} \sup \! \left\{ \tfrac{t_{\mathcal{A}}({\mathcal C})}{m} \, \right\vert m \in \mathbb{N}_{\geq 1}, \, \mathcal{C} \in \left( \mathcal{A}_{\phi, \xi} \right)^{m} \right\} .
\end{displaymath} \end{definition}

Clearly, for any diffuse submeasure~$\phi \colon \mathcal{A} \to \mathbb{R}$, the function $h_{\phi}$ is well defined, that is, $h_{\phi}$ only takes values in $\mathbb{R}_{>0}$. In the definition of $h_\phi$, the covering number $c_{\mathcal{A}} (\mathcal{A}_{\phi,\xi})$ measures how thickly $\mathcal{A}_{\phi,\xi}$ covers the unit $1$ of the Boolean algebra $\mathcal{A}$. This quantity is then divided by a normalizing factor $\xi$ to compensate for the fact that the elements of $\mathcal{A}_{\phi, \xi}$ become smaller as $\xi$ approaches $0$. (For an application in a different context of the covering number of the family 
$\mathcal{A}_{\phi,\xi}$, see \cite{MichaelHrusak}.)

By Lemma~\ref{lemma:covering.numbers}, we have the following reformulation in terms of uniform covers.

\begin{cor}\label{corollary:covering.numbers} Let $\mathcal{A}$ be a Boolean algebra and let $\phi \colon \mathcal{A} \to \mathbb{R}$ be a diffuse submeasure. Then, for every $\xi \in \mathbb{R}_{> 0}$, \begin{displaymath}
	\left. h_\phi(\xi) \, = \, \tfrac{1}{\xi}\sup \left\{ \tfrac{t_{\mathcal{A}}(\mathcal{C})}{m} \, \right\vert m \in \mathbb{N}_{\geq 1}, \, \mathcal{C} \in \left(\mathcal{A}_{\phi, \xi}\right)^{m} \! \textit{ uniform cover in } \mathcal{A} \right\} .
\end{displaymath} \end{cor}

Furthermore, an application of the Hahn--Banach extension theorem yields the subsequent description, where $\tfrac{1}{0} \defeq \infty$. For the proof of Proposition~\ref{proposition:hahn.banach} and for the statement of Theorem~\ref{theorem:christensen}, we fix one more piece of notation: given two sets $A \subseteq S$, let $\chi_{A} \colon S \to \{ 0,1 \}$ denote the corresponding \emph{indicator function} defined by $\chi_{A}(x) \defeq 1$ for all~$x \in A$ and $\chi_{A}(x) \defeq 0$ for all~$x \in S\setminus A$.

\begin{prop}\label{proposition:hahn.banach} Let $\mathcal{A}$ be a Boolean algebra and let $\phi \colon \mathcal{A} \to \mathbb{R}$ be a diffuse submeasure. For every $\xi \in \mathbb{R}_{>0}$, \begin{displaymath}
	\left. h_{\phi}(\xi) \, = \, \min \! \left\{ \tfrac{1}{\mu (1)} \, \right\vert \mu \colon \mathcal{A} \to \mathbb{R} \textit{ measure with } \mathcal{A}_{\phi,\xi} \subseteq \mathcal{A}_{\mu,\xi} \right\} .
\end{displaymath} \end{prop}

\begin{proof} Let $\xi \in \mathbb{R}_{> 0}$ be fixed.
	
($\leq$) Consider any measure $\mu \colon \mathcal{A} \to \mathbb{R}$ with $\mathcal{A}_{\phi,\xi} \subseteq \mathcal{A}_{\mu,\xi}$. If $\mathcal{C} = (C_{i})_{i < m} \in (\mathcal{A}_{\phi,\xi})^{m}$ for some $m \in \mathbb{N}_{\geq 1}$, then \begin{displaymath}
	t_{\mathcal{A}}(\mathcal{C}) \mu (1) \, \leq \, \sum\nolimits_{i < m} \mu (C_{i}) \, \leq \, m \xi
\end{displaymath} and thus $\tfrac{t_{\mathcal{A}}(\mathcal{C})}{m\xi} \leq \tfrac{1}{\mu (1)}$. Therefore, $h_{\phi}(\xi) \leq \tfrac{1}{\mu (1)}$ as desired.

($\geq$) Appealing to Stone's representation theorem for Boolean algebras~\cite{stone}, we may and will assume that $\mathcal{A}$ is a Boolean subalgebra of $\mathcal{P}(S)$ for some set $S$. Consider the seminorm $p \colon \ell^{\infty}(S) \to \mathbb{R}_{\geq 0}$ defined by \begin{align*}
	\left. p(f) \, \defeq \, \inf \!\left\{ \xi \sum\nolimits_{i < m} r_{i} \, \right\vert m \in \mathbb{N}, \, (r_{i})_{i<m} \in (\mathbb{R}_{\geq 0})^{m}\!, \, \right. (&B_{i})_{i<m} \in (\mathcal{A}_{\phi,\xi})^{m}\!, \\
	& \left. \vert f \vert \leq \sum\nolimits_{i<m}r_{i}\chi_{B_{i}} \right\}
\end{align*} for every $f \in \ell^{\infty}(S)$. Since $\mathbb{Q}$ is dense in $\mathbb{R}$, it follows that \begin{displaymath}
	\left. p(\chi_{S}) = \inf \!\left\{ \tfrac{\xi m}{k} \, \right\vert m,k \in \mathbb{N}_{\geq 1}, \, (B_{i})_{i<m} \in (\mathcal{A}_{\phi,\xi})^{m}, \, \chi_{S} \leq \tfrac{1}{k} \sum\nolimits_{i<m} \chi_{B_{i}} \right\} = h_{\phi}(\xi)^{-1} .
\end{displaymath} Concerning the linear functional $I \colon \mathbb{R}\chi_{S} \to \mathbb{R}, \, r \chi_{S} \mapsto rh_{\phi}(\xi)^{-1}$, we note that \begin{displaymath}
	\vert I(r\chi_{S}) \vert \, = \, \left\vert rh_{\phi}(\xi)^{-1} \right\vert \, = \, \vert r \vert h_{\phi}(\xi)^{-1} \, = \, \vert r \vert p(\chi_{S}) \, = \, p(r\chi_{S}) 
\end{displaymath} for all $r \in \mathbb{R}$. Therefore, the Hahn--Banach extension theorem asserts the existence of a linear functional $J \colon \ell^{\infty}(S) \to \mathbb{R}$ such that $J(\chi_{S}) = h_{\phi}(\xi)^{-1}$ and $\vert J(f) \vert \leq p(f)$ for every $f \in \ell^{\infty}(S)$. Let us define \begin{displaymath}
	\mu \colon \, \mathcal{A} \, \longrightarrow \, \mathbb{R}, \quad A \, \longmapsto \, J(\chi_{A})
\end{displaymath} and observe that $\mu (\emptyset) = 0$ and $\mu (S) = h_{\phi}(\xi)^{-1}$, and moreover $\mu (A \cup B) = \mu (A) + \mu(B)$ for any two disjoint $A,B \in \mathcal{A}$. Straightforward calculations now show that \begin{displaymath}
	\mu^{+} \colon \, \mathcal{A} \, \longrightarrow \, \mathbb{R}_{\geq 0}, \quad A \, \longmapsto \, \sup \{ \mu(B) \mid A \supseteq B \in \mathcal{A} \}
\end{displaymath} constitutes a measure (we refer to~\cite[Theorem~2.2.1(4)]{rao} for the details). Furthermore, since $p(\chi_{B}) \leq p(\chi_{A})$ for any $B \subseteq A \subseteq S$, it follows that \begin{displaymath}
	\mu^{+}(A) \, = \, \sup \{ J(\chi_{B}) \mid A \supseteq B \in \mathcal{A} \} \, \leq \, \sup \{ p(\chi_{B}) \mid A \supseteq B \in \mathcal{A} \} \, \leq \, p(\chi_{A})
\end{displaymath} for every $A \in \mathcal{A}$. Therefore, if $A \in \mathcal{A}_{\phi,\xi}$, then $\mu^{+}(A) \leq p(\chi_{A}) \leq \xi$, hence $A \in \mathcal{A}_{\mu^{+},\xi}$. Finally, let us observe that $\mu^{+}(S) \leq p(\chi_{S}) = h_{\phi}(\xi)^{-1} = \mu (S) \leq \mu^{+}(S)$, which means that $\mu^{+}(S) = h_{\phi}(\xi)^{-1}$. This completes the proof. \end{proof}

The asymptotic behavior of $h_\phi(\xi)$ as $\xi\to 0$ will be fundamental to our considerations. 
As it turns out, this behavior is quite rigid, as partly indicated by the following immediate consequence of points (i), (ii), and (iii) of  Theorem~\ref{theorem:submeasure.classification}, which is proved later. 
\begin{cor}
Let $\phi$ be a diffuse submeasure. Then the limit $\lim_{\xi\to 0} h_\phi(\xi)$, possibly infinite, exists. 
\end{cor}
\noindent Informed by the corollary above, in Definition~\ref{D:class}, we divide the class of diffuse 
submeasures according to their asymptotic behavior at $0$. Our 
choice of this division is further justified by its interactions with concentration of measure (see Theorem~\ref{theorem:covering.concentration.for.non.elliptic.submeasures} 
Example~\ref{example:berry.esseen})
and dynamics of $L_0$-groups (see Corollary~\ref{corollary:whirly.amenability}, and Proposition~\ref{proposition:reflecting.amenability}). 
We recall Landau's big $O$ notation: for two functions $f,g \colon \mathbb{R}_{>0} \to \mathbb{R}_{>0}$, \begin{displaymath}
	f(x) \, = \, O(g(x)) \, \text{ as } \, x \to 0 \quad \, :\Longleftrightarrow \, \quad \limsup\nolimits_{x \to 0} \tfrac{f(x)}{g(x)} \, < \, \infty .
\end{displaymath}

\begin{definition}\label{D:class} A diffuse submeasure $\phi$ is called \begin{enumerate}
	\item[---$\,$] \emph{elliptic} if $h_{\phi}(\xi) = O(\xi)$ as $\xi\to 0$,
	\item[---$\,$] \emph{hyperbolic} if $\frac{1}{h_{\phi}(\xi)} = O(\xi)$ as $\xi\to 0$,
	\item[---$\,$] \emph{parabolic} if $\phi$ is neither elliptic, nor hyperbolic.
\end{enumerate} \end{definition}

Evidently, the three notions defined above are mutually exclusive. We note that a diffuse submeasure $\phi$ is elliptic if and only if \begin{displaymath}
	\sup\nolimits_{\xi \in \mathbb{R}_{>0}} \tfrac{h_{\phi}(\xi)}{\xi} \, < \, \infty .
\end{displaymath} Clearly, the latter implies the former. Conversely, $\tfrac{h_{\phi}(\xi)}{\xi} \leq \tfrac{1}{\xi^{2}}$ for all $\xi \in \mathbb{R}_{>0}$, so that \begin{displaymath}
	\qquad \quad \limsup\nolimits_{\xi \to 0} \tfrac{h_{\phi}(\xi)}{\xi} \, < \, \infty \quad \Longrightarrow \quad \sup\nolimits_{\xi \in \mathbb{R}_{>0}} \tfrac{h_{\phi}(\xi)}{\xi} \, < \, \infty .
\end{displaymath}
	
The subsequent theorem is the main result of this section. It gives initial justification to the importance of the function introduced in Definition~\ref{D:class}.

\begin{thm}\label{theorem:submeasure.classification} Let $\phi$ be a diffuse submeasure. 
\begin{enumerate}
	\item[\textnormal{(i)}]  $\phi$ is hyperbolic if and only if it is pathological, in which case $\lim_{\xi\to 0} \xi h_{\phi}(\xi) = 1$.
	\item[\textnormal{(ii)}] If $\phi$ is parabolic, then $\lim_{\xi\to 0} h_{\phi}(\xi)$ exists and is finite.
	\item[\textnormal{(iii)}] If $\phi$ is elliptic, then $\lim_{\xi \to 0} h_{\phi}(\xi) = 0$. 
	\item[\textnormal{(iv)}] If $\phi$ is a measure, then $\lim_{\xi \to 0} h_{\phi}(\xi) = \tfrac{1}{\phi (1)}$, where $\frac{1}{0} =\infty$. 
\end{enumerate} \end{thm}

Note that the obvious estimate $h_\phi(\xi)\leq 1/\xi$ and (ii) and (iii) of Theorem~\ref{theorem:submeasure.classification} imply that $\phi$ is hyperbolic precisely when $h_\phi$ is unbounded. 
Also, it follows immediately from points (i), (ii), and (iv) that every non-zero diffuse measure is a parabolic submeasure. Of course, a zero measure is hyperbolic. The converses to (ii) and (iii) do not hold. A family of elliptic submeasures, the existence of which witnesses that the implication in (ii) cannot be reversed, is constructed in Example~\ref{example:berry.esseen}. For an example of a parabolic submeasure $\phi$ with $\lim_{\xi \to 0} h_{\phi}(\xi) = 0$, illustrating the failure of the converse to~(iii), see Example~\ref{example:easy}.

We remark here that (i) in Theorem~\ref{theorem:submeasure.classification} is essentially a reformulation of the following characterization of pathological submeasures due to Christensen~\cite{christensen}.

\begin{thm}[\cite{christensen}, Theorem~5]\label{theorem:christensen} Let $S$ be a set and $\mathcal{A}$ be a Boolean subalgebra of $\mathcal{P}(S)$. If $\phi \colon \mathcal{A} \to \mathbb{R}$ is a pathological submeasure, then for every $\xi \in \mathbb{R}_{>0}$ there exist $m \in \mathbb{N}_{\geq 1}$, $C_{0},\ldots,C_{m-1} \in \mathcal{A}_{\phi,\xi}$ and $a_{0},\ldots,a_{m-1} \in \mathbb{R}_{\geq 0}$ such that $\sum_{i < m} a_{i} = 1$ and $\sum_{i < m} a_{i}\chi_{C_{i}} \geq 1 - \xi$. \end{thm}

Christensen's Theorem~\ref{theorem:christensen} immediately entails the following corollary, which constitutes the essential ingredient in the proof of (i) in Theorem~\ref{theorem:submeasure.classification}.

\begin{cor}\label{corollary:christensen} Let $\mathcal{A}$ be a Boolean algebra. If $\phi \colon \mathcal{A} \to \mathbb{R}$ is a pathological submeasure, then for every $\xi \in \mathbb{R}_{>0}$ there exist $m \in \mathbb{N}_{\geq 1}$ and $\mathcal{C} \in \left(\mathcal{A}_{\phi,\xi}\right)^{m}$ such that $\tfrac{t_{\mathcal{A}}(\mathcal{C})}{m} \geq 1-\xi$. \end{cor}

\begin{proof} Again, thanks to Stone's representation theorem for Boolean algebras~\cite{stone}, we may and will assume that $\mathcal{A}$ is a Boolean subalgebra of $\mathcal{P}(S)$ for some set $S$. Consider any pathological submeasure $\phi \colon \mathcal{A} \to \mathbb{R}$ and let $\xi \in \mathbb{R}_{>0}$. Since $\mathbb{Q}$ is dense in $\mathbb{R}$, Christensen's Theorem~\ref{theorem:christensen} entails the existence of $m, p_{0},\ldots,p_{m-1},q \in \mathbb{N}_{\geq 1}$ as well as $C_{0},\ldots,C_{m-1} \in \mathcal{A}_{\phi,\xi}$ such that $\sum_{i < m} p_{i} = q$ and $\sum_{i < m} \tfrac{p_{i}}{q}\chi_{C_{i}} \geq 1 - \xi$. Let us consider the sequence $\mathcal{C}^{\ast} \defeq (C^{\ast}_{j})_{j < q} \in \mathcal{A}^{q}$ defined by setting \begin{displaymath}
	C^{\ast}_{k + \sum_{i < \ell} p_{i}} \, \defeq \, C_{i}
\end{displaymath} for any $\ell < m$ and $k < p_{\ell}$. Then \begin{displaymath}
	\sum\nolimits_{j < q} \chi_{C^{\ast}_{j}} \, = \, \sum\nolimits_{i < m} p_{i}\chi_{C_{i}} \, \geq \, (1-\xi )q \, ,
\end{displaymath} hence $\tfrac{t_{\mathcal{A}}(\mathcal{C}^{\ast})}{q} \geq 1-\xi$ as desired. \end{proof}

The proof of (ii) in Theorem~\ref{theorem:submeasure.classification} relies on the following general convergence result.

\begin{lem}\label{lemma:convergence} Let $f \colon \mathbb{R}_{> 0} \to \mathbb{R}_{\geq 0}$. If $\sup_{\xi \in \mathbb{R}_{> 0}} \tfrac{f(\xi)}{\xi} < \infty$ and, for all $\xi, \zeta \in \mathbb{R}_{>0}$, 
\begin{displaymath}
	 f(\xi+\zeta) \, \geq \, f(\xi) + f(\zeta)- f(\xi) f(\zeta) ,
\end{displaymath} then $\lim_{\zeta\to 0} \frac{f(\zeta)}{\zeta}$ exists and is finite. \end{lem}

\begin{proof} Let $\left. M \defeq \sup \! \left\{ \tfrac{f(\xi)}{\xi} \, \right\vert \xi \in \mathbb{R}_{> 0} \right\}$. For a start, we prove that \begin{equation}\label{E:prep}
	\forall \xi \in \mathbb{R}_{> 0} \ \forall k \in \mathbb{N}_{\geq 1} \colon \quad \tfrac{f(k\xi)}{k\xi} \, \geq \, \tfrac{f(\xi)}{\xi} - M^2k\xi .
\end{equation} Let $\xi \in \mathbb{R}_{>0}$. We prove the inequality by induction over $k \in \mathbb{N}_{\geq 1}$. Clearly, if $k = 1$, then the desired statement holds trivially. Furthermore, if $\tfrac{f(k\xi)}{k\xi} \, \geq \, \tfrac{f(\xi)}{\xi} - M^{2}k\xi$ for some $k \in \mathbb{N}_{\geq 1}$, then \begin{align*}
	f((k+1)\xi) \, &\geq \, f(k\xi) + f(\xi) - f(k\xi) f(\xi) \, \geq \, f(k\xi) + f(\xi) - M^2\xi^{2}k \\
		&\geq \, kf(\xi) - M^2k^{2}\xi^{2} + f(\xi) - M^2\xi^{2}k \, = \, (k+1)f(\xi) - M^2\xi^{2}(k^{2}+k) \\
		&\geq \, \, (k+1)f(\xi) - M^2\xi^{2}(k+1)^{2} ,
\end{align*} that is, $\tfrac{f((k+1)\xi)}{(k+1)\xi} \geq \tfrac{f(\xi)}{\xi} - M^{2}\xi (k+1)$. This completes our induction and therefore proves~\eqref{E:prep}.
	
Let $L \defeq \limsup_{\zeta\to 0} \tfrac{f(\zeta)}{\zeta}$. Clearly, $L \leq M < \infty$. We prove that $\tfrac{f(\xi)}{\xi} \to L$ as $\xi \to 0$. Of course, this holds trivially if $L=0$. So, assume that $L > 0$. Fix $\epsilon \in (0,L)$. It will suffice to show that 
\begin{equation}\label{limit}
	 \xi \in \left( 0, \tfrac{\epsilon}{2M^{2}+1} \right) \quad \Longrightarrow \quad \tfrac{f(\xi)}{\xi} \, > \, (1-\epsilon )(L-\epsilon ) .
\end{equation} 
By definition of $L$, there exists $\zeta \in (0,\xi)$ such that $\tfrac{f(\zeta)}{\zeta} > L - \tfrac{\epsilon}{2}$ and $\frac{\lfloor \xi / \zeta\rfloor}{\lfloor \xi/\zeta\rfloor + 1} > 1-\epsilon$. Let $k \defeq \lfloor \xi/\zeta \rfloor$, so that $\xi = k\zeta + r$ for some $r \in [0,\zeta )$. Note that $\tfrac{k\zeta}{k\zeta + r} \geq \tfrac{k}{k+1} > 1-\epsilon$. It follows that \begin{align*}
	\tfrac{f(\xi)}{\xi} \, &\geq \, \tfrac{1}{k\zeta + r}\left( f(k\zeta) + f(r) - f(k\zeta)f(r)\right) \, \geq \, \tfrac{k\zeta}{k\zeta + r}\left(\tfrac{f(k\zeta)}{k\zeta} - \tfrac{f(k\zeta)}{k\zeta}f(r)\right) \\
		&\stackrel{\eqref{E:prep}}{\geq} \, \tfrac{k\zeta}{k\zeta + r}\left( \left(\tfrac{f(\zeta)}{\zeta} - M^2k\zeta \right) - M^{2}r \right) \, = \, \tfrac{k\zeta}{k\zeta + r}\left( \tfrac{f(\zeta)}{\zeta} - M^2\xi \right) \, > \, (1-\epsilon )(L-\epsilon) .
\end{align*} This proves~\eqref{limit} and thus completes our proof. \end{proof}

\begin{proof}[Proof of Theorem~\ref{theorem:submeasure.classification}] Let $\phi$ be a diffuse submeasure on a Boolean algebra $\mathcal{A}$.
	
(i) For a start, let us note that $h_{\phi}(\xi) \leq \tfrac{1}{\xi}$ for every $\xi \in \mathbb{R}_{>0}$. Now, if~$\phi$ is pathological, then Corollary~\ref{corollary:christensen} yields that \begin{displaymath}
	h_{\phi}(\xi) \, \geq \, \tfrac{1-\xi}{\xi}
\end{displaymath}for all $\xi \in \mathbb{R}_{>0}$, which~therefore entails that $\xi h_{\phi}(\xi) \longrightarrow 1$ as $\xi \to 0$. The latter condition clearly implies that $\phi$ is hyperbolic. Furthermore, if $\phi$ is hyperbolic, then $h_{\phi}$ must be unbounded. It only remains to argue that, if $h_{\phi}$ is unbounded, then $\phi$ will be pathological. To this end, let us assume that $\phi$ is non-pathological, that is, there exists a measure $\mu \colon \mathcal{A} \to \mathbb{R}$ with $0 \ne \mu \leq \phi$. Then Proposition~\ref{proposition:hahn.banach} entails that $h_{\phi}(\xi) \leq \tfrac{1}{\mu(1)}$ for all $\xi \in \mathbb{R}_{>0}$. In particular, $h_\phi$ is bounded. This proves~(i).
	
(ii) Suppose that $\phi$ is parabolic. Since $\phi$ is not hyperbolic, $h_{\phi}$ is bounded by~(i). Consider the function 
\begin{displaymath}
	f \colon \, \mathbb{R}_{>0} \, \longrightarrow \, \mathbb{R}_{>0}, \qquad \xi \, \longmapsto \, \xi h_{\phi}(\xi) .
\end{displaymath} We prove that, for all $\xi, \zeta \in \mathbb{R}_{>0}$, 
\begin{equation}\label{E:ineq}
	 f(\xi+\zeta) \, \geq \, f(\xi) + f(\zeta)- f(\xi)\cdot f(\zeta) \, .
\end{equation} 
For this purpose, fix $\xi, \zeta, \epsilon \in \mathbb{R}_{>0}$. Due to Lemma~\ref{lemma:covering.numbers}, there exist 
$k_{\xi},k_{\zeta},m_{\xi},m_{\zeta} \in \mathbb{N}_{\geq 1}$, some uniform $k_{\xi}$-cover $\mathcal{C}_{\xi} = (C_{\xi, i})_{i < m_{\xi}} \in (\mathcal{A}_{\phi,\xi})^{m_{\xi}}$ in $\mathcal{A}$, as well as some uniform $k_{\zeta}$-cover $\mathcal{C}_{\zeta} = (C_{\zeta, j})_{j < m_{\zeta}} \in (\mathcal{A}_{\phi,\zeta})^{m_{\zeta}}$ in $\mathcal{A}$ such that \begin{equation}\label{hypothesis}
	(1-\epsilon)f(\xi) \, \leq \, \tfrac{k_{\xi}}{m_{\xi}} \, \leq \, f(\xi), \qquad \qquad (1-\epsilon) f(\zeta) \, \leq \, \tfrac{k_{\zeta}}{m_{\zeta}} \, \leq \, f(\zeta) .
\end{equation} Put $m \defeq m_{\xi}\cdot m_{\zeta}$ and consider \begin{displaymath}
	\mathcal{S} \, \defeq \, \langle \{ C_{\xi, i} \mid i < m_{\xi} \} \cup \{ C_{\zeta, j} \mid j < m_{\zeta} \} \rangle_{\mathcal{A}} \, \in \, \Pi (\mathcal{A}) \, .
\end{displaymath} Furthermore, let us define a sequence $\mathcal{B} \defeq (B_{\ell})_{\ell < m} \in \mathcal{A}^{m}$ by setting, 
for each pair $(i,j) \in \{ 0,\ldots, m_{\xi}-1\} \times \{ 0,\ldots, m_{\zeta} -1 \}$, 
 \begin{displaymath}
	B_{i\cdot m_{\zeta} + j} \, \defeq \, C_{\xi , i} \vee C_{\zeta , j} \, . 
\end{displaymath} 
As $\phi$ is a submeasure, $\mathcal{B}$ belongs to $(\mathcal{A}_{\phi,\xi + \zeta})^{m}$. Since $\mathcal{C}_{\xi}$ is a uniform $k_{\xi}$-cover in $\mathcal{A}$ and $\mathcal{C}_{\zeta}$ is a uniform $k_{\zeta}$-cover in $\mathcal{A}$, it follows that, for each $S \in \mathcal{S}$, \begin{align*}
	\vert \{ \ell < m \mid S \leq B_{\ell} \} \vert \, & = \, \vert \{ (i,j) \mid i < m_{\xi} , \, j < m_{\zeta} , \, S \leq C_{\xi , i} \vee C_{\zeta , j} \} \vert \\
	& = \, \vert \{ i < m_{\xi} \mid S \leq C_{\xi , i} \} \vert \cdot m_{\zeta} + m_{\xi} \cdot \vert \{ j < m_{\zeta} \mid S \leq C_{\zeta , j} \} \vert \\
	& \qquad \qquad \qquad \quad - \vert \{ i < m_{\xi} \mid S \leq C_{\xi , i} \} \vert \cdot \vert \{ j < m_{\zeta} \mid S \leq C_{\zeta , j} \} \vert \\
	& = \, k_{\xi} \cdot m_{\zeta} + m_{\xi} \cdot k_{\zeta} - k_{\xi} \cdot k_{\zeta} \, .
\end{align*} By Remark~\ref{remark:covering.number.partition.refinement}, as $\langle \mathcal{B} \rangle_{\mathcal{A}} \preceq \mathcal{S}$, this shows that $t_{\mathcal{A}}(\mathcal{B}) = k_{\xi} \cdot m_{\zeta} + m_{\xi} \cdot k_{\zeta} - k_{\xi} \cdot k_{\zeta}$. Thus, appealing to~\eqref{hypothesis}, we conclude that
\begin{displaymath}
	f(\xi+\zeta) \, \geq \, \tfrac{k_{\xi} \cdot m_{\zeta} + m_{\xi} \cdot k_{\zeta} - k_{\xi} \cdot k_{\zeta}}{m_{\xi}\cdot m_{\zeta}} \, = \, \tfrac{k_{\xi}}{m_{\xi}} + \tfrac{k_{\zeta}}{m_{\zeta}} - \tfrac{k_{\xi}}{m_{\xi}}\cdot \tfrac{k_{\zeta}}{m_{\zeta}} \, \geq \, (1-\epsilon)(f(\xi) + f(\zeta))-f(\xi) \cdot f(\zeta).
\end{displaymath} This proves~\eqref{E:ineq}. Since the function $h_{\phi}$ is bounded, assertion~\eqref{E:ineq} and Lemma~\ref{lemma:convergence} together imply the desired conclusion. 

(iii) is obvious. 

(iv) Of course, if $\phi = 0$, then $\phi$ is pathological, thus hyperbolic by~(i), and therefore \begin{displaymath}
	\lim\nolimits_{\xi \to 0} h_{\phi}(\xi) \, = \, \lim\nolimits_{\xi \to 0} \tfrac{1}{\xi} \, = \, \infty .
\end{displaymath} Suppose now that $\phi$ is a non-zero measure. In particular, $\phi$ is non-pathological. This implies, by~(ii) and~(iii), the existence of the limit $a \defeq \lim_{\xi \to 0} h_{\phi}(\xi) \in \mathbb{R}$. We will prove that $a = \tfrac{1}{\phi(1)}$. By Proposition~\ref{proposition:hahn.banach}, we have $h_{\phi}(\xi) \leq \tfrac{1}{\phi(1)}$ for every $\xi \in \mathbb{R}_{> 0}$. Hence, $a \leq \tfrac{1}{\phi(1)}$. To prove the reverse inequality, we will show that \begin{equation}\label{diffusion}
	\forall \theta \in \mathbb{R}_{> 1} \ \forall n \in \mathbb{N}_{\geq 1} \colon \qquad h_{\phi}\!\left( \tfrac{\theta \phi (1)}{n} \right) \, \geq \, \tfrac{1}{\theta \phi (1)} .
\end{equation} To this end, let $\theta \in \mathbb{R}_{> 1}$ and $n \in \mathbb{N}_{\geq 1}$. Since $\phi$ is diffuse, $\mathcal{A}$ admits a finite partition of the unity, $\mathcal{B}$, such that $\phi (B) \leq (\theta -1)\tfrac{\phi (1)}{n}$ for every $B \in \mathcal{B}$. Note that, if $\mathcal{B}' \subseteq \mathcal{B}$ and $\phi \! \left(\bigvee \mathcal{B}' \right) < \tfrac{\phi (1)}{n}$, then \begin{displaymath}
	\phi \!\left( B \vee \bigvee \mathcal{B}' \right) \, \leq \, \phi (B) + \phi \! \left(\bigvee \mathcal{B}' \right) \, < \, (\theta - 1)\tfrac{\phi(1)}{n} + \tfrac{\phi (1)}{n} \, = \, \theta \tfrac{\phi (1)}{n}
\end{displaymath} for any $B \in \mathcal{B}\setminus \mathcal{B}'$. Using this observation, one can select a sequence of pairwise disjoint subsets $\mathcal{B}_{0},\ldots,\mathcal{B}_{n-1} \subseteq \mathcal{B}$ such that $\mathcal{B} = \bigcup_{i < n} \mathcal{B}_{i}$ and $\phi (\bigvee \mathcal{B}_{i}) < \theta \tfrac{\phi (1)}{n}$ for each $i<n$. Consider the sequence $\mathcal{C} \defeq (C_{i})_{i < n} \in \mathcal{A}^{n}$ given by $C_{i} \defeq \bigvee \mathcal{B}_{i}$ for each $i < n$. As $\phi (C_{i}) < \theta \tfrac{\phi (1)}{n}$~for~all~$i < n$, \begin{displaymath}
	h_{\phi}\!\left( \tfrac{\theta \phi (1)}{n} \right) \, \geq \, \tfrac{t_{\mathcal{A}}(\mathcal{C})}{n(\theta \phi(1)/n)} \, = \, \tfrac{1}{\theta \phi (1)} .
\end{displaymath} This proves~\eqref{diffusion}. From~\eqref{diffusion}, we now infer that \begin{displaymath}
	a \, = \, \lim\nolimits_{n \to \infty} h_{\phi}\!\left( \tfrac{\theta \phi (1)}{n} \right) \, \geq \, \tfrac{1}{\theta \phi (1)}
\end{displaymath} for every $\theta \in \mathbb{R}_{> 1}$. Thus, $a \geq \tfrac{1}{\phi (1)}$ as desired. \end{proof}

Below, we describe an example of a diffuse submeasure that shows that 
the converse to the implication in  (iii) of Theorem~\ref{theorem:submeasure.classification} fails to hold. 
It is a parabolic submeasure that is far from being a measure.

\begin{exmpl}\label{example:easy}
There exists a diffuse submeasure $\phi$ such that 
\begin{enumerate}
\item[(i)] $\phi$ is parabolic, and 

\item[(ii)] $\lim_{\xi\to 0} h_\phi(\xi) \, = \,0$.
\end{enumerate}

The submeasure $\phi$ will be defined on the Boolean algebra $\mathcal{A}$ of all 
clopen subsets of the topological product space $X \defeq \prod_{n=0}^\infty K_n$ for an appropriate choice of 
positive integers $(K_{n})_{n \in \mathbb{N}}$. To guarantee that $\phi$ is not elliptic, as implied by point (i), we need to make sure that 
\[
\limsup\nolimits_{\xi\to 0} \tfrac{h_\phi(\xi)}{\xi} \, = \, \infty, 
\]
which will follow if we find a sequence $({\mathcal B}_n)_{n\in \mathbb{N}}$ of partitions 
of $X$ into clopen sets and a sequence $(\xi_n)_{n \in \mathbb{N}}$ of positive real numbers such that 
\begin{equation}\label{E:guarparab}
\begin{split}
&\,\phi(A) \, \leq \, \xi_n \, \hbox{ for all } n \in \mathbb{N} \hbox{ and } A\in {\mathcal B}_n, \hbox{ and }\\
&\lim\nolimits_{n\to \infty} |{\mathcal B}_n|\, \xi_n^2 \, = \, 0.
\end{split}
\end{equation}
Note that the above condition implies that $\lim_{n\to\infty} \xi_n = 0$ and, in turn, that $\phi$ will be diffuse. 
To furthermore guarantee point~(ii) and, in turn, prove the remaining part of point~(i), by Proposition~\ref{proposition:hahn.banach} and the convergence established in Theorem~\ref{theorem:submeasure.classification}, it will suffice to find a sequence $(\mu_n)_{n \geq 1}$ of measures 
on $\mathcal{A}$ such that, for the sequence $(\xi_n)_{n \in \mathbb{N}}$ as above, 
\begin{equation}\label{E:notm}
\begin{split}
&\,\hbox{for all } A \in \mathcal{A} \hbox{ and } n \geq 1, \hbox{ if } \phi(A)\, \leq \, \xi_n, \hbox{ then } \mu_n(A) \, \leq \, \xi_n, \hbox{ and }\\
&\lim\nolimits_{n\to \infty} \mu_n(X) \, = \, \infty.
\end{split}
\end{equation}

We take a sequence $(M_n)_{n\geq 1}$ of natural numbers such that, for each $n\geq 1$, 
\begin{equation}\label{E:mmm}
n|M_n,\; 1 \, \leq \, \tfrac{\sqrt{M_n}}{n} \, \leq \, \tfrac{\sqrt{M_{n+1}}}{n+1}, \hbox{ and }\lim\nolimits_{n\to\infty} \tfrac{\sqrt{M_n}}{n} \, = \, \infty.
\end{equation}
So, for example, letting $M_n \defeq n^3$ for each $n \geq 1$ will work. We set 
\[
K_0 \, \defeq \, 1\;\hbox{ and }\; K_n \, \defeq \, \tfrac{M_n}{n} \; \hbox{ for each } \,n\geq 1
\]
in the above definition of $X$. We also set 
\[
\xi_0 \, \defeq \, 1 \;\hbox{ and }\; \xi_n \, \defeq \, \tfrac{1}{\sqrt{M_n}} \,\hbox{ for each }\,n\geq 1.
\]

For $n \in \mathbb{N}$ and $i< K_n$, let 
\[
[i,n] \, \defeq \, \{ x\in X\mid x_n \, = \, i\}.
\]
Furthermore, consider the set of finite sequences \begin{displaymath}
	\left. S \, \defeq \, \! \left\{ (i_{k},n_{k})_{k=1}^{p} \, \right\vert p \in \mathbb{N}, \, n_{1},\ldots,n_{p} \in \mathbb{N}, \, i_{1} < K_{n_{1}},\ldots, i_{p} < K_{n_{p}} \right\} .
\end{displaymath} We define $\phi \colon \mathcal{A} \to \mathbb{R}$ by setting \begin{equation}\label{E:deph}
	\phi(A) \, \defeq \, \inf \left\{ \sum\nolimits_{k=1}^{p} \xi_{n_k} \left\vert \, (i_{k},n_{k})_{k=1}^{p} \in S, \, A \, \subseteq \, \bigcup\nolimits_{k=1}^p [i_k, n_k] \right\} \right.
\end{equation} for every $A \in \mathcal{A}$. Clearly, $\phi \colon \mathcal{A} \to \mathbb{R}$ is a submeasure and $\phi (X) = 1$. We have the following claim that asserts that the infimum in \eqref{E:deph} is attained. 

\textit{Claim.} Let $A \in \mathcal{A}$. There exists $(i_{k},n_{k})_{k=1}^{p} \in S$ such that 
\[
A \, \subseteq \, \bigcup\nolimits_{k=1}^p [i_k, n_k]\, \hbox{ and }\,\phi(A) \, = \, \sum\nolimits_{k=1}^p \xi_{n_k}. 
\]

\noindent {\em Proof of Claim.} 
Since $A$ is clopen, there exists a natural number $N$ such that, for $x,y\in X$, if $x_{n} = y_{n}$ for all $n\leq N$, 
then $x\in A$ if and only if $y\in A$. Fix such an $N$ for the remainder of the proof of the claim. 

If $\phi(A)\geq 1$, it suffices to take $p=1$ and $n_{1} = i_{1} = 0$. 
So, let us assume $\phi(A)<1$. It will suffice to show that, for every sequence $(i_{k},n_{k})_{k=1}^{p} \in S$, if 
\begin{equation}\label{E:sttr}
A \, \subseteq \, \bigcup\nolimits_{k=1}^p [i_k, n_k]\, \hbox{ and }\, \sum\nolimits_{k=1}^p \xi_{n_k} \, < \, 1, 
\end{equation}
then
\[
A \, \subseteq \, \bigcup \{ [i_k, n_k] \mid k \in \{ 1,\ldots,p \}, \, n_{k} \leq N \} . 
\]

Assume, towards a contradiction, that there is a sequence $(i_{k},n_{k})_{k=1}^{p} \in S$ for which the above implication fails.  
By the choice of $N$, we can find $s\in \prod_{n=0}^N K_n$ such~that 
\begin{equation}\label{E:slal}
B \, \subseteq \, A\; \hbox{ and }\; 
B \cap \bigcup \{ [i_k, n_k] \mid k \in \{ 1,\ldots,p \}, \, n_{k} \leq N \} \, = \, \emptyset,
\end{equation} where $\left. B \defeq \! \left\{ x\in \prod\nolimits_{n=0}^\infty K_n \, \right\vert {x\!\! \upharpoonright_{N}} = s \right\}$.
Note that there is $n>N$ such that 
\begin{equation}\label{E:quant}
\forall i<K_n\ \exists k\in \{ 1,\ldots , p\} \colon \quad i \, = \, i_k\,\hbox{ and }\,n \, = \, n_k .
\end{equation}
Otherwise, we can produce $y\in \prod_{n=0}^\infty K_n$ such that 
\[
 {y\!\! \upharpoonright_{N}} \, = \, s\;\hbox{ and }\; y \, \not\in \, \bigcup \{ [i_k, n_k] \mid k \in \{ 1,\ldots,p \}, \, n_{k} > N \},
\]
which, by \eqref{E:slal}, implies that $y\in A$ and $y\not\in  \bigcup_{k=1}^p [i_k, n_k]$, leading to a contradiction with \eqref{E:sttr}. 
So, fix $n>N$ such that \eqref{E:quant} holds. Then, by \eqref{E:mmm}, we have 
\[
\sum\nolimits_{k=1}^p \xi_{n_k} \, \geq \, K_n \xi_{n} \, = \, \tfrac{M_n}{n} \tfrac{1}{\sqrt{M_n}} \, = \, \tfrac{\sqrt{M_n}}{n} \, \geq \, 1,
\]
contradicting \eqref{E:sttr}. The claim follows. \hfill $\qed_{\text{Claim}}$
\smallskip

We claim that $\phi$ satisfies conditions \eqref{E:guarparab} and \eqref{E:notm}, and therefore (i) and (ii). 
To see \eqref{E:guarparab}, for each $n \in \mathbb{N}$, note that 
\[
{\mathcal B}_n \, \defeq \, \{ [i,n]\mid i<K_n\}
\]
is a finite partition of $X$ into elements of $\mathcal{A}$ and that $\phi([i,n])\leq \xi_n$ for all $i \in K_{n}$, and moreover
\[
\lim\nolimits_{n\to\infty} |{\mathcal B}_n|\, \xi_n^2 \, = \, \lim\nolimits_{n\to\infty} \tfrac{M_n}{n} \!\left(\tfrac{1}{\sqrt{M_n}}\right)^2 \, = \, \lim\nolimits_{n\to \infty} \tfrac{1}{n} \, = \, 0. 
\]

To see \eqref{E:notm}, for each $n \geq 1$, we consider the product measure \begin{displaymath}
	\mu_{n} \, \defeq \, \bigotimes\nolimits_{j=1}^{\infty} \nu_{n,j} ,
\end{displaymath} where for $j \ne n$, $\nu_{n,j}$ 
is the measure on $K_j$ assigning weight $\frac{1}{K_j} = \frac{j}{M_j}$ to each singleton $\{ i\}$ for $i<K_j$, while $\nu_{n,n}$ 
is the measure on $K_n$ assigning weight $\frac{1}{\sqrt{M_n}}$ to each singleton $\{ i\}$ for $i<K_n$. So, for $j\not=n$, $\nu_{n,j}$ 
is a probability measure, while the total mass of $\nu_{n,n}$ is equal to 
\[
K_n \tfrac{1}{\sqrt{M_n}} \, = \, \tfrac{M_n}{n}\tfrac{1}{\sqrt{M_n}} \, = \, \tfrac{\sqrt{M_n}}{n}.
\]
It follows that 
\begin{equation}\label{E:tot}
\mu_n(X) \, = \, \tfrac{\sqrt{M_n}}{n},
\end{equation}
so $\lim_{n\to\infty} \mu_n(X)=\infty$.

It only remains to see that, for each $n\geq 1$ and each $A \in \mathcal{A}$, 
if $\phi(A)\leq\xi_n$, then $\mu_n(A) \leq \xi_n$; we will actually show that 
\begin{equation}\label{E:phmu}
\phi(A)\, \leq \, \xi_n\quad \Longrightarrow\quad \mu_n(A) \, \leq \, \phi(A).
\end{equation} To this end, fix $n\geq 1$, which will remain fixed for the remainder of the example. First, we point out that since $\mu_n$ is a measure, it follows from~\eqref{E:tot} that, for all $j\geq 1$ and $i<K_j$, 
\begin{equation}\label{E:meva}
\mu_n([i,j]) \, = \, \tfrac{\tfrac{\sqrt{M_n}}{n}}{K_j} \, = \, \tfrac{\sqrt{M_n}}{M_j}\tfrac{j}{n}. 
\end{equation} Now, let us call a sequence $(i_k, n_k)_{k=1}^p \in S$ {\em tight} if 
\[
\phi\!\left(\bigcup\nolimits_{k=1}^p [i_k, n_k]\right) \, = \, \sum\nolimits_{k=1}^p \xi_{n_k}. 
\]
We claim that, for every tight sequence $(i_k, n_k)_{k=1}^p \in S$,  
\begin{equation}\label{E:esti}
\phi\!\left(\bigcup\nolimits_{k=1}^p [i_k, n_k]\right) \, \leq \, \xi_n \quad \Longrightarrow \quad \mu_n\!\left(\bigcup\nolimits_{k=1}^p [i_k, n_k]\right) \, \leq \, \phi\!\left(\bigcup\nolimits_{k=1}^p [i_k, n_k]\right). 
\end{equation}
We prove~\eqref{E:esti} by induction on $p$, with the usual convention for $p=0$: the sequence is empty, it is tight, 
and the implication~\eqref{E:esti} holds since $\bigcup_{k=1}^p [i_k, n_k] = \emptyset$. So, fix $p\geq 0$ and assume that~\eqref{E:esti} holds 
for $p$; we prove it for $p+1$. Let $(i_k, n_k)_{k=1}^{p+1} \in S$ be a tight sequence. 
Set 
\[
C \, \defeq \, \bigcup\nolimits_{k=1}^{p+1} [i_k, n_k]\; \hbox{ and }\; B \, \defeq \, \bigcup\nolimits_{k=1}^p [i_k, n_k]. 
\]
Suppose that $\phi (C) \leq \xi_{n}$. We observe that $(i_k, n_k)_{k=1}^p$ is tight since otherwise, $p>0$ and $\phi(B) < \sum_{k=1}^p \xi_{n_k}$, so 
\[
\phi(C) \, \leq \, \phi(B) + \phi([i_{p+1}, n_{p+1}]) \, \leq \, \phi(B)+\xi_{n_{p+1}} \, < \, \sum\nolimits_{k=1}^{p+1} \xi_{n_k},
\]
a contradiction. 
Thus, by inductive assumption, it follows that 
\begin{equation}\label{E:finca}
\begin{split}
\phi(C) \, &= \, \sum\nolimits_{k=1}^p \xi_{n_k} + \xi_{n_{p+1}} \, = \, \phi(B) + \xi_{n_{p+1}}\\
&\geq \, \mu_n(B) + \xi_{n_{p+1}} \, = \,
\mu_n(B) + \tfrac{1}{\sqrt{M_{n_{p+1}}}}.
\end{split}
\end{equation}
Note that since 
\[
\xi_n \, \geq \, \phi(C) \, = \, \sum\nolimits_{k=1}^{p+1} \xi_{n_k},
\]
we have $n_{p+1}\geq n$. 
Using $n_{p+1}\geq n$ and~\eqref{E:mmm}, we see that 
\[
\tfrac{1}{\sqrt{M_{n_{p+1}}}} \, \geq \, \tfrac{\sqrt{M_n}}{M_{n_{p+1}}} \tfrac{n_{p+1}}{n}.
\]
Thus, continuing with~\eqref{E:finca} and using~\eqref{E:meva}, we get 
\[
\phi(C) \, \geq \, \mu_n(B) +  \tfrac{\sqrt{M_n}}{M_{n_{p+1}}} \tfrac{n_{p+1}}{n} \, = \, \mu_n(B) + \mu_n([i_{p+1}, n_{p+1}]) \, \geq \, \mu_n(C).
\]
The inductive argument for~\eqref{E:esti} is completed.

Now, we prove~\eqref{E:phmu}. Fix any $A \in \mathcal{A}$ with $\phi(A)\leq\xi_n$. By our Claim above, there exists a sequence $(i_{k},n_{k})_{k=1}^{p} \in S$ such that $A \subseteq \bigcup\nolimits_{k=1}^p [i_k, n_k]$ and $\phi(A) = \sum\nolimits_{k=1}^p \xi_{n_k}$. 
It is clear that this sequence is tight. Therefore, by~\eqref{E:esti}, we have  
\[
\phi(A) \, = \, \phi\!\left(\bigcup\nolimits_{k=1}^p [i_k, n_k]\right) \, \geq \, \mu_n\! \left(\bigcup\nolimits_{k=1}^p [i_k, n_k]\right) \, \geq \, \mu_n(A),
\]
as required. 
\end{exmpl}

\section{L{\'e}vy nets from submeasures} \label{subsection:levy.nets}

In this section, we combine the quantitative classification from Section~\ref{subsection:classification} with the results of Section~\ref{section:covering.concentration} to exhibit new examples of L\'evy nets: we prove that any non-elliptic submeasure gives rise to a L\'evy net (Theorem~\ref{theorem:covering.concentration.for.non.elliptic.submeasures}). For this purpose, let us introduce the following family of pseudo-metrics, the definition of which may be compared with Definition~\ref{definition:covering.metric}

\begin{definition}\label{definition:submeasure.metric} Let $\mathcal{A}$ be a Boolean algebra and let $\phi \colon \mathcal{A} \to \mathbb{R}$ be a submeasure. For $\mathcal{B} \in \Pi ({\mathcal A})$ and a set $\Omega$, we define a pseudo-metric \begin{displaymath}
	\delta_{\phi, \mathcal{B}} \colon \, \Omega^{\mathcal{B}} \times \Omega^{\mathcal{B}} \, \longrightarrow \, \mathbb{R}_{\geq 0}
\end{displaymath} by setting 
\begin{displaymath}
	\delta_{\phi, \mathcal{B}} (x,y) \, \defeq \, \phi \! \left( \bigvee \{ B \in \mathcal{B} \mid x(B) \ne y(B) \} \right)  .
\end{displaymath} Given a standard Borel probability space $(\Omega,\mu)$, we let \begin{displaymath}
	\mathcal{X}(\Omega,\mu,\mathcal{B},\phi) \, \defeq \, \left(\Omega^{\mathcal{B}},\delta_{\phi, \mathcal{B}},\mu^{\otimes \mathcal{B}}\right) .
\end{displaymath} \end{definition}

Let $\lambda$ denote the Lebesgue measure on the standard Borel space $\mathbb{I} \defeq [0,1] \subseteq \mathbb{R}$.

\begin{remark}\label{remark:submeasure.metric} Let $\mathcal{A}$ be a Boolean algebra. Consider a submeasure $\phi \colon \mathcal{A} \to \mathbb{R}$ and let $\mathcal{B} \in \Pi ({\mathcal A})$. If $(\Omega_{0},\mu_{0})$ and $(\Omega_{1},\mu_{1})$ are two standard Borel probability spaces and $\pi \colon \Omega_{0} \to \Omega_{1}$ is a measurable map with $\pi_{\ast}(\mu_{0}) = \mu_{1}$,~then 
\begin{displaymath}
	\widehat{\pi} \colon \, \left( \Omega_{0}^{\mathcal{B}}, \delta_{\phi, \mathcal{B}} \right) \, \longrightarrow \, 
	\left( \Omega_{1}^{\mathcal{B}}, \delta_{\phi, \mathcal{B}} \right), \qquad x \, \longmapsto \, \pi \circ x
\end{displaymath} 
is a $1$-Lipschitz map and $\widehat{\pi}_{\ast}\bigl(\mu_{0}^{\otimes \mathcal{B}}\bigr) = \mu_{1}^{\otimes \mathcal{B}}$, thus Remark~\ref{remark:concentration}(2) asserts that 
\begin{displaymath}
	\alpha_{\mathcal{X}(\Omega_{1},\mu_{1},\mathcal{B},\phi)} \, \leq \, \alpha_{\mathcal{X}(\Omega_{0},\mu_{0},\mathcal{B},\phi)} .
\end{displaymath} In particular, since for every standard Borel probability space $(\Omega,\mu)$ there exists a measurable map $\psi \colon \mathbb{I} \to \Omega$ with $\psi_{\ast}(\lambda) = \mu$ (for instance, see~\cite[Lemma~4.2]{ShioyaBook}), this entails that \begin{displaymath}
	\alpha_{\mathcal{X}(\Omega,\mu,\mathcal{B},\phi)} \, \leq \, \alpha_{\mathcal{X}(\mathbb{I},\lambda,\mathcal{B},\phi)} .
\end{displaymath} \end{remark}

\begin{definition}\label{definition:covering.concentration} Let $\mathcal{A}$ be a Boolean algebra. We say that a submeasure $\phi \colon \mathcal{A} \to \mathbb{R}$ has \emph{covering concentration} if, for every $\epsilon \in \mathbb{R}_{>0}$, 
there exists ${\mathcal C}\in \Pi({\mathcal A})$ such that 
\begin{displaymath}
	\sup \{ \alpha_{\mathcal{X}(\mathbb{I},\lambda,\mathcal{B},\phi)}(\epsilon) \mid \mathcal{B} \in \Pi 
	(\mathcal{A}), \,  \mathcal{C} \preceq \mathcal{B} \} \, \leq \, \epsilon .
\end{displaymath} \end{definition}

\begin{remark}\label{remark:covering.concentration} Let $\mathcal{A}$ be a Boolean algebra. It follows from Remark~\ref{remark:concentration}(1) that a submeasure $\phi \colon \mathcal{A} \to \mathbb{R}$ has covering concentration if and only if there exists a sequence $(\mathcal{C}_{\ell})_{\ell \in \mathbb{N}} \in \Pi(\mathcal{A})^{\mathbb{N}}$ such that, for every $\epsilon \in \mathbb{R}_{>0}$, 
\begin{displaymath}
	\sup \{ \alpha_{\mathcal{X}(\mathbb{I},\lambda,\mathcal{B},\phi)}(\epsilon) \mid \mathcal{B} \in \Pi ({\mathcal A}), \, \mathcal{C}_{\ell}  \preceq \mathcal{B} \} \, \longrightarrow \, 0 \; \hbox{ as } \; \ell \to \infty.
\end{displaymath} 
\end{remark}

For clarification, let us point out the following.

\begin{lem} Every submeasure having covering concentration is diffuse. \end{lem}

\begin{proof} Let $\mathcal{A}$ be a Boolean algebra. Suppose that $\phi \colon \mathcal{A} \to \mathbb{R}$ is a submeasure with covering concentration. Let $\epsilon \in \mathbb{R}_{>0}$. By assumption, there exists~$\mathcal{B} \in \Pi (\mathcal{A})$ with $\alpha_{\mathcal{X}(\mathbb{I},\lambda,\mathcal{B},\phi)}(\epsilon) < \tfrac{1}{2}$. We claim that $\phi (B) < \epsilon$ for each~$B \in \mathcal{B}$. To see this, let~$B \in \mathcal{B}$. Note that $\lambda^{\otimes \mathcal{B}}(T) = \tfrac{1}{2}$ for the measurable subset \begin{displaymath}
	\left. T \, \defeq \, \left\{ x \in \mathbb{I}^{\mathcal{B}} \, \right\vert x(B) \leq \tfrac{1}{2} \right\} \, \subseteq \, \mathbb{I}^{\mathcal{B}} .
\end{displaymath} Now, if $\phi (B) \geq \epsilon$, then $B_{\delta_{\phi,\mathcal{B}}}(T,\epsilon) = T$, which implies that $\lambda^{\otimes \mathcal{B}}(B_{\delta_{\phi,\mathcal{B}}}(T,\epsilon)) = \tfrac{1}{2}$, so $\alpha_{\mathcal{X}(\mathbb{I},\lambda,\mathcal{B},\phi)}(\epsilon) \geq \tfrac{1}{2}$, contradicting our choice of $\mathcal{B}$. Hence, $\phi (B) < \epsilon$ as desired. \end{proof}

By force of Corollary~\ref{corollary:covering.concentration}, we arrive at our third main result.

\begin{thm}\label{theorem:covering.concentration.for.non.elliptic.submeasures} 
Every hyperbolic or parabolic submeasure has covering concentration. \end{thm}

\begin{proof} Let $\mathcal{A}$ be a Boolean algebra and consider any non-elliptic diffuse submeasure $\phi \colon \mathcal{A} \to \mathbb{R}$. Let 
$\epsilon \in \mathbb{R}_{>0}$. Fix any $r \in \mathbb{R}_{\geq 0}$ with $\exp \! \left( -\tfrac{r\epsilon^{2}}{16} \right) \leq \epsilon$. By our assumption, there exists some $\xi \in \mathbb{R}_{>0}$ such that \begin{equation}\label{E:hin}
	\tfrac{h_{\phi}(\xi)}{\xi} \, \geq \, r \, .
\end{equation} Now, we find $m\in {\mathbb N}_{>0}$ and a sequence $\mathcal{C} = (C_{i})_{i < m} \in \left( \mathcal{A}_{\phi,\xi} \right)^{m}$ such that 
\begin{equation}\label{E:ttwo}
	\tfrac{t_{\mathcal{A}}({\mathcal C})}{m\xi^{2}} \, \geq \, \tfrac{h_{\phi}(\xi)}{2\xi} \, . 
\end{equation} Let $\mathcal{B} \in \Pi ({\mathcal A})$ with $\langle \mathcal{C} \rangle_{\mathcal{A}} \preceq \mathcal{B}$. By Remark~\ref{remark:covering.number.concrete}(2), the sequence $\mathcal{C}_{\mathcal{B}} \defeq (C_{\mathcal{B},i})_{i < m}$, defined by \begin{displaymath}
	C_{\mathcal{B},i} \, \defeq \, \{ B \in \mathcal{B} \mid B \leq C_{i} \}
\end{displaymath} for all $i < m$, constitutes a $t_{\mathcal{A}}(\mathcal{C})$-cover of the set $\mathcal{B}$. Furthermore, note that, by subadditivity of the submeasure $\phi$, we have $\delta_{\phi, \mathcal{B}} \leq d_{\mathcal{C}_{\mathcal{B}},(\phi (C_{i}))_{i < m}}$ on ${\mathbb I}^{\mathcal B}$. (For the definition of the latter pseudo-metric, see Definition~\ref{definition:covering.metric}, page~\pageref{definition:covering.metric}.) Consequently, combined with~\eqref{E:ttwo} and~\eqref{E:hin}, Corollary~\ref{corollary:covering.concentration} asserts that \begin{align*}
	\sup \{ \alpha_{\mathcal{X}(\mathbb{I},\lambda,\mathcal{B},\phi)}(&\epsilon) \mid \mathcal{B} \in \Pi ({\mathcal A}), \, \langle \mathcal{C} \rangle_{\mathcal{A}} \preceq \mathcal{B} \} \, \leq \, \exp \!\left( -\tfrac{t_{\mathcal{A}}(\mathcal{C})\epsilon^{2}}{8\sum_{i < m} \phi (C_{i})^{2}} \right) \\
	& \leq \, \exp \!\left( -\tfrac{t_{\mathcal{A}}(\mathcal{C})\epsilon^{2}}{8m \xi^{2}} \right) \, \leq \, \exp \!\left( -\tfrac{h_{\phi}(\xi)\epsilon^{2}}{16\xi} \right) \, \leq \, \exp \! \left( -\tfrac{r\epsilon^{2}}{16} \right) \, \leq \, \epsilon . \qedhere
\end{align*} \end{proof}

We conclude this section by exhibiting a family of elliptic submeasures without covering concentration: in fact, we construct a diffuse submeasure $\phi$ 
that does not have concentration and is such that $h_\phi(\xi)/\xi$ does not converge to $0$ fast, as $\xi \to 0$. 
The example involves an application of the Berry--Esseen theorem~\cite{berry,esseen} (see also~\cite[Chapter~XVI.5]{FellerBook2}). 
A precise statement is given below.

\begin{exmpl}\label{example:berry.esseen} Fix any function $\theta \colon {\mathbb R}_{>0}\to {\mathbb R}_{>0}$ such that $\lim_{\xi\to 0} \theta(\xi) = 0$. There exists a diffuse submeasure $\phi$ such that \begin{enumerate}
	\item[(i)] $\phi$ does not have covering concentration, and 
	\item[(ii)] $\limsup_{\xi\to 0} \frac{ h_\phi(\xi)/\xi}{\theta(\xi)} = \infty$. 
\end{enumerate}

We split our description of the example and the arguments associated with them into several parts. 
	
\smallskip 

{\it A general claim.} 
Assume we are given positive integers $M_1, \dots, M_k$. Define 
\begin{equation}\label{E:ttl} 
	T \, \defeq \, M_1\times \cdots \times M_k \;\;\hbox{ and }\;\; T_{\leq} \, \defeq \, \bigcup\nolimits_{i=0}^k M_1\times \cdots\times M_i.
\end{equation} 
For each $y\in 2^T$, we define an extension $\bar{y} \in 2^{T_\leq}$ recursively as follows: let 
\begin{equation}\label{E:ybar}
\bar{y}(t) \, \defeq \, y(t),\,\hbox{ for } t \in T; 
\end{equation}
and if $i \in \{ 0,\ldots,k-1 \}$ and $\bar{y}(s)$ is defined for all $s\in T_\leq$ with $|s|\geq i+1$, then, for $t\in T_\leq$ with $|t|=i$, put 
\begin{equation}\label{E:ybarr}
	\bar{y}(t) \, \defeq \, \begin{cases} 
	\, 0 & \text{if } |\{ j< M_{i+1}\mid \bar{y}(tj) =1\}| \leq \frac{M_{i+1}}{2} , \\
	\, 1 & \text{otherwise.}
	\end{cases}
\end{equation} 
Let
\begin{equation}\label{E:seta}
	\left. A \, \defeq \, \! \left\{ y\in 2^T \, \right\vert \bar{y}(\emptyset) =0 \right\} . 
\end{equation}

Assume, additionally, we are given positive real numbers $d_1, \dots, d_k$. 
For each $y\in 2^T$, define another extension $\hat{y} \in 2^{T_\leq}$ recursively as follows: we let 
\begin{equation}\label{E:yhat}
\hat{y}(t) \, \defeq \, y(t),\,\hbox{ for }t\in T; 
\end{equation}
and if $i \in \{ 0,\ldots,k-1 \}$ and 
$\hat{y}(s)$ is defined for all $s \in T_\leq$ with $|s|\geq i+1$, then, for $t\in T_\leq$ with $|t|=i$, put 
\begin{equation}\label{E:yhatt}
	\hat{y}(t) \, \defeq \, \begin{cases} 
	\, 0 &\text{if } |\{ j< M_{i+1}\mid \hat{y}(tj) =1\}| <\frac{M_{i+1}}{2}+d_{i+1}, \\
	\, 1 &\text{otherwise.}
	\end{cases}
\end{equation} 
Let
\begin{equation}\label{E:setb}
	\left. B \, \defeq \, \! \left\{ y\in 2^T \, \right\vert \hat{y}(\emptyset) =0 \right\}. 
\end{equation}

Finally, define a binary relation ${\sim} \subseteq 2^{T} \times 2^{T}$ as follows. For $x,y\in 2^T$, we write $x\sim y$ precisely if there exists a subset 
$S\subseteq T_\leq\setminus \{ \emptyset\}$ such that 
\begin{equation}\label{E:relsim}
\begin{split}
	\forall i < k \ \forall s\in T \colon \ &\Bigl( |s|=i \ \Longrightarrow \  |S\cap \{ sj\mid j < M_{i+1}\}| < d_{i+1}  \Bigr)\\
	 &\hbox { and }\\
	\forall t\in T \colon \ &\Bigl( x(t)\not= y(t) \ \Longrightarrow \, \bigl(\exists i \in \{ 0,\ldots,k\} \colon \ {t\!\!\upharpoonright_{\{ 1,\ldots,i\}} } \in S\bigr)\Bigr).
\end{split} 
\end{equation}
The relation $\sim$ is symmetric and reflexive. 

We point out that the two operations $2^{T} \ni y\mapsto \bar{y} \in 2^{T_\leq}$ and $2^{T} \ni y\mapsto \hat{y} \in 2^{T_\leq}$, the sets $A$, $B$, and the relation $\sim$ defined above depend on the sequences 
$M_1, \dots , M_k$ and $d_1, \dots , d_k$. We do not reflect this dependence in our notation as we do not want to burden the symbols with subscripts. 
However, the reader should keep this dependence in mind. 

\begin{claim}\label{Cl:msds} 
\begin{enumerate}
\item[$(i)$] $\, |A| \, \geq \,  2^{|T|-1}$. 

\item[$(ii)$] $\left. \left\{ y\in 2^T \, \right\vert \exists x\in A\colon \, x\sim y \right\} \, \subseteq \, B$. 
\end{enumerate}
\end{claim} 

\noindent {\it Proof of Claim~\ref{Cl:msds}.} To see (i), consider the bijection 
\[
2^T \longrightarrow \, 2^{T}, \quad y \, \longmapsto \, 1-y
\]
where, for each $s\in T$, $(1-y)(s) \defeq 1-y(s)$. Now (i) is an immediate consequence (with $t=\emptyset$) of 
the implication
\[
\bar{y}(t) \, = \, 1 \quad \Longrightarrow \quad \overline{1-y}(t) \, = \, 0,
\]
which holds for all $t\in T_\leq$ and is proved by induction on $k-|t|$. 

The inclusion in point (ii) is proved by induction on $k$, that is, on the length of the sequence $(M_i)_{i=1}^k$. 

Assume first that $k=1$. In this case, we can identify $T$ with $M_1$. We have 
\[
\left. A \, = \, \left\{ x\in 2^{M_1} \, \right\vert \lvert \{ j<M_1\,\vert \, x(j) = 1\} \rvert \leq \tfrac{M_1}{2}\right\}  
\]
and 
\[
\left. B \, = \, \left\{ y\in 2^{M_1} \, \right\vert \lvert \{ j<M_1\,\vert \, y(j) = 1\} \rvert < \tfrac{M_1}{2}+d_1\right\}.
\]
On the other hand, if $x\sim y$, then there is $S\subseteq M_1$ such that 
\[
\left\{ j<M_1\,\vert\, x(j)\not= y(j)\right\}\, \subseteq \, S\ \hbox{ and }\ |S|<d_1, 
\]
and (ii) for $k=1$ follows immediately. 

We show now the inductive step: given sequence $(M_i)_{i=1}^k$ and $(d_i)_{i=1}^k$ with $k>1$, we consider the sequences $(M_i)_{i=2}^k$ 
and $(d_i)_{i=2}^k$ and, assuming 
the inclusion in point (ii) holds for them, we prove the inclusion for $(M_i)_{i=1}^k$ and $(d_i)_{i=1}^k$. Define 
\begin{displaymath}
	T^0 \, \defeq \, M_2\times \cdots \times M_k , \qquad \quad T^0_{\leq} \, \defeq \, \bigcup\nolimits_{i=1}^k M_2\times \cdots\times M_i.
\end{displaymath} 
Let the operations $x\mapsto \overline{x}^0$, $x\mapsto \widehat{x}^0$, the sets 
$A^0$, $B^0$, and the relation $\sim_0$ be defined in the manner analogous to $x\mapsto\overline{x}$, $x\mapsto\widehat{x}$, 
$A$, $B$, and $\sim$, but for the sequences $(M_i)_{i=2}^k$ and $(d_i)_{i=2}^k$ instead of $(M_i)_{i=1}^k$ and $(d_i)_{i=1}^k$. 
By induction, we assume that 
\begin{equation}\label{E:abin}
\left. \left\{ y\in 2^{T_0} \,\right\vert \exists x\in A^0\colon \, x\sim_0 y \right\} \, \subseteq \, B^0.
\end{equation}
For $x\in 2^T$ and $j <M_1$, let $x_j\in 2^{T^0}$ be defined by 
\[
x_j(s) \, \defeq \, x(js).
\]
We note two essentially tautologous equations, justification of which we leave to the reader:
\begin{equation}\label{E:taut}
\overline{x}(j) \, = \, \overline{x_j}^0(\emptyset)\ \hbox{ and }\ \widehat{x}(j) \, = \, \widehat{x_j}^0(\emptyset). 
\end{equation}
The following three implications hold for all $x,y\in 2^T$: 
\begin{align}
&x\in A \, \Longrightarrow \, \Bigl( \left|\left\{ j<M_1\left\vert \, x_j\in A^0\right\}\right| \right. \, \geq \, \tfrac{M_1}{2}\Bigr), \label{E:aaa} \\
&\Bigl( \left|\left\{ j<M_1\left\vert \, y_j\in B^0\right\}\right| \right. \, > \, \tfrac{M_1}{2}- d_1 \Bigr) \, \Longrightarrow \, y\in B, \label{E:bbb}\\
&x\sim y \, \Longrightarrow \,
\Bigl( \left|\left\{ j<M_1\left\vert \, x_j\sim_0 y_j\right\}\right| \right. \, > \, M_1-d_1\Bigr) .\label{E:ccc}
\end{align}
Implication \eqref{E:aaa} follows from the definitions of $A$ and $A^0$ and from the first equation of \eqref{E:taut}. Similarly, implication 
\eqref{E:bbb} follows from the definitions of $B$ and $B^0$ and from the second equation of \eqref{E:taut}. 
To get \eqref{E:ccc}, observe that if $S\subseteq T_\leq\setminus \{ \emptyset\}$ witnesses that 
$x\sim y$, then, for $j<M_1$, if the one-element sequence whose only entry is $j$ is not in $S$, then the set 
\[
\left. \left\{ s\in T^0_\leq \, \right\vert js\in S \right\}
\]
witnesses that $x_j\sim_0y_j$; therefore, \eqref{E:ccc} follows since the set $S$ satisfies the first clause of \eqref{E:relsim} (for $i=0$). 

Now we aim to prove $y\in B$ assuming that $x\in A$ and $x\sim y$. By \eqref{E:aaa} and \eqref{E:ccc}, 
\[
\left\lvert \left\{ j<M_1 \left\vert \, x_j\in A^0\hbox{ and } x_j\sim_0y_j\right\} \right\rvert \right. \, > \, \tfrac{M_1}{2}-d_1. 
\]
Applying our inductive assumption \eqref{E:abin} to this inequality, we get 
\[
\left|\left\{ j<M_1\left\vert \, y_j\in B^0 \right\}\right| \right. \, > \, \tfrac{M_1}{2}-d_1, 
\]
which yields $y\in B$ by \eqref{E:bbb}, as required. Therefore, the claim is proved. \hfill $\qed_{\text{Claim}\!~\ref{Cl:msds}}$

\smallskip

{\it A consequence of the Berry--Esseen theorem.} As a result of the Berry--Esseen theorem, there exists an increasing function 
$C\colon [1/2, 1)\to {\mathbb R}_{>0}$ with the following property: for all~$a,b \in \mathbb{R}_{> 0}$ with $b\leq a$ and $a+b=1$, for every 
$d \in \mathbb{R}_{\geq 0}$, and for every finite sequence $X_1, \dots , X_n$ of independent random variables such that 
\begin{displaymath}
	\forall i \in \{ 1,\ldots,n \} \colon \qquad \mathbb{P}[X_i=0] \, = \, a, \quad \mathbb{P}[X_i=1] \, = \, b ,
\end{displaymath} we have \begin{equation}\label{E:BE}
	\mathbb{P}\! \left[\tfrac{1}{\sqrt{n}} \sum\nolimits_{i=1}^n (X_i-b) <d \right] \, < \, \tfrac{1}{2} + C(a)\!\left(d + \tfrac{1}{\sqrt{n}}\right). 
\end{equation} It follows from~\eqref{E:BE} that, if $a \in \left[ \tfrac{1}{2}, \tfrac{3}{4}\right]$ and $\delta \in \mathbb{R}_{\geq 0}$, then \begin{equation}\label{E:XF}
	\mathbb{P} \! \left[ \vert \{ i \in \{ 1,\ldots,n \} \mid X_{i} = 1 \} \vert < \tfrac{n}{2} +\delta\sqrt{n} \, \right] - \tfrac{1}{2} \, < \, K \! \left(\delta + \left(a-\tfrac{1}{2}\right)\!\sqrt{n} + \tfrac{1}{\sqrt{n}}\right) ,
\end{equation} where $K \defeq \max \! \left\{ C\!\left(\tfrac{3}{4}\right) \! , 1 \right\}$. Indeed, assuming that $a\leq \tfrac{3}{4}$ and substituting \begin{displaymath}
	d \, \defeq \, \delta + \left(a- \tfrac{1}{2}\right)\!\sqrt{n}
\end{displaymath} in~\eqref{E:BE}, we obtain \begin{equation}\label{E:Xin}
	\mathbb{P} \! \left[\tfrac{1}{\sqrt{n}} \sum\nolimits_{i=1}^n (X_i-b) <  \delta + \left(a- \tfrac{1}{2}\right)\!\sqrt{n}\, \right] \, < \, \tfrac{1}{2} + K \! \left(\delta + \left(a- \tfrac{1}{2}\right)\!\sqrt{n} + \tfrac{1}{\sqrt{n}}\right).
\end{equation} A quick calculation shows that the condition \begin{displaymath}
	\tfrac{1}{\sqrt{n}} \sum\nolimits_{i=1}^n (X_i-b) \, < \,  \delta + \left(a-\tfrac{1}{2}\right)\!\sqrt{n} 
\end{displaymath} is equivalent to \begin{displaymath}
	\sum\nolimits_{i=1}^n X_i \, < \, \tfrac{n}{2} + \delta \sqrt{n} ,
\end{displaymath} which, in turn, is equivalent to the condition \begin{displaymath}
	\vert \{ i \in \{ 1,\ldots,n \} \mid X_{i} = 1 \} \vert \, < \, \tfrac{n}{2} +\delta\sqrt{n} .
\end{displaymath} Putting the above equivalences together with~\eqref{E:Xin}, we arrive at~\eqref{E:XF}.
	
{\it Defining a submeasure.} For any sequence of positive integers $M = (M_i)_{i \in \mathbb{N}_{\geq 1}}$ and any sequence of positive reals $w=(w_i)_{i \in \mathbb{N}}$, we define the submeasure \begin{displaymath}
	\phi_{M,w} \colon \, \mathcal{P}\!\left( \prod\nolimits_{i \in \mathbb{N}_{\geq 1}} M_{i} \right) \! \, \longrightarrow \, \mathbb{R}
\end{displaymath} by setting \begin{displaymath}
	\phi_{M,w}(A) \, \defeq \, \inf \left\{ \sum\nolimits_{s \in S} w_{|s|} \left\vert \, S \subseteq \bigcup\nolimits_{i \in \mathbb{N}_{\geq 1}} \prod\nolimits_{j=1}^{i-1} M_j , \ A\subseteq \bigcup\nolimits_{s\in S} [s]_{M} \right\} \right. ,
\end{displaymath} where $\! \left. [s]_{M} \defeq \left\{ x \in \prod\nolimits_{i \in \mathbb{N}_{\geq 1}} M_{i} \, \right\vert {x \!\! \upharpoonright_{\{ 1,\ldots, i-1 \}}} = s \right\}$ for any $s \in \prod\nolimits_{j=1}^{i-1} M_j$ with $i \in \mathbb{N}_{\geq 1}$.

\smallskip

{\it Choosing the parameters.} To determine the submeasure $\phi_{M,w}$ we only need to specify the two sequences $M$ and $w$. 
We pick $M$ and $w$ in agreement with the following four conditions: 
\begin{align}
	&\lim\nolimits_{i\to\infty} w_i \, = \, 0; \label{E:wze} \\
	&\lim\nolimits_{i \to\infty} w_i^2 \, M_1\cdots  M_i\, \theta(w_i) \, = \, 0; \label{E:wmt}\\
	&1\leq w_0\,\ \hbox{ and }\ \frac{1}{M_i}\leq w_i\,\hbox{ for all } \, i\in{\mathbb N}_{\geq 1}; \label{E:mwi} 
\end{align} and there exists a sequence $(\epsilon_k)_{k \in \mathbb{N}}$ of positive reals such that \begin{align}
	\epsilon_0 \, < \, \tfrac{1}{4} \ \hbox{ and } \ \epsilon_{k-1} \, &= \, K\! \left( \tfrac{1}{w_k\sqrt{M_k}} + \sqrt{M_k}\, \epsilon_k + \tfrac{1}{\sqrt{M_k}}\right) \; \hbox{for all $k \in \mathbb{N}_{\geq 1}$}. \label{E:wse}
\end{align} 
Note that the equation in~\eqref{E:wse} determines $(\epsilon_k)_{k\in \mathbb{N}}$ from $\epsilon_0$. So, given $\epsilon_0$, we can define the whole sequence 
$(\epsilon_k)_{k\in \mathbb{N}}$; the only issue in question is whether $\epsilon_k>0$ for all $k\in \mathbb{N}_{\geq 1}$. 
	
The sequences $M$ and $w$ are constructed as follows. The constant $K\geq 1$ was defined above. Let $w_{0} \defeq 1$. Since $\lim_{\xi\to 0} \theta(\xi)=0$, for each $i\in {\mathbb N}_{\geq 1}$, we find a positive real $w_i$ so that \begin{equation}\label{E:one}
	w_i \, \leq \, 2^{-i} \ \hbox{ and } \ 2^{2i+5} M_1\cdots M_{i-1} K^{i} \sqrt{\theta(w_i)} \, < \, 1,
\end{equation} with the usual convention that the product $M_1\cdots M_{i-1}$ equals $1$ if $i = 1$. Then, using \eqref{E:one} and the fact that $1\leq \tfrac{\sqrt{m+1}}{\sqrt{m}} \leq 2$ for all $m \in \mathbb{N}_{\geq 1}$, we find a positive integer $M_i$ so that \begin{equation}\label{E:MM}
	2^{i}w_i\sqrt{M_1\cdots M_{i-1}} \sqrt{\theta(w_i)} \, \leq \, \tfrac{1}{\sqrt{M_i}} \, \leq \, 2^{i+1}w_i\sqrt{M_1\cdots M_{i-1}} \sqrt{\theta(w_i)}. 
\end{equation}

Let us check that the chosen sequences $w$ and $M$ meet the four conditions stated above. Evidently, \eqref{E:wze} is satisfied due to the first 
assertion of~\eqref{E:one}. Also, the first inequality in~\eqref{E:MM} gives~\eqref{E:wmt}. To get \eqref{E:mwi}, note that the first inequality in \eqref{E:mwi} 
is obvious since $w_0=1$. To see the second inequality of~\eqref{E:mwi}, observe that, since $K\geq 1$, \eqref{E:one} implies 
\[
2^{i+1}\sqrt{w_i}\sqrt{M_1\cdots M_{i-1}} \sqrt{\theta(w_i)} \, < \, 1\, \hbox{ for all } \, i\in {\mathbb N}_{\geq 1}.
\]
This inequality, when applied to the second inequality in \eqref{E:MM}, gives 
\[
 \tfrac{1}{\sqrt{M_i}} \, \leq \, \sqrt{w_i} \, \hbox{ for all } \, i\in {\mathbb N}_{\geq 1}, 
\]
which immediately yields the remainder of \eqref{E:mwi}. 
The second inequality in~\eqref{E:MM}, together with \eqref{E:one}, guarantees that, for each 
$k \in \mathbb{N}$, the series 
\begin{displaymath}
	\epsilon_k \, \defeq \, \sum\nolimits_{i=k+1}^\infty \left( \left(\tfrac{1}{w_i}+1\right) \tfrac{\sqrt{M_{k+1}\cdots M_{i-1}}}{\sqrt{M_i}} K^{i-k}\right)
\end{displaymath} 
converges, and that $\epsilon_0 < \tfrac{1}{4}$, again with the usual convention that the product $M_{k+1}\cdots M_{i-1}$ is equal to $1$ if $i=k+1$. 
It is clear that $\epsilon_k>0$ for each $k \in \mathbb{N}$. It is also easy to check that the sequence $(\epsilon_k)_{k \in \mathbb{N}}$ satisfies 
the equation in~\eqref{E:wse}.
	
Let $M$ and $w$ be sequences as above. Consider the Boolean algebra $\mathcal{A}$ of all clopen subsets of topological product space $Z \defeq \prod\nolimits_{k \in \mathbb{N}_{\geq 1}} M_k$, and note that the submeasure \begin{displaymath}
	\phi \, \defeq \, \phi_{M,w}\!\!\upharpoonright_{\mathcal{A}} \colon \, \mathcal{A} \, \longrightarrow \, \mathbb{R} 
\end{displaymath} is diffuse due to~\eqref{E:wze}. Additionally, for each $k \in \mathbb{N}_{\geq 1}$, let 
\begin{equation}\label{E:del}
	\delta_k \, \defeq \, \tfrac{1}{w_k\sqrt{M_k}} 
\end{equation} and consider the partition \begin{displaymath}
	\mathcal{B}_{k} \, \defeq \, \{ [s]_{M} \mid s \in M_{1} \times \cdots \times M_{k} \} \, \in \, \Pi (\mathcal{A}) .
\end{displaymath}

\smallskip

{\it Checking (i), that is, lack of covering concentration.} Denote by $\mu$ the normalized counting measure on~$2 = \{ 0,1 \}$. 
We will prove that, for each $k \in \mathbb{N}_{\geq 1}$, 
\begin{equation}\label{lack}
	\alpha_{\mathcal{X}(2,\mu,\mathcal{B}_{k},\phi)}(1) \, \geq \, \tfrac{1}{4} .
\end{equation} 
By Remark~\ref{remark:submeasure.metric} and $\{ \mathcal{B}_{k} \mid k \in \mathbb{N}_{\geq 1} \}$ being cofinal in $(\Pi(\mathcal{A}),{\preceq})$, 
this will imply that $\phi \colon \mathcal{A} \to \mathbb{R}$ does not have covering concentration. 
Inequality \eqref{lack} will be witnessed by the sets $A'$ and $B'$ defined below. 
The idea for the definitions of these two sets comes from \cite[Theorem~4.2]{FarahSolecki}.

To prove~\eqref{lack}, let $k \in \mathbb{N}_{\geq 1}$. Let $T$ and $T_{\leq}$ be as in \eqref{E:ttl} for $M_1, \dots , M_k$ chosen as above. 
For $\delta_1, \dots, \delta_k$ chosen as in \eqref{E:del}, set 
\[
d_i = \delta_i\sqrt{M_i},\;\hbox{ for } i\in \{ 1, \dots , k\}, 
\]
and define the operations 
\[
2^T\ni y \, \longmapsto \, \bar{y} \in 2^{T_\leq} \ \hbox{ and }\ 2^T\ni y \, \longmapsto \, \hat{y} \in 2^{T_\leq}
\]
as in \eqref{E:ybar}, \eqref{E:ybarr}, \eqref{E:yhat}, and \eqref{E:yhatt} for the sequences $(M_i)_{i=1}^k$ and $(d_i)_{i=1}^k$ described above. 
Furthermore, let $A,\, B$, and $\sim$ be as in \eqref{E:seta}, \eqref{E:setb}, and 
\eqref{E:relsim}.
Recall that, by Claim~\ref{Cl:msds}(i), 
\begin{equation}\label{E:half}
	|A| \, \geq \,  2^{|T|-1}. 
\end{equation}

%Since $\delta_i\sqrt{M_i} = \tfrac{1}{w_i} \geq 1$ for all $i\in {\mathbb N}_{\geq 1}$, we have $A\subseteq B$. 
Let $\mathcal{B} \defeq \mathcal{B}_{k}$ and consider the bijection $f \colon T \to \mathcal{B}, \, s \mapsto [s]_{M}$.
We will prove that 
\begin{equation}\label{E:hull}
	\left\{ y \in 2^{\mathcal{B}} \left\vert \, \exists x \in A' \colon \, \delta_{\phi, \mathcal{B}}(x,y) <1 \right\} \, \subseteq \, B', \right.
\end{equation} 
where 
\begin{displaymath}
	\left. A' \, \defeq \, \! \left\{ x \in 2^{\mathcal{B}} \, \right\vert x \circ f \in A \right\} \ \hbox{ and } \ \left. B' \, \defeq \, \! \left\{ x \in 2^{\mathcal{B}} \, \right\vert x \circ f \in B \right\}. 
\end{displaymath} 
We will also prove that \begin{equation}\label{E:upp}
	|B| \, \leq \, \tfrac{3}{4}\,2^{|T|}. 
\end{equation} Formula~\eqref{E:hull} together with \eqref{E:upp} and \eqref{E:half} will show~\eqref{lack}. 
	
We start with showing \eqref{E:hull}. Recall first that, by Claim~\ref{Cl:msds}(ii), 
\begin{equation}\label{E:hu}
	\left. \left\{ y\in 2^T \,\right\vert \exists x\in A\colon \, x\sim y \right\} \, \subseteq \, B.
\end{equation} 

We claim that 
\begin{equation}\label{E:lto}
	\forall x,y \in 2^{\mathcal{B}} \colon \quad d_{\mathcal{B},\phi}(x,y) < 1 \ \Longrightarrow \ (x \circ f) \sim (y \circ f) . 
\end{equation} 
To see this, let $x,y \in 2^{\mathcal{B}}$ be such that 
\begin{equation}\label{E:gaga}
d_{\mathcal{B},\phi}(x,y) \, < \, 1.
\end{equation}
Set 
\[
T' \, \defeq \, \{  t\in T \mid x([t]_{M})\not= y([t]_{M}) \}, 
\]
and note that \eqref{E:gaga} implies that there exists $S\subseteq  \bigcup_{k=0}^\infty M_1\times\cdots \times M_k$ such that 
\begin{equation}\label{E:cloin}
\bigcup\nolimits_{t\in T'} [t]_M \, \subseteq \, \bigcup\nolimits_{s\in S} [s]_M 
\end{equation}
and 
\begin{equation}\label{E:sumlo}
	\sum\nolimits_{s\in S} w_{|s|} \, < \, 1. 
\end{equation} 
Now, \eqref{E:cloin} implies that 
\begin{equation}\label{E:clcon}
\begin{split}
&\forall t\in T'\, \exists i \in \{ 0,\ldots,k\} \colon \ {t\!\!\upharpoonright_{\{ 1,\ldots,i\}} } \in S\; \hbox{ or }\\ 
&\exists k\in {\mathbb N}\,  \exists s\in M_1\times\cdots \times M_k\colon \left( sj\in S \hbox{ for all }j\in M_{k+1}\right). 
\end{split} 
\end{equation}
The second clause of \eqref{E:clcon} gives $k\in {\mathbb N}$ such that 
\[
M_{k+1} w_{k+1} \, \leq \, \sum\nolimits_{s\in S} w_{|s|}. 
\]
Since, by \eqref{E:mwi}, $1\leq M_{k+1}w_{k+1}$, the above inequality contradicts \eqref{E:sumlo}. Thus, the first clause of \eqref{E:clcon} holds. 
In particular, we have that $S\subseteq  T_{\leq}$ since $T'\subseteq T$. Furthermore, $\emptyset\in S$ together with $w_0\geq 1$ from \eqref{E:mwi} 
would also contradict \eqref{E:sumlo}. Thus, $\emptyset\not\in S$. 

To sum up, we have $S\subseteq  T_{\leq}\setminus \{ \emptyset\}$ such that 
\begin{displaymath}
	\forall t\in T \colon \ \Bigl( x([t]_{M})\not= y([t]_{M}) \ \Longrightarrow \, \bigl(\exists i \in \{ 0,\ldots,k\} \colon \ {t\!\!\upharpoonright_{\{ 1,\ldots,i\}} } \in S\bigr)\Bigr)
\end{displaymath} 
and for which \eqref{E:sumlo} holds. 
Now note that if $i \in \{ 0,\ldots,k-1\}$, then, by \eqref{E:sumlo}, for each $s\in T_\leq$ with $|s|=i$, we have 
\begin{displaymath}
	w_{i +1} \, |S\cap \{ sj\mid j< M_{i+1}\}| \, = \, \sum\nolimits_{sj\in S} w_{|sj|} \, < \, 1 \, = \, \delta_{i+1} w_{i+1} \sqrt{M_{i+1}} ,
\end{displaymath} 
which implies that 
\begin{displaymath}
	|S\cap \{ sj\mid j< M_{i+1}\}| \, < \, \delta_{i+1} \sqrt{M_{i+1}} .
\end{displaymath} 
Thus, $S$ witnesses that $(x \circ f) \sim (y \circ f)$. This proves~\eqref{E:lto}. Clearly, from~\eqref{E:lto} together with~\eqref{E:hu}, 
the inclusion~\eqref{E:hull} follows immediately.
	
Now we prove~\eqref{E:upp}. To this end, choose any family of independent random variables $(X_{t})_{t \in T}$ defined on a common domain 
$\Omega$ such that, for each $t\in T$, we have \begin{displaymath}
	\mathbb{P}[X_t=0] \, = \, \tfrac{1}{2} \, = \, \mathbb{P}[X_t=1]. 
\end{displaymath} We define a family of random variables $(Y_{s})_{s \in T_{\leq}}$ on the same domain $\Omega$ recursively as follows. 
For each $t \in T$, let $Y_{t} \defeq X_{t}$. Furthermore, if $i \in \{ 0,\ldots,k-1\}$ and $Y_s$ is defined for all $s\in T_\leq$ with $|s|\geq i+1$, 
then, for each $t\in T_\leq$ with $|t|=i$, we define 
\begin{displaymath}
	Y_{t}(\omega) \, \defeq \, \begin{cases}
			\, 0 & \text{if } |\{ j\in M_{i+1}\mid Y_{tj}(\omega) =1\}| < \frac{M_{i+1}}{2} + \delta_{i+1} \sqrt{M_{i+1}}, \\
			\, 1 & \text{otherwise.}
		\end{cases}
\end{displaymath} 
for all $\omega \in \Omega$. Define also, for $t\in T_\leq$, the set 
\begin{displaymath}
	\left. B_t \, \defeq \, \left\{ y\in 2^T \, \right| \hat{y}(t)=0 \right\} .
\end{displaymath} 
We leave it to the reader to verify by induction on $k-|t|$ that, for each $t\in T_\leq$, \begin{displaymath}
	\tfrac{|B_t|}{2^{|T|}} \, = \, \mathbb{P}[Y_t=0]. 
\end{displaymath} Since $B_\emptyset =B$, the equation above gives \begin{equation*}
	\tfrac{|B|}{2^{|T|}} \, = \, \mathbb{P}[Y_\emptyset =0]. 
\end{equation*} 
Therefore, to prove~\eqref{E:upp}, it remains to show that ${\mathbb P}[Y_\emptyset =0] \leq \tfrac{3}{4}$. In fact, we will prove that \begin{equation}\label{E:ept}
	\mathbb{P}[Y_\emptyset =0] - \tfrac{1}{2} \, \leq \, \epsilon_0, 
\end{equation} which will suffice by~\eqref{E:wse}. To this end, let us note that, for every $i \in \{ 0,\ldots,k \}$, there are real numbers $0< b_i\leq a_i$ with $a_i+b_i=1$ and such that, for all $t\in T_\leq$,  
\[
|t| \, = \, i \quad \Longrightarrow \quad \bigl(\, {\mathbb P}[Y_t=0] \, = \, a_i \; \hbox{ and }\; {\mathbb P}[Y_t=1] \, = \, b_i\,\bigr). 
\]
Evidently, $a_k=b_k=1/2$. Furthermore, for each $i \in \{ 0,\ldots,k \}$, $(Y_t \mid t \in T_{\leq}, \, |t|=i)$ is a family of independent random variables. Observe now that, by~\eqref{E:wse}, the sequence $(\epsilon_{i})_{i \in \mathbb{N}}$ is decreasing from $\epsilon_{0} < \tfrac{1}{4}$, so that in particular \begin{equation}\label{E:eof}
	\forall i \in \{ 0,\ldots,k\} \colon \quad \epsilon_i \, < \, \tfrac{1}{4}. 
\end{equation} Using~\eqref{E:XF}, \eqref{E:wse}, \eqref{E:del} and~\eqref{E:eof}, we see by induction on $k-i$ that \begin{equation}\label{E:ai}
	\forall i \in \{ 0,\ldots,k\} \colon \quad a_i-\tfrac{1}{2} \, < \, \epsilon_i. 
\end{equation} Now, \eqref{E:ai} and~\eqref{E:eof} together imply that \begin{displaymath}
	\forall i \in \{ 0,\ldots,k\} \colon \quad a_i \, < \, \tfrac{3}{4} ,
\end{displaymath} which gives~\eqref{E:ept} for $i=0$, as required. 
	
\smallskip
	
{\it Checking (ii), that is, the submeasure is elliptic (by (i)), but barely.} For every $i \in \mathbb{N}_{\geq 1}$, 
considering the partition of $Z$ into the sets $[s]_{M} \in \mathcal{A}$ with $s \in M_{1} \times \cdots \times M_{i}$, we conclude that \begin{displaymath}
	\tfrac{h_\phi(w_i)}{w_i} \, \geq \, \tfrac{1}{w_i^2 M_1\cdots M_i} . 
\end{displaymath} From~\eqref{E:wmt} and~\eqref{E:wze}, it follows that \begin{displaymath}
	\limsup\nolimits_{\xi\to 0} \tfrac{h_\phi(\xi)/\xi}{ \theta(\xi)} \, \geq \, \limsup\nolimits_{i\to\infty} \tfrac{1}{w_i^2 M_1\cdots M_i\, \theta(w_i)} \, = \, \infty, 
\end{displaymath} as required. 
	
\end{exmpl}

\section{Dynamical background}\label{section:dynamics}

The purpose of this section is to provide some background material necessary for the topological applications of our concentration results, which are given in the subsequent Section~\ref{section:applications}. These applications will concern topological dynamics, that is, the structure of topological groups reflected by their flows. To be more precise, if $G$ is a topological group, then a \emph{$G$-flow} is any non-empty compact Hausdorff space $X$ together with a continuous action of $G$ on $X$. The study of such objects is intimately linked with properties of certain function spaces naturally associated with the acting group. Some aspects of this correspondence, in particular concerning amenability, extreme amenability, and the connection with measure concentration, will be summarized below. For more details, we refer to~\cite{PestovBook,PachlBook}.

Now let $G$ be a topological group. Denote by $\mathcal{U}(G)$ the neighborhood filter of the neutral element in $G$ and endow $G$ with its \emph{right uniformity} defined by the basic entourages \begin{displaymath}
	\left\{ (x,y) \in G \times G \left\vert \, yx^{-1} \in U \right\}, \right.
\end{displaymath} 
where $U \in \mathcal{U}(G)$. 
In particular, a function $f \colon G \to \mathbb{R}$ is called \emph{right-uniformly continuous} if for every $\epsilon \in \mathbb{R}_{> 0}$ there exists $U \in \mathcal{U}(G)$ such that \begin{displaymath}
	\forall x,y \in G \colon \qquad yx^{-1} \in U \ \Longrightarrow \ \vert f(x) - f(y) \vert \, \leq \, \epsilon .
\end{displaymath} The set $\mathrm{RUCB}(G)$ of all right-uniformly continuous, bounded real-valued functions on~$G$, equipped with the pointwise operations and the supremum norm, constitutes a commutative unital real Banach algebra. A subset $H \subseteq \mathrm{RUCB}(G)$ is called \emph{UEB} (short for \emph{uniformly equicontinuous, bounded}) if $H$ is $\Vert \cdot \Vert_{\infty}$-bounded and \emph{right-uniformly equicontinuous}, that is, for every $\epsilon \in \mathbb{R}_{> 0}$ there is $U \in \mathcal{U}(G)$ such that \begin{displaymath}
	\forall f \in H \ \forall x,y \in G \colon \qquad yx^{-1} \in U \ \Longrightarrow \ \vert f(x) - f(y) \vert \, \leq \, \epsilon .
\end{displaymath}
The set $\mathrm{RUEB}(G)$ of all UEB subsets of $\mathrm{RUCB}(G)$ forms a convex vector bornology on $\mathrm{RUCB}(G)$. The \emph{UEB topology} on the dual Banach space $\mathrm{RUCB}(G)^{\ast}$ is defined as the topology of uniform convergence on the members of $\mathrm{RUEB}(G)$. This is a locally convex linear topology on the vector space $\mathrm{RUCB}(G)^{\ast}$ containing the weak-${}^{\ast}$ topology, that is, the initial topology generated by the maps $\mathrm{RUCB}(G)^{\ast} \to \mathbb{R}, \, \mu  \mapsto \mu (f)$ where $f \in \mathrm{RUCB}(G)$. More detailed information on the UEB topology is to be found in~\cite{PachlBook}. Furthermore, let us recall that the set 
\begin{displaymath}
	\mathrm{M}(G) \, \defeq \, \{ \mu \in \mathrm{RUCB}(G)^{\ast} \mid \mu \text{ positive}, \, \mu (\mathbf{1}) = 1 \} 
\end{displaymath} of all \emph{means} on $\mathrm{RUCB}(G)$ constitutes a compact Hausdorff space with respect to the weak-${}^{\ast}$ topology. The set $\mathrm{S}(G)$ of all (necessarily positive, linear) unital ring homomorphisms from $\mathrm{RUCB}(G)$ to $\mathbb{R}$ is a closed subspace of $\mathrm{M}(G)$, called the \emph{Samuel compactification} of $G$. For $g \in G$, let $\lambda_{g} \colon G \to G, \, x \mapsto gx$ and $\rho_{g} \colon G \to G, \, x \mapsto xg$. Note that $G$ admits an affine continuous action on $\mathrm{M}(G)$ given by \begin{displaymath}
	(g \mu)(f) \, \defeq \, \mu (f \circ \lambda_{g}) ,
\end{displaymath} 
where $g \in G, \, \mu \in \mathrm{M}(G), \, f \in \mathrm{RUCB}(G)$, 
and that $\mathrm{S}(G)$ constitutes a $G$-invariant subspace of $\mathrm{M}(G)$. Let us recall that $G$ is \emph{amenable} (resp., \emph{extremely amenable}) if $\mathrm{M}(G)$ (resp., $\mathrm{S}(G)$) admits a $G$-fixed point. It is well known that $G$ is amenable (resp., extremely amenable) if and only if every continuous action of $G$ on a non-void compact Hausdorff space admits a $G$-invariant regular Borel probability measure (resp., a $G$-fixed point). For a comprehensive account on (extreme) amenability of topological groups, the reader is referred to~\cite{PestovBook}. Below we recollect two specific results in that direction (Theorem~\ref{theorem:topological.day} and Theorem~\ref{theorem:whirly.groups}), relevant for Section~\ref{section:applications}.

First, regarding amenability of topological groups, we recall the following result from~\cite{SchneiderThom}, which will be used in the proof of Theorem~\ref{theorem:whirly.amenability}. Given a measurable space $\Omega$, let us denote by $\mathrm{Prob}(\Omega)$ the set of all probability measures on $\Omega$ and by $\mathrm{Prob}_{\mathrm{fin}}(\Omega)$ the convex envelope of the set of Dirac measures in $\mathrm{Prob}(\Omega)$.

\begin{thm}[\cite{SchneiderThom}, Theorem~3.2]\label{theorem:topological.day} A topological group $G$ is amenable if and only if, for every $\epsilon \in \mathbb{R}_{> 0}$, every $H \in \mathrm{RUEB}(G)$ and every finite subset $E \subseteq G$, there exists $\mu \in \mathrm{Prob}_{\mathrm{fin}}(G)$ such that, for 
$g \in E$ and $f \in H$, 
\begin{displaymath}
		\left\lvert \int f \, d\mu - \int f \circ \lambda_{g} \, d\mu \right\rvert \, \leq \, \epsilon .
\end{displaymath} \end{thm}

The result above suggests the following definition.

\begin{definition} Let $G$ be a topological group. A net $(\mu_{i})_{i \in I}$ of Borel probability measures on $G$ is said to \emph{UEB-converge to invariance (over $G$)} if, 
for all $g \in G$ and  $H \in \mathrm{RUEB}(G)$, 
\begin{displaymath}
		\sup\nolimits_{f \in H} \left\lvert \int f \, d\mu_{i} - \int f \circ \lambda_{g} \, d\mu_{i} \right\rvert \, \longrightarrow \, 0, \hbox{ as } i \longrightarrow I .
\end{displaymath} \end{definition}

For readers primarily interested in metrizable topological groups, we include the subsequent clarifying remark. Let us recall that, by well-known work of Birkhoff~\cite{birkhoff} and Kakutani~\cite{Kakutani36}, a topological group $G$ is first-countable if and only if $G$ is metrizable, in which case $G$ admits a metric $d$ both generating the topology of $G$ and being right-invariant, in the sense that $d(xg,yg) = d(x,y)$ for all $g,x,y \in G$.

\begin{remark} Let $G$ be a metrizable topological group and let $d$ be a right-invariant metric on $G$ generating the topology of $G$. Consider the set \begin{displaymath}
	\left. \mathrm{Lip}_{1}^{1} (G,d) \, \defeq \, \left\{ f \in [-1,1]^{G} \, \right\vert \forall x,y \in G \colon \, \vert f(x) - f(y) \vert \leq d(x,y) \right\} .
\end{displaymath} Then a net $(\mu_{i})_{i \in I}$ of Borel probability measures on $G$ UEB-converges to invariance over $G$ if and only if, for every $g \in G$, \begin{displaymath}
	\sup\nolimits_{f \in \mathrm{Lip}_{1}^{1} (G,d)} \left\lvert \int f \, d\mu_{i} - \int f \circ \lambda_{g} \, d\mu_{i} \right\rvert \, \longrightarrow \, 0, \hbox{ as } i \longrightarrow I .
\end{displaymath} A proof of this fact is to be found in~\cite[Corollary~3.6]{EquivariantConcentration}. \end{remark}

Second, let us recall that concentration of measure (Section~\ref{section:measure.concentration}) provides a very prominent method for proving extreme amenability of topological groups. This approach goes back to the seminal work of Gromov and Milman~\cite{GromovMilman} and has since been used in establishing extreme amenability for many concrete examples of Polish groups (see~\cite[Chapter~4]{PestovBook} for an overview). Below we mention a refined version of this method, as developed in~\cite{pestov10,PestovSchneider}. As usual, we define the \emph{support} of a Borel probability measure $\mu$ on a topological space $X$ to be \begin{displaymath}
	\spt \mu \, \defeq \, \{ x \in X \mid \forall U \subseteq X \text{ open: } \, x \in U \Longrightarrow \, \mu (U) > 0 \} ,
\end{displaymath} which is easily seen to constitute a closed subset of $X$. The following notion first appeared in~\cite{pestov10}, but originates in~\cite{GlasnerTsirelsonWeiss,GlasnerWeiss}.

\begin{definition} A topological group $G$ is called \emph{whirly amenable} if \begin{enumerate}
	\item[---$\,$] $G$ is amenable, and
	\item[---$\,$] any $G$-invariant regular Borel probability measure on a $G$-flow has support contained in the set of $G$-fixed points.
\end{enumerate} \end{definition}

Of course, whirly amenability implies extreme amenability. Note that the converse does not hold: the Polish group $\Aut (\mathbb{Q},{<})$, carrying the topology of pointwise~convergence, is extremely amenable~\cite{pestov98}, but not whirly amenable~\cite[Remark~1.3]{GlasnerTsirelsonWeiss}.

In order to establish whirly (hence extreme) amenability of topological groups of measurable maps the next section, we will combine the results of Section~\ref{subsection:levy.nets} with the strategy provided by the following theorem, which generalizes earlier results by Pestov~\cite[Theorem~5.7]{pestov10} and Glasner--Tsirelson--Weiss~\cite[Theorem~1.1]{GlasnerTsirelsonWeiss}.

\begin{thm}[\cite{PestovSchneider}, Theorem~3.9]\label{theorem:whirly.groups} Let $G$ be a topological group. If there exists a net $(\mu_{i})_{i \in I}$ of Borel probability measures on $G$ such that \begin{enumerate}
	\item[---$\,$] $(\mu_{i})_{i \in I}$ concentrates in $G$ (with respect to the right uniformity),
	\item[---$\,$] $(\mu_{i})_{i \in I}$ UEB-converges to invariance over $G$,
\end{enumerate} then $G$ is whirly amenable. \end{thm}

For a quantitative generalization of Theorem~\ref{theorem:whirly.groups} in the context of Gromov's observable diameters~\cite[Chapter~3$\tfrac{1}{2}$]{Gromov99}, the reader is referred to~\cite[Theorem~1.2]{EquivariantConcentration}.

\section{Topological groups of measurable maps}\label{section:applications}

This final section is devoted to applications of our results in topological dynamics. More precisely, we establish whirly amenability (thus, extreme amenability) of topological groups of measurable maps over parabolic or hyperbolic submeasures, with coefficients in any amenable topological group. Such groups, introduced for the Lebesgue measure by Hartman--Mycielski~\cite{HartmanMycielski} and later studied for pathological submeasures by Herer--Christensen~\cite{HererChristensen}, have more recently attracted growing attention in the context of extreme amenability~\cite{glasner98,pestov02,FarahSolecki,pestov10,sabok,PestovSchneider}, representation theory~\cite{solecki14}, and ample generics~\cite{KaichouhLeMaitre,KwiatkowskaMalicki}.

We choose an abstract approach to topological groups of measurable maps, following Fremlin~\cite[493A]{fremlin}. A more concrete description based on Stone's representation theorem for Boolean algebras~\cite{stone} will be given in Remark~\ref{remark:stone}. Let $\phi \colon \mathcal{A} \to \mathbb{R}$ be a submeasure on a Boolean algebra $\mathcal{A}$ and let $G$ be a topological group. By a \emph{finite $G$-partition of unity in $\mathcal{A}$} we mean a family $A = (A_{g})_{g \in G} \in \mathcal{A}^{G}$ such that \begin{enumerate}
	\item[---$\,$] $\{ g \in G \mid A_{g} \ne 0 \}$ is finite,
	\item[---$\,$] $\bigvee_{g \in G} A_{g} = 1$, and
	\item[---$\,$] $A_{g} \wedge A_{h} = 0$ for any two distinct $g,h \in G$.
\end{enumerate} Consider the topological group $L_{0}(\phi,G)$ consisting of all finite $G$-partitions of unity in $\mathcal{A}$, equipped with the multiplication defined by  \begin{displaymath}
	(A \cdot B)_{g} \, \defeq \, \bigvee\nolimits_{h \in G} A_{h} \wedge B_{h^{-1}g} \, = \, \bigvee\nolimits_{h \in G} A_{gh^{-1}} \wedge B_{h}
\end{displaymath} for $A,B \in L_{0}(\phi,G)$ and $g \in G$, and endowed with the \emph{topology of convergence in~$\phi$}. To be precise about the topology, let \begin{displaymath}
	N_{\phi}(A,U,\epsilon) \, \defeq \, \! \left\{ B \in L_{0}(\phi,G) \left\vert \, \phi \! \left(\bigvee \{ A_{g} \wedge B_{h} \mid g,h \in G, \, h \notin Ug \}\right) \! < \epsilon \right\} \right.
\end{displaymath} for any $A \in L_{0}(\phi,G)$, $U \in \mathcal{U}(G)$ and $\epsilon \in \mathbb{R}_{> 0}$. Then a subset $M \subseteq L_{0}(\phi,G)$ is \emph{open} if and only if \begin{displaymath}
	\forall A \in M \ \exists U \in \mathcal{U}(G) \ \exists \epsilon \in \mathbb{R}_{> 0} \colon \qquad N_{\phi}(A,U,\epsilon) \, \subseteq \, M .
\end{displaymath} 
In turn, a neighborhood basis at the neutral element $e_{L_{0}(\phi,G)} \in L_{0}(\phi,G)$, determined by $\left(e_{L_{0}(\phi,G)}\right)\!_{e_{G}} = 1$ and $\left(e_{L_{0}(\phi,G)}\right)\!_{g} = 0$ whenever $g \in G \setminus \{ e_{G}\}$, is given by the family of sets 
\begin{displaymath}
	N_{\phi}(U,\epsilon) \, \defeq \, N_{\phi}\!\left(e_{L_{0}(\phi,G)},U,\epsilon\right) \! \, = \, \! \left\{ A \in L_{0}(\phi,G) \left\vert \, \phi \! \left(\bigvee\nolimits_{g \in G\setminus U} A_{g} \right) \! < \epsilon \right\} \right. \! , 
\end{displaymath}
where $U \in \mathcal{U}(G)$ and $\epsilon \in \mathbb{R}_{> 0}$. For every~$\mathcal{B} \in \Pi(\mathcal{A})$, a straightforward computation reveals that the map \begin{displaymath}
	\gamma_{\mathcal{B}} \colon \, G^{\mathcal{B}} \, \longrightarrow \, L_{0}(\phi,G), \quad f \, \longmapsto \, {\left( \bigvee f^{-1}(g) \right)}_{g \in G}
\end{displaymath} is a continuous homomorphism.

Thanks to Stone's representation theorem~\cite{stone}, every Boolean algebra is isomorphic to a Boolean subalgebra of $\mathcal{P}(X)$ for some set $X$. In the subsequent remark, we recast the abstract construction above for such concrete algebras of sets.

\begin{remark}\label{remark:stone} Let $X$ be a set and let $\mathcal{A}$ be a Boolean subalgebra of $\mathcal{P}(X)$. Moreover, let $\phi \colon \mathcal{A} \to \mathbb{R}$ be a submeasure and let $G$ be a topological group. Consider the topological group \begin{displaymath}
	\left. \widetilde{L}_{0}(\phi,G) \, \defeq \, \left\{ f \in G^{X} \, \right\vert \exists \mathcal{B} \in \Pi(\mathcal{A}) \, \forall B \in \mathcal{B} \colon \, f \text{ is constant on } B \right\}
\end{displaymath} with the pointwise multiplication, that is, the subgroup structure inherited from~$G^{X}$, and the topology defined as follows: a subset $M \subseteq \widetilde{L}_{0}(\phi,G)$ is \emph{open} if and only if \begin{align*}
	\forall f \in M \ \exists U \in \mathcal{U}(G) \ &\exists \epsilon \in \mathbb{R}_{> 0} \colon \\
	& \left. \left\{ h \in \widetilde{L}_{0}(\phi,G) \, \right\vert \phi (\{ x \in X \mid h(x) \notin Uf(x) \}) < \epsilon \right\} \, \subseteq \, M .
\end{align*} Then the map \begin{displaymath}
	\xi \colon \, \widetilde{L}_{0}(\phi,G) \, \longrightarrow \, L_{0}(\phi,G) , \quad f \, \longmapsto \, {\left( \bigvee f^{-1}(g) \right)}_{g \in G}
\end{displaymath} is an isomorphism of topological groups. Furthermore, for every $\mathcal{B} \in \Pi (\mathcal{A})$, denoting by $\pi_{\mathcal{B}} \colon X \to \mathcal{B}$ the associated projection, we observe that \begin{displaymath}
	\gamma_{\mathcal{B}}(f) \, = \, \xi(f \circ \pi_{\mathcal{B}})
\end{displaymath} for all $f \in G^{\mathcal{B}}$. \end{remark}

In general---in fact, in most interesting cases---the topological groups resulting from the construction outlined above will not be Hausdorff, let alone Polish. However, starting from a standard probability space and a Polish group, one may equivalently study the topological dynamics of a corresponding Polish group described in the following remark.

\begin{remark} Let $(\Omega,\mu)$ be a standard probability space and let $G$ be a Polish~group. The topological group $\widehat{L}_{0}(\mu,G)$ consisting of all equivalence classes of $\mu$-measurable functions from $\Omega$ to $G$ up to equality $\mu$-almost everywhere, endowed with the pointwise multiplication (of representatives of equivalence classes) and the usual topology of convergence in measure with respect to $\mu$, is Polish~\cite[Proposition~7]{moore}. It is not difficult to see that the Hausdorff quotient of $\widetilde{L}_{0}(\mu,G)$, that is, the topological quotient group \begin{displaymath}
	\widetilde{L}_{0}(\mu,G) \big{/} \bigcap \mathcal{U}\!\left(\widetilde{L}_{0}(\mu,G) \right)
\end{displaymath} is isomorphic to a dense topological subgroup of $\widehat{L}_{0}(\mu,G)$. Consequently, from a dynamical perspective, there is no essential difference between the topological groups $L_{0}(\mu,G) \cong \widetilde{L}_{0}(\mu,G)$ and $\widehat{L}_{0}(\mu,G)$: their flows are in natural one-to-one correspondence. \end{remark}

We proceed to studying whirly amenability for groups of measurable maps, which will be the content of Theorem~\ref{theorem:whirly.amenability}. Preparing the proof of Theorem~\ref{theorem:whirly.amenability}, we need to establish some additional notation. To this end, let $G$ be a topological group. If $I$ is a set, $i \in I$ and $a \in G^{I \setminus \{ i \}}$, then we define $\eta_{i,a} \colon G \to G^{I}$ by \begin{displaymath}
	\eta_{i,a}(g)(j) \, \defeq \, \begin{cases}
		\, g & \text{if } j=i , \\
		\, a(j) & \text{otherwise} 
	\end{cases}
\end{displaymath} for all $g \in G$ and $j \in I$. Furthermore, if $\phi$ is a submeasure on a Boolean algebra $\mathcal{A}$, then, for any subset $H \subseteq \mathrm{RUCB}(L_{0}(\phi,G))$, we let \begin{displaymath}
	[H] \defeq \left\{ f \circ \gamma_{\mathcal{B}} \circ \eta_{B,a} \left\vert \, f \in H, \, \mathcal{B} \in \Pi(\mathcal{A}), \, B \in \mathcal{B}, \, a \in G^{\mathcal{B}\setminus \{ B \}} \right\} . \right.
\end{displaymath} The following two lemmata are straightforward adaptations of the corresponding results in~\cite{PestovSchneider}. We include the proofs for the sake of convenience.

\begin{lem}[cf.~\cite{PestovSchneider}, Lemma~4.3]\label{lemma:ueb} If $\phi$ is a submeasure on a Boolean algebra $\mathcal{A}$ and $G$ is a topological group, then,  
for each $H \in \mathrm{RUEB}(L_{0}(\phi, G))$,
\begin{displaymath}
	 [H] \in \mathrm{RUEB}(G) .
\end{displaymath} \end{lem}

\begin{proof} Consider any $H \in \mathrm{RUEB}(L_{0}(\phi,G))$. Of course, $[H]$ is norm-bounded as the set $H$ is. In order to prove that $[H]$ is right-uniformly equicontinuous, let $\epsilon \in \mathbb{R}_{> 0}$. Since $H \in \mathrm{RUEB}(L_{0}(\phi,G))$, there exists $U \in \mathcal{U}(L_{0}(\phi,G))$ such that $\vert f(x) - f(y) \vert \leq \epsilon$ for all $f \in H$ and $x,y \in L_{0}(\phi,G)$ with $xy^{-1} \in U$. According to the definition of the topology of $L_{0}(\phi,G)$, we find $V \in \mathcal{U}(G)$ and $\epsilon' \in \mathbb{R}_{> 0}$ such that $N_{\phi}(V,\epsilon') \subseteq U$. We are going to verify that $\vert f'(x) - f'(y) \vert \leq \epsilon$ for all $f' \in [H]$ and all $x,y \in G$ with $xy^{-1} \in V$. To this end, let $f \in H$, $\mathcal{B} \in \Pi(\mathcal{A})$, $B \in \mathcal{B}$ and $a \in G^{\mathcal{B}\setminus \{ B \}}$. Then, for any $x,y \in G$ with $xy^{-1} \in V$, we observe that \begin{align*}
	\gamma_{\mathcal{B}}(\eta_{B,a}(x))\gamma_{\mathcal{B}}(\eta_{B,a}(y))^{-1} \, &= \, \gamma_{\mathcal{B}}\!\left( \eta_{B,a}(x)\eta_{B,a}(y)^{-1} \right) \\
	& = \, \gamma_{\mathcal{B}}\!\left( \eta_{B,e_{G^{\mathcal{B}\setminus \{ B \}}}}\!\left(xy^{-1}\right) \! \right) \, \in \, \gamma_{\mathcal{B}}\!\left(V^{\mathcal{B}}\right) \, \subseteq \, N_{\phi}(V,\epsilon')
\end{align*} and therefore $\vert f(\gamma_{\mathcal{B}}(\eta_{B,a}(x))) - f(\gamma_{\mathcal{B}}(\eta_{B,a}(y))) \vert \leq \epsilon$. Hence, $[H] \in \mathrm{RUEB}(G)$. \end{proof}

\begin{lem}[cf.~\cite{PestovSchneider}, Lemma~4.4]\label{lemma:convergence.to.invariance} Let $\phi$ be a submeasure on a non-zero Boolean algebra $\mathcal{A}$ and let $G$ be a topological group. If $(\mathcal{B}_{i},\mu_{i})_{i \in I}$ is a net in $\Pi(\mathcal{A}) \times \mathrm{Prob}(G)$ such that \begin{itemize}[leftmargin=7mm]
	\item[---] $\, \forall \mathcal{B} \in \Pi(\mathcal{A}) \, \exists i_{0} \in I \, \forall i \in I \colon \ i_{0} \leq i \, \Longrightarrow \, \mathcal{B} \preceq \mathcal{B}_{i}$,
	\item[---] $\, \forall g \in G \, \forall H \in \mathrm{RUEB}(G) \colon \ \sup\nolimits_{f \in H} \left\lvert \int f \, d\mu_{i}  - \int f \circ \lambda_{g} \, d\mu_{i} \right\rvert \cdot \vert \mathcal{B}_{i} \vert \, \longrightarrow \, 0, \hbox{ as } i \to I$,
\end{itemize} then the net $\bigl( ( \gamma_{\mathcal{B}_{i}})_{\ast}\bigl(\mu_{i}^{\otimes \mathcal{B}_{i}}\bigr) \bigr)_{i \in I}$ UEB-converges to invariance over $L_{0}(\phi,G)$. \end{lem}

\begin{proof} For each $i \in I$, let us consider the corresponding push-forward Borel probability measure $\nu_{i} \defeq ( \gamma_{\mathcal{B}_{i}})_{\ast}\bigl(\mu_{i}^{\otimes \mathcal{B}_{i}}\bigr)$ on $L_{0}(\phi, G)$. We will show that $(\nu_{i})_{i \in I}$ UEB-converges to invariance over $L_{0}(\phi,G)$. For this, let $H \in \mathrm{RUEB}(L_{0}(\phi,G))$, $A = (A_{g})_{g \in G} \in L_{0}(\phi,G)$ and $\epsilon \in \mathbb{R}_{> 0}$. Note that \begin{displaymath}
	\mathcal{B} \, \defeq \, \{ A_{g} \mid g \in G \} \setminus \{ 0 \} \, \in \, \Pi (\mathcal{A})
\end{displaymath} and put $E \defeq \{ g \in G \mid A_{g} \ne 0 \} \cup \{ e \}$. According to Lemma~\ref{lemma:ueb} and our assumptions, there exists $i_{0} \in I$ such that, for every $i \in I$ with $i \geq i_{0}$, we have $\mathcal{B} \preceq \mathcal{B}_{i}$ and \begin{equation}\label{almost.invariance}
	\forall g \in E \colon \quad \sup\nolimits_{f \in [H]} \left\lvert \int f \, d\mu_{i} - \int f \circ \lambda_{g} \, d\mu_{i} \right\rvert \, \leq \, \tfrac{\epsilon}{\vert \mathcal{B}_{i} \vert} .
\end{equation} We claim that \begin{equation}\label{claim}
	\forall i \in I , \, i \geq i_{0} \colon \quad \sup\nolimits_{f \in H} \left\lvert \int f \, d\nu_{i} - \int f \circ \lambda_{A} \, d\nu_{i} \right\rvert \, \leq \, \epsilon .
\end{equation} To prove this, let $i \in I$ with $i \geq i_{0}$. Since $\mathcal{B} \preceq \mathcal{B}_{i}$, we find $s \in E^{\mathcal{B}_{i}}$ with $A = \gamma_{\mathcal{B}_{i}}(s)$. Let $n_{i} \defeq \vert \mathcal{B}_{i} \vert$ and pick an enumeration $\mathcal{B}_{i} = \{ B_{ij} \mid j < n_{i} \}$. For each $j < n_{i}$, let us define $a_{j} \in E^{\mathcal{B}_{i}}$ by \begin{align*}
	& a_{j}(B) \, \defeq \, \begin{cases}
		\, s_{\ell} & \text{if $B = B_{i\ell}\,$ for $\ell \in \{ 0,\ldots,j \}$}, \\
		\, e & \text{otherwise}
	\end{cases}
\end{align*} for each $B \in \mathcal{B}_{i}$, and let $b_{j} \defeq {a_{j}\!\!\upharpoonright_{\mathcal{B}_{i}\setminus \{ B_{ij} \}}} \in E^{\mathcal{B}_{i}\setminus \{ B_{ij} \}}$. Furthermore, let us define $a_{-1} \defeq e \in E^{\mathcal{B}_{i}}$. For all $j < n_{i}$ and $z \in G^{\mathcal{B}_{i}\setminus \{B_{ij}\}}$, note that $\lambda_{a_{j}} \circ \eta_{B_{ij},z} = \eta_{B_{ij},b_{j}z} \circ \lambda_{s_{j}}$ and $\lambda_{a_{j-1}} \circ \eta_{B_{ij},z} = \eta_{B_{ij},b_{j}z}$. Combining these observations with~\eqref{almost.invariance} and Fubini's theorem, we conclude that \begin{align*}
	\left\lvert \int f \! \right. & \left. \circ \, \lambda_{\gamma_{\mathcal{B}_{i}}(a_{j-1})}\, d\nu_{i} - \int f \circ \lambda_{\gamma_{\mathcal{B}_{i}}(a_{j})}\, d\nu_{i} \right\rvert \\
	&= \, \left\vert \int \left(f \circ \lambda_{\gamma_{\mathcal{B}_{i}}(a_{j-1})} \circ \gamma_{\mathcal{B}_{i}} \right) - \left(f \circ \lambda_{\gamma_{\mathcal{B}_{i}}(a_{j})} \circ \gamma_{\mathcal{B}_{i}} \right) \, d\mu_{i}^{\otimes \mathcal{B}_{i}} \right\vert \\
	&= \, \left\vert \int \left(f \circ \gamma_{\mathcal{B}_{i}} \circ \lambda_{a_{j-1}}\right) - \left(f \circ \gamma_{\mathcal{B}_{i}} \circ \lambda_{a_{j}}\right) \, d\mu_{i}^{\otimes \mathcal{B}_{i}} \right\vert \\
	&= \, \left\vert \int \left( \int f \circ \gamma_{\mathcal{B}_{i}} \circ \lambda_{a_{j-1}} \circ \eta_{B_{ij},z} \, d\mu_{i} \right. \right. \\
	& \hspace{50.5mm} \left. \left. - \int f \circ \gamma_{\mathcal{B}_{i}} \circ \lambda_{a_{j}}\circ \eta_{B_{ij},z} \, d\mu_{i} \right) d\mu_{i}^{\otimes \, \mathcal{B}_{i}\setminus \{ B_{ij} \} }(z) \right\vert \\
	&= \, \left\vert \int \left( \int f \circ \gamma_{\mathcal{B}_{i}} \circ \eta_{B_{ij},b_{j}z} \, d\mu_{i} - \int f \circ \gamma_{\mathcal{B}_{i}} \circ \eta_{B_{ij},b_{j}z} \circ \lambda_{s_{j}} \, d\mu_{i} \right) d\mu_{i}^{\otimes \, \mathcal{B}_{i}\setminus \{ B_{ij} \} }(z) \right\vert \\
	&\leq \, \int \left\vert \int f \circ \gamma_{\mathcal{B}_{i}} \circ \eta_{B_{ij},b_{j}z} \, d\mu_{i} - \int f \circ \gamma_{\mathcal{B}_{i}} \circ \eta_{B_{ij},b_{j}z} \circ \lambda_{s_{j}} \, d\mu_{i} \right\vert \, d\mu_{i}^{\otimes \, \mathcal{B}_{i}\setminus \{ B_{ij} \}}(z) \\
	&\leq \int \tfrac{\epsilon}{n_{i}} \, d\mu_{i}^{\otimes \, B_{i}\setminus \{ B_{ij} \} }(z) \, = \,  \tfrac{\epsilon}{n_{i}}
\end{align*} for all $j \in \{ 0,\ldots,n_{i}-1 \}$ and $f \in H$. For every $f \in H$, it follows that \begin{displaymath}
	\left\lvert \int f \, d\nu_{i} - \int f \circ \lambda_{A} \, d\nu_{i} \right\rvert \, \leq \, \sum_{j=0}^{n_{i}-1} \left\lvert \int f \circ \lambda_{\gamma_{\mathcal{B}_{i}}(a_{j-1})} \, d\nu_{i} - \int f \circ \lambda_{\gamma_{\mathcal{B}_{i}}(a_{j})} \, d\nu_{i} \right\rvert \, \leq \, \epsilon ,
\end{displaymath} which proves~\eqref{claim} and hence completes the argument. \end{proof}

We arrive at our fourth and final main result.

\begin{thm}\label{theorem:whirly.amenability} Let $\phi$ be a submeasure and let $G$ be a topological group. If $\phi$ has covering concentration and $G$ is amenable, then $L_{0}(\phi,G)$ is whirly amenable. \end{thm}

\begin{proof} Let $\phi$ be defined on the Boolean algebra $\mathcal{A}$. Since the desired conclusion is trivial if $\mathcal{A} = \{ 0 \}$, we may and will assume that $\mathcal{A} \ne \{ 0 \}$. According to Theorem~\ref{theorem:topological.day}, we find a net $(\mathcal{B}_{j},\mu_{j})_{j \in J}$ in $\Pi(\mathcal{A}) \times \mathrm{Prob}_{\mathrm{fin}}(G)$ such that \begin{itemize}[leftmargin=7mm]
	\item[---] $\, \forall \mathcal{B} \in \Pi(\mathcal{A}) \, \exists j_{0} \in J \, \forall j \in J \colon \ j_{0} \leq j \, \Longrightarrow \, \mathcal{B} \preceq \mathcal{B}_{j}$,
	\item[---] $\, \forall g \in G \, \forall H \in \mathrm{RUEB}(G) \colon \ \sup\nolimits_{f \in H} \left\lvert \int f \, d\mu_{j}\!  - \!\int f \circ \lambda_{g} \, d\mu_{j} \right\rvert \! \cdot \! \vert \mathcal{B}_{j} \vert \, \longrightarrow \, 0, \hbox{ as }j \to J$.
\end{itemize} Suppose that $\phi$ has covering concentration. By Remark~\ref{remark:covering.concentration}, we find $(\mathcal{C}_{\ell})_{\ell \in \mathbb{N}} \in \Pi(\mathcal{A})^{\mathbb{N}}$ such that, for every $\epsilon \in \mathbb{R}_{>0}$, \begin{equation}\label{concrete.concentration}
	\sup \{ \alpha_{\mathcal{X}(\mathbb{I},\lambda,\mathcal{B},\phi)}(\epsilon) \mid \mathcal{B} \in \Pi (\mathcal{A}), \, \mathcal{C}_{\ell} \preceq \mathcal{B} \} \, \longrightarrow \, 0, \hbox{ as }\ell \to \infty .
\end{equation} Consider the directed set $(I,\leq_{I})$ where $I \defeq \{ (\ell,j) \in \mathbb{N} \times J \mid \mathcal{C}_{\ell} \preceq \mathcal{B}_{j} \}$ and 
\begin{displaymath}
	(\ell_{0},j_{0}) \, \leq_{I} \, (\ell_{1},j_{1}) \quad :\Longleftrightarrow \quad \ell_{0} \leq \ell_{1}, \ j_{0} \leq_{J} j_{1}  .
\end{displaymath} For every $(\ell,j) \in I$, define $\mathcal{B}_{(\ell,j)} \defeq \mathcal{B}_{j}$ and $\mu_{(\ell,j)} \defeq \mu_{j}$. For each $i \in I$, let us consider \begin{displaymath}
	\nu_{i} \, \defeq \, ( \gamma_{\mathcal{B}_{i}})_{\ast}\!\left(\mu_{i}^{\otimes \mathcal{B}_{i}}\right) \, \in \, \mathrm{Prob}(L_{0}(\phi, G)) .
\end{displaymath} By Lemma~\ref{lemma:convergence.to.invariance}, the net $(\nu_{i})_{i \in I}$ UEB-converges to invariance over $L_{0}(\phi,G)$.
	
Thanks to Theorem~\ref{theorem:whirly.groups}, it remains to show that $(\nu_{i})_{i \in I}$ concentrates in $L_{0}(\phi,G)$. For each $i \in I$, we find a finite subset $S_{i} \subseteq G$ and a probability measure $\sigma_{i}$ on the discrete measurable space $S_{i}$ such that $\mu_{i}$ equals the push-forward measure of $\sigma_{i}$ along the map $S_{i} \to G, \, g \mapsto g$. According to~\eqref{concrete.concentration}, Remark~\ref{remark:submeasure.metric} and Remark~\ref{remark:concentration}(3), the net $(\mathcal{X}(S_{i},\sigma_{i},\mathcal{B}_{i},\phi))_{i \in I}$ constitutes a L\'evy net. Thus, by Remark~\ref{remark:concentration.in.uniform.spaces}, it suffices to verify that the family $(\gamma_{\mathcal{B}_{i}})_{i \in I}$ is uniformly equicontinuous. For this purpose, let $U \in \mathcal{U}(G)$ and $\epsilon \in \mathbb{R}_{> 0}$. For all $i \in I$ and $g,h \in G^{\mathcal{B}_{i}}$, we have \begin{align*}
	\phi \! &\left( \bigvee\nolimits_{x \in G\setminus U} \gamma_{\mathcal{B}_{i}}\!\left(hg^{-1}\right)_{x} \right) \, = \, \phi \! \left( \bigvee\nolimits_{x \in G\setminus U} \bigvee \!\left(hg^{-1}\right)^{-1}\!(x) \right)\\
	& \leq \, \phi \! \left( \bigvee\nolimits_{x \in G\setminus \{ e \}} \bigvee \!\left(hg^{-1}\right)^{-1}\!(x) \right) \! \, = \, \phi \! \left( \bigvee \{ B \in \mathcal{B}_{i} \mid g(B) \ne h(B) \} \right) \! \, = \, \delta_{\phi, \mathcal{B}_{i}}(g,h) ,
\end{align*} and therefore \begin{displaymath}
	\delta_{\phi, \mathcal{B}_{i}}(g,h) < \epsilon \quad \Longrightarrow \quad \gamma_{\mathcal{B}_{i}}(h)\gamma_{\mathcal{B}_{i}}(g)^{-1} = \gamma_{\mathcal{B}_{i}}\!\left(hg^{-1} \right) \in N_{\phi}(U,\epsilon) .
\end{displaymath} Hence, due to Remark~\ref{remark:concentration.in.uniform.spaces}, the net $(\nu_{i})_{i \in I}$ concentrates in $L_{0}(\phi,G)$, so that $L_{0}(\phi,G)$ is whirly amenable by Theorem~\ref{theorem:whirly.groups}. \end{proof}

\begin{cor}\label{corollary:whirly.amenability} Let $\phi$ be a parabolic or hyperbolic submeasure. If $G$ is an amenable topological group, then $L_{0}(\phi,G)$ is whirly amenable. \end{cor}

\begin{proof} This is an immediate consequence of Theorem~\ref{theorem:covering.concentration.for.non.elliptic.submeasures} and Theorem~\ref{theorem:whirly.amenability}. \end{proof}

We conclude with a partial converse of Corollary~\ref{corollary:whirly.amenability}.

\begin{prop}\label{proposition:reflecting.amenability} Let $G$ be a topological group. If $\phi$ is an elliptic or parabolic submeasure and $L_{0}(\phi,G)$ is amenable, then $G$ is amenable. \end{prop}

\begin{proof} We generalize an argument from~\cite[Theorem~1.1~(2)$\Longrightarrow$(1)]{PestovSchneider}. Let $\phi$ be defined on the Boolean algebra $\mathcal{A}$. Since $\phi$ is not pathological, we find a non-zero measure $\mu \colon \mathcal{A} \to \mathbb{R}$ such that $\mu \leq \phi$. Define $\Phi \colon \mathrm{RUCB}(G) \to \mathrm{RUCB}(L_{0}(\phi,G))$ by \begin{displaymath}
	\Phi (f)(A) \, \defeq \, \tfrac{1}{\mu (1)} \sum\nolimits_{g \in G} f(g)\mu (A_{g}) \, ,
\end{displaymath}
where $f \in \mathrm{RUCB}(G)$ and $A = (A_{g})_{g \in G} \in L_{0}(\phi,G)$. 
To check that $\Phi$ is well defined, let $f \in \mathrm{RUCB}(G)$. Since \begin{displaymath}
	\sup\nolimits_{A \in L_{0}(\phi,G)} \vert \Phi (f)(A) \vert \, = \, \Vert f \Vert_{\infty} ,
\end{displaymath} it follows that $\Phi (f) \in \ell^{\infty}(L_{0}(\phi,G))$. In order to show that $\Phi (f) \in \mathrm{RUCB}(L_{0}(\phi,G))$, let $\epsilon \in \mathbb{R}_{> 0}$. As $f \in \mathrm{RUCB}(G)$, there exists $U \in \mathcal{U}(G)$ such that \begin{displaymath}
	\forall g,h \in G \colon \qquad hg^{-1} \in U \ \Longrightarrow \ \vert f(g) - f(h) \vert \, \leq \, \tfrac{\epsilon}{2} \, .
\end{displaymath} Consider \begin{displaymath}
	\epsilon' \, \defeq \, \tfrac{\epsilon \mu (1)}{4\Vert f \Vert_{\infty} + 1} \, .
\end{displaymath} Then $V \defeq N_{\phi}(U,\epsilon')$ constitutes a neighborhood of the neutral element in $L_{0}(\phi,G)$. Let $A = (A_{g})_{g \in G}, \, B = (B_{h})_{h \in G} \in L_{0}(\phi,G)$ with $BA^{-1} \in V$. Then $\phi (C) < \epsilon '$ for \begin{displaymath}
	C \, \defeq \, \bigvee \{ A_{g} \wedge B_{h} \mid g,h \in G, \, h \notin Ug \} \, .
\end{displaymath} Since $\mu$ is a measure, we conclude that \begin{align*}
	\Phi (f)(A) - \Phi(f)(B) \, &= \, \tfrac{1}{\mu(1)} \sum\nolimits_{g,h \in G} (f(g) - f(h))\mu(A_{g} \wedge B_{h}) \\
		&= \, \tfrac{1}{\mu(1)} \sum\nolimits_{g,h \in G} (f(g) - f(h))\mu(A_{g} \wedge B_{h} \wedge C) \\
		& \qquad + \tfrac{1}{\mu(1)} \sum\nolimits_{g,h \in G} (f(g) - f(h))\mu(A_{g} \wedge B_{h} \wedge \neg C) \, ,
\end{align*} which, as $\mu \leq \phi$, readily implies that \begin{align*}
	\vert \Phi (f)&(A) - \Phi(f)(B) \vert \, \leq \, \tfrac{2\Vert f \Vert_{\infty} \epsilon'}{\mu (1)} + \tfrac{\epsilon}{2} \, \leq \, \epsilon \, .
\end{align*} This shows that $\Phi (f) \in \mathrm{RUCB}(L_{0}(\phi,G))$. Therefore, $\Phi$ is well-defined. It is straightforward to check that $\Phi$ is linear, positive, and unital. Furthermore, if $f \in \mathrm{RUCB}(G)$ and $g \in G$, then \begin{align*}
	\Phi (f \circ \lambda_{g}) (A) \, &= \, \tfrac{1}{\mu (1)} \sum\nolimits_{h \in G} f(gh)\mu (A_{h}) \, = \, \tfrac{1}{\mu (1)} \sum\nolimits_{h \in G} f(h)\mu \! \left(A_{g^{-1}h}\right) \\
		&= \, \Phi (f) \! \left(\left(A_{g^{-1}h}\right)_{h \in G}\right) \, = \, \Phi (f) \! \left( \gamma_{\{ 1 \}}(g) A \right) \, = \, \left(\Phi (f) \circ \lambda_{\gamma_{\{ 1 \}}(g)}\right)\!(A)
\end{align*} for every $A = (A_{h})_{h \in G} \in L_{0}(\phi,G)$, that is, $\Phi (f \circ \lambda_{g}) = \Phi (f) \circ \lambda_{\gamma_{\{ 1 \}}(g)}$. Assuming that $L_{0}(\phi,G)$ is amenable and considering a left-invariant mean $\mathbf{m} \colon \mathrm{RUCB}(L_{0}(\phi, G)) \to \mathbb{R}$, we deduce from the properties of $\Phi$ that $\mathbf{m} \circ \Phi \colon \mathrm{RUCB}(G) \to \mathbb{R}$ is a left-invariant mean, whence $G$ is amenable. \end{proof}

The subsequent corollary generalizes the main result of~\cite{PestovSchneider} from non-zero diffuse measures to arbitrary parabolic submeasures.

\begin{cor}\label{corollary:reflecting.amenability} Let $\phi$ be a parabolic submeasure and let $G$ be a topological group. Then the following are equivalent. \begin{enumerate}
	\item[---$\,$] $G$ is amenable.
	\item[---$\,$] $L_{0}(\phi,G)$ is amenable.
	\item[---$\,$] $L_{0}(\phi,G)$ is whirly amenable.
\end{enumerate} \end{cor}

%%%%%%%%%%%%%%%%%%%%%%%%%%%%%%%%%%

\smallskip

\noindent {\bf Acknowledgment}. We thank Paul Larson for several remarks that helped improve the presentation of our arguments.

\end{document}